\documentclass[a4paper,11pt,oneside,final]{amsart}
\usepackage{fullpage}
\usepackage[english]{babel}
\usepackage[T1]{fontenc}
\usepackage{mathtools}
\usepackage{enumerate}

\usepackage{etex}
\usepackage[utf8x]{inputenc}
\usepackage{amsmath,amssymb,amsfonts,amsthm}
\usepackage{mathrsfs}
\usepackage{hyperref}
\usepackage{tikz-cd}
\usepackage{ctable}

\theoremstyle{plain}
\newtheorem{thm}{Theorem}[section]
\newtheorem{cor}[thm]{Corollary}
\newtheorem{lem}[thm]{Lemma}
\newtheorem{prop}[thm]{Proposition}

\theoremstyle{definition}
\newtheorem{rem}[thm]{Remark}

\newtheorem*{condition*}{Condition}

\newtheorem{defin}[thm]{Definition}
\newtheorem{example}[thm]{Example}
\newtheorem*{example*}{Example}
\numberwithin{equation}{section}

\DeclareMathOperator{\bfe}{\mathbf{e}}

\DeclareMathOperator{\bfk}{\mathbf{k}}

\DeclareMathOperator{\bfa}{\mathbf{a}}
\DeclareMathOperator{\bfb}{\mathbf{b}}
\DeclareMathOperator{\bfu}{\mathbf{u}}

\DeclareMathOperator{\bfX}{\mathbf{X}}

\DeclareMathOperator{\bfx}{\mathbf{x}}

\DeclareMathOperator{\bft}{\mathbf{t}}

\DeclareMathOperator{\ZZ}{\mathbb{Z}}

\DeclareMathOperator{\sing}{sing}

\DeclareMathOperator{\codim}{codim}

\DeclareMathOperator{\GL}{GL}
\DeclareMathOperator{\Tr}{Tr}

\newcommand{\R}{\mathbb{R}}
\newcommand{\Q}{\mathbb{Q}}
\newcommand{\Z}{\mathbb{Z}}
\newcommand{\F}{\mathbb{F}}

\newcommand{\C}{\mathbb{C}}
\newcommand{\A}{\mathbb{A}}
\renewcommand{\P}{\mathbb{P}}
\renewcommand{\GL}{\text{GL}}

\newcommand{\calE}{\mathcal{E}}
\newcommand{\bSpec}{\textbf{Spec}}
\newcommand{\calM}{\mathcal{M}}
\newcommand{\calP}{\mathcal{P}}

\newcommand{\bx}{\mathbf{x}}
\newcommand{\by}{\mathbf{y}}
\newcommand{\bz}{\mathbf{z}}

\newcommand{\bt}{\mathbf{t}}

\newcommand{\ba}{\mathbf{a}}
\newcommand{\bu}{\mathbf{u}}
\renewcommand{\bx}{\mathbf{x}}
\newcommand{\bX}{\mathbf{X}}
\newcommand{\bZ}{\mathbf{Z}}

\newcommand{\ep}{\varepsilon}
\newcommand{\al}{\alpha}

\newcommand{\ga}{\gamma}
\newcommand{\om}{\omega}
\newcommand{\sig}{\sigma}
\newcommand{\beq}{\begin{equation}}
\newcommand{\eeq}{\end{equation}}
\newcommand{\ndiv}{\nmid}
\newcommand{\con}{\equiv}
\newcommand{\modd}[1]{\; ( \text{mod} \; #1)}
\newcommand{\Sig}{\Sigma}

\newcommand{\glo}{\mathrm{glo}}

\newcommand{\loc}{\mathrm{loc}}

\definecolor{pink}{rgb}{1,.2,.6}
\definecolor{orange}{rgb}{0.7,0.3,0}
\definecolor{blue}{rgb}{.2,.6,.75}
\definecolor{green}{rgb}{.4,.7,.4}
\definecolor{purple}{RGB}{127,0,255}
 \definecolor{gray}{RGB}{211,211,211}

\begin{document}
\title{Counting integral points in thin sets of type II: singularities,  sieves, and stratification}

\author[Bonolis]{Dante Bonolis}
\address{Duke University, 120 Science Drive, Durham NC 27708}
\email{dante.bonolis@duke.edu}

\author[Pierce]{Lillian B. Pierce}
\address{Duke University, 120 Science Drive, Durham NC 27708}
\email{lillian.pierce@duke.edu}

\author[Woo]{Katharine Woo}
\address{Princeton University, Fine Hall, Washington Road, Princeton NJ 08544}
\email{khwoo@princeton.edu }

\begin{abstract}
  Consider an absolutely irreducible polynomial $F(Y,X_1,\ldots,X_n) \in \Z[Y,X_1,\ldots,X_n]$ that is monic in $Y$  and is a polynomial in $Y^m$ for an integer $m \geq 1$. Let $N(F,B)$ count the number of $\bx \in [-B,B]^n \cap \Z^n$ such that $F(y,\bx)=0$ is solvable for $y \in\Z$. In nomenclature of Serre, bounding $N(F,B)$ corresponds to counting integral points in an affine thin set of type II. Previously, in this generality Serre proved $N(F,B) \ll_F B^{n-1/2}(\log B)^{\ga}$ for some $\ga<1$.  When $m \geq 2$, this new work proves  $N(F,B) \ll_{n,F,\ep} B^{n-1+1/(n+1) + \ep}$    under  a nondegeneracy condition that encapsulates that $F(Y,\bX)$  is truly a polynomial in $n+1$ variables, even after performing any $\GL_n(\Q)$ change of variables on $X_1,\ldots,X_n$. Under GRH, this result also holds when $m=1$.  We show that generic polynomials  satisfy the relevant nondegeneracy condition.  Moreover, for a certain  class of polynomials, we prove the stronger bound $N(F,B) \ll_{F} B^{n-1}(\log B)^{e(n)}$, comparable to a conjecture of Serre. A key strength of these results is that they require no nonsingularity property of $F(Y,\bX)$.  The Katz-Laumon stratification for character sums, in a   new uniform formulation appearing in a companion paper of Bonolis, Kowalski and Woo, is a key ingredient in the sieve method we develop to prove upper bounds that explicitly control any dependence on the size of the coefficients of $F$. 
\end{abstract}

\maketitle

\section{Introduction}\label{sec_introduction}

Let $F(Y,X_1,...,X_n)\in \Z[Y,X_1,...,X_n]$ be a polynomial, and define the counting function $$N(F,B) = \#\{\bx \in \Z^n \cap [-B,B]^n: \exists y\in \Z, F(y,\bx)=0\}.$$
 Trivially, $N(F,B) \ll B^n$, and this is sharp in the case that $F(Y,\bX)$ has degree at most 1 in $Y$. 
 If $\deg_Y F \geq 2$ and $F$ is absolutely irreducible, the best known bound (in such generality) is $N(F,B)\ll_{F} B^{n-1/2}(\log B)^{\ga}$ for some $\ga<1$ due to Serre \cite{Ser97}.
 A conjecture of Serre  on thin sets of type II suggests that for suitable polynomials $F(Y,\bX),$ the trivial bound for $N(F,B)$ may be reduced by a full factor of $B$, and this is the motivation for the present work.
As an illustration, here is a special consequence of our main work:

\begin{thm}\label{thm_cyclic_uniform}
 Fix $n \geq 2$. Let $H(\bX) \in \Z[X_1,\ldots,X_n]$ be an  irreducible polynomial and define  $F(Y,\bX) = Y^d - H(\bX)$ for an integer $d \geq 2$.
Suppose there is no $L\in \GL_n(\Q)$ such that $H(L(\bX)) \in \Z[(X_i)_{i \in I}]$ for a subset $I \subsetneq\{1,\ldots,n\}$. 
     Then  there is an integer $e(n)\geq 1$ such that for all $B \geq 1$,
     \[ N(F,B) \ll_{n,\deg F,\ep} \log(\|F\|+2)^{e(n)} B^{n-1+\frac{1}{n+1}+\ep},  \] 
     for all $\ep>0$.
 \end{thm}
 (Here and throughout,  $\|F\|$  denotes the maximum absolute value of any coefficient of $F$. We follow the convention that $A \ll_\kappa B$ implies that $|A| \leq C_\kappa B$ with the implied constant $C_\kappa$ depending on $\kappa$.)
 A strength of Theorem \ref{thm_cyclic_uniform} is that it requires \emph{no assumption of nonsingularity} for $F$. Many previous works on bounding $N(F,B)$ in various contexts, such as \cite{Pie06,HB08,Mun09,HBPie12,Bon21,BCLP23,BonPie24} have   relied on $F$ satisfying an appropriate nonsingularity condition. A central aim of the present work is to bound $N(F,B)$ for singular polynomials $F$.  
 
Our main work goes beyond the cyclic structure of the polynomial in Theorem \ref{thm_cyclic_uniform}. 
Consider any absolutely irreducible polynomial of the form
 \beq\label{F_dfn_intro}
F(Y,X_1,...,X_n) = Y^{md} + Y^{m(d-1)} f_1(\bX) +\cdots + Y^mf_{d-1}(\bX)+ f_d(\bX),
\eeq
in which each $f_j(\bX) \in \Z[X_1,\ldots,X_n]$ is a polynomial (not necessarily homogeneous) of degree $k_j \geq 0$, and $f_d \not\con 0$.
If $m=1$, (\ref{F_dfn_intro}) simply specifies that $F$ is monic   in $Y$; if $m \geq 2$, this structure imposes that $F$ is a polynomial in $Y^m$. 
 It is natural to require that $f_d \not\con 0$ (since otherwise $y=0$ provides  a solution to $F(y,\bx)=0$ for each $\bx \in [-B,B]^n$ and so $B^n \ll N(F,B) \ll B^n$); this certainly holds if $F(Y,\bX)$ is absolutely irreducible. 
 
Our most general theorem (Theorem \ref{thm_main}) bounds $N(F,B)$ with an exponent as strong as Theorem \ref{thm_cyclic_uniform}   under two conditions: (i) 
a nondegeneracy condition that encapsulates that $F(Y,\bX)$  is truly a polynomial in $n+1$ variables, even after performing any $\GL_n(\Q)$ change of variables on $X_1,\ldots,X_n$;  (ii) the condition $m \geq 2$ in (\ref{F_dfn_intro}). We furthermore can allow $m=1$, under the assumption of GRH. Under a stronger nondegeneracy condition, we show the dependence on $\|F\|$ is at most polylogarithmic (Theorem \ref{thm_explicit}).
In particular, our main theorems  require  no assumption of nonsingularity.
Moreover, we show that in the moduli space of polynomials of the form (\ref{F_dfn_intro}) for fixed $m, d, k_1,\ldots,k_j$ (and even among the subspace of singular polynomials), polynomials that satisfy the relevant nondegeneracy condition (i)   are generic (Theorem \ref{thm_genericity}). 
We furthermore prove  for a certain   class of polynomials the stronger bound $N(F,B) \ll_{F} B^{n-1}(\log B)^{e(n)}$, which is comparable to a conjecture of Serre (Theorems \ref{thm_allowable_not_strongly} and \ref{thm_genuine_not_strongly}).  We next state these general results precisely.

\subsection{General setting of new work}

The first nondegeneracy condition we require  takes the following form.
Let $M$ be a finite (nontrivial) extension of $\Q(X_1,...,X_n)=\Q(\bX)$. We say   $M$ is an \textbf{n-genuine extension} of $\Q(\bX)$ if for every $G(Y,\bX)\in \Z[Y,\bX]$ such that $$M = \Q(\bX)[Y]/(G(Y,\bX)),$$ 
$G(Y,\bX)$ has nonzero degree in each of $Y,X_1,\ldots,X_n$. 
We say an absolutely irreducible polynomial $F(Y,\bX)\in \Z[Y,\bX]$ that is monic in $Y$ is \textbf{$(1,n)$-allowable} if for every linear change of variables $\sigma\in \GL_n(\Q)$, the polynomial  $F_\sigma(Y,X_1,...,X_n) := F(Y,\sigma(\bX))$ has the property that $\Q(\bX)[Y]/(F_\sig(Y,\bX))$ is an $n$-genuine extension. (In particular, $\deg_Y F \geq 2$.)
Our first main result is:
\begin{thm}\label{thm_main}
Fix $n \geq 2$. Let $F \in \Z[Y,X_1,\ldots,X_n]$ with $\deg_Y F \geq 2$ be a $(1,n)$-allowable polynomial of the  form (\ref{F_dfn_intro})  for an integer $m \geq 2$. Then there exists an integer $h=h(n,\deg F) \geq 1$ such that for all $B \geq 1$,
\[ N(F,B) \ll_{n,\deg F,\ep} \|F\|^{h} B^{n-1+\frac{1}{n+1} + \ep},\]
for every $\ep>0$. If $\deg_Y F \geq 2$ and the same hypotheses hold with  $m =1$, the same result holds, conditional on GRH.
\end{thm}

The proof of this theorem relies on a dichotomy, so that it is a consequence of two main theorems. On one side of the dichotomy leading to Theorem \ref{thm_main}, we prove a slightly better bound, with a polylogarithmic dependency on $\|F\|$, under a stronger nondegeneracy condition.  The main strategy to prove this result is a polynomial sieve.

  Precisely, we say that a finite (nontrivial) extension $M$ of $\Q(X_1,...,X_n)=\Q(\bX)$ is a \textbf{strongly $n$-genuine extension} of $\Q(\bX)$ if for all subextensions $M'$ satisfying  $$\Q(\bX) \subsetneq M' \subset M,$$
    $M'$ is an $n$-genuine extension. 
We say that an absolutely irreducible polynomial $F(Y,\bX) \in \Z[Y,\bX]$ that is monic in $Y$ is  \textbf{strongly $(1,n)$-allowable} if  for every linear change of variables $\sigma\in \GL_n(\Q)$, the polynomial   $F_\sigma(Y,X_1,...,X_n) := F(Y,\sigma(\bX))$ has the property that $\Q(\bX)[Y]/(F_\sig(Y,\bX))$ is a strongly $n$-genuine extension. (In particular, $\deg_Y F \geq 2$.)

\begin{thm}\label{thm_explicit}
Fix $n \geq 2$. Let $F \in \Z[Y,X_1,\ldots,X_n]$ with $\deg_Y F \geq 2$ be a strongly $(1,n)$-allowable polynomial of the form (\ref{F_dfn_intro}) for an integer $m \geq 2$. Then there exists an integer $e(n) \geq 1$ such that for all $B \geq 1$,
\[ N(F,B) \ll_{n,\deg F,\ep} \log(\|F\|+2)^{e(n)} B^{n-1+\frac{1}{n+1} + \ep},\]
for every $\ep>0$.
If $\deg_Y F \geq 2$ and the same hypotheses hold with $m =1$, the same result holds, conditional on GRH.
\end{thm}

On the other side of the dichotomy leading to Theorem \ref{thm_main}, we must consider the case in which  $F$ is $(1,n)$-allowable but not strongly $(1,n)$-allowable. In this case, a completely different argument  achieves a stronger exponent,  at the expense of the dependency on the coefficients of $F$:
 \begin{thm}\label{thm_allowable_not_strongly}
Fix $n \geq 2$. Let $F(Y,\bX) \in \Z[Y,X_1,\ldots,X_n]$ with $\deg_Y F\geq 2$ be $(1,n)$-allowable but not strongly $(1,n)$-allowable. Then there exists an integer $e(n) \geq 1$ such that for all $B \geq 1$,  
\[ N(F,B) \ll_{n,F} B^{n-1} (\log (B+2))^{e(n)}.\] 
\end{thm}
The exponent in this case concurs with a conjecture of Serre, which we review momentarily. In this theorem, the dependence on $F$ in the implicit constant is inexplicit. But we show that within the dichotomy leading to Theorem \ref{thm_main}, either we can apply the polynomial sieve (leading to the inequality stated for $N(F,B)$ in Theorem \ref{thm_explicit}) or we can refine the strategy for Theorem \ref{thm_allowable_not_strongly} to attain $N(F,B) \ll_{n,\deg F} \|F\|^{h}B^{n-1}(\log (B+2))^e$ for some $e=e(n)$ and $h=h(n,\deg F)$. Taken together, this ultimately verifies the remaining cases of Theorem \ref{thm_main}. Note that this result does not require the polynomial to have the form (\ref{F_dfn_intro}) for an integer $m \geq 2$, and hence sidesteps the need to assume GRH  when $m=1$.
 
\subsection{Relation to previous  results}
 The special case of Theorem \ref{thm_cyclic_uniform} in which $H(\bX)$ is a nonsingular form includes the main theorem of Munshi's work \cite{Mun09} as formulated and proved by Bonolis \cite[Remark 1]{Bon21}.  The special case of Theorem \ref{thm_main} in which $F(Y, \bX) = g(Y) - f(\bX)$ for a fixed polynomial $g$ and a nonsingular form $f$ recovers many cases of the main theorem of \cite{Bon21}  (see Example \ref{ex: sepvars} and details in Remark \ref{remark_sepvars}).
 Theorem \ref{thm_explicit} immediately implies Theorem \ref{thm_cyclic_uniform}, since under the hypotheses of Theorem \ref{thm_cyclic_uniform}, $F(Y,\bX)=Y^d  - H(\bX)$ is a strongly $(1,n)$-allowable polynomial 
 (as proved in Example \ref{ex: generic cyclic}). 
 Finally, the special case of Theorem \ref{thm_explicit} in which $F(Y,\bX)=0$ defines a nonsingular hypersurface in weighted projective space $\mathbb{P}(e,1,\ldots,1)$ (for an integer $e \geq 1$) recovers the main theorem of Bonolis and Pierce's work \cite{BonPie24}   (see Example \ref{ex: generic nonsingular weighted} and Remark \ref{remark_correction}). 

\subsection{Relation to Serre's conjecture on thin sets}
Recall Serre's formulation of  thin sets  \cite[\S 9.1]{Ser97}: a subset $M \subset \mathbb{P}^{n-1}(\Q)$  is a   thin set of type I if  there is a Zariski-closed subvariety $F \subsetneq \mathbb{P}^{n-1}$ such that $M \subset F(\Q)$.  On the other hand, a subset $M \subset \mathbb{P}^{n-1}(\Q)$  is a   thin set of type II in $\mathbb{P}^{n-1}$ if there is an irreducible projective algebraic variety $Z$ of dimension $n-1$ and a generically surjective morphism $\pi: Z \rightarrow \P^{n-1}$ of degree $D \geq 2$ with $M \subset \pi (Z(\Q))$.   Any thin set is contained in a finite union of thin sets of type I and type II. Affine thin sets of each type are defined analogously, with $\mathbb{P}^{n-1}$ replaced by $\mathbb{A}^n$.  One motivation to show   a thin set contains few points comes from specialization of Galois groups: as Serre observed, for $\bx$ outside of a thin set, $F(Y,\bx)$ has the same Galois group as $F(Y,\bX)$ \cite[p. 122]{Ser97}.

\subsubsection{Projective setting}
 Let $M \subset \mathbb{P}^{n-1}(\Q)$ be a thin set. Define the height function $H(\bx)=\max_{1 \leq i \leq n}|x_i|$ for   $\bx = [x_1: \cdots : x_n] \in \mathbb{P}^{n-1}(\Q)$ with $(x_1,\ldots,x_n) \in \Z^n$ and $\gcd(x_1,\ldots, x_n)=1$, and the associated counting function for the thin set:
 \[M_H(B) =\{ \bx \in M: H(\bx) \leq B\}.\]
 Trivially, $M_H(B) \ll B^n$. Serre conjectured in \cite[p. 178]{Ser97} that 
 \beq\label{Serre_M_conjecture} 
 M_H(B) \ll_M B^{n-1} (\log B)^c
 \eeq
 for some $c$. If $M$ is of type I this can be interpreted as a statement of the Dimension Growth Conjecture, which predicts that if $Z$ is an irreducible projective variety in $\mathbb{P}^{n-1}(\Q)$ of degree $D \geq 2$ then $Z_H(B) \ll_{Z,\ep} B^{\dim Z+\ep}$ for all $B \geq 1$ and $\ep>0$. This has been resolved in many cases by \cite{Bro03a,BroHB06a,BroHB06b,HB94,HB02,Sal07}, and in all remaining cases by Salberger \cite{Sal23}. A uniform version, in which the implied constant depends only on $n, \deg Z, \ep$ is also known, as long as $\deg Z \neq 3$; see \cite{Sal23}.
 
For thin sets of type II in $\mathbb{P}^{n-1}(\Q)$, the conjecture (\ref{Serre_M_conjecture}) has been proved (with $B^\ep$ instead of a logarithmic factor)  when $n=2$, and when $n=3$ if the cover has degree at least 3,  by Broberg \cite{Bro03N}. (For $n=3$ and degree 2, Broberg proves $M_H(B)\ll B^{9/4+\ep}$.) It has also been proved for the special case of nonsingular cyclic covers of $\P^{n-1}$ when $n \geq 10$, by Heath-Brown and Pierce in \cite{HBPie12} (which also obtains somewhat weaker nontrivial bounds for $3 \leq n \leq 9$).

\subsubsection{Affine setting}
As Serre describes in \cite[p. 122]{Ser97}, an affine thin set $M$ of type II in $\mathbb{A}^n(\Q)$ has a corresponding  polynomial $F(Y,X_1,\ldots,X_n) \in \Q(X_1,\ldots,X_n)[Y]$ that is absolutely irreducible (that is, irreducible in $\overline{\Q}(X_1,\ldots,X_n)[Y]$) and of degree $D \geq 2$ in $Y$, for which
\[ M = \{\bx \in \Q^n: \exists y \in \Q, F(y,\bx) =0, \text{$y$ not a pole of any coefficient of $F$}\}.\] 
  If $F$ satisfies appropriate homogeneity conditions, a thin set of type II in $\mathbb{P}^{n-1}(\Q)$ can be described in an analogous way. 

Our study of $N(F,B)$ is thus the study of counting integral points on affine thin sets of type II. 
 Serre made initial progress: if $F \in \Z[Y,X_1,\ldots,X_n]$  is absolutely irreducible   and $\deg_Y F \geq 2$, then by \cite[Thm. 1, p. 177]{Ser97},
 \beq\label{Serre_bound_intro}
 N(F,B) \ll_{F} B^{n-1/2}(\log B)^\ga
 \eeq
 for some $\ga<1$, and unspecified dependence of the implicit constant on $F$. An examination of Serre's original argument (which is done in \cite{BPW25x}) shows that this bound (with $\gamma=1$) can be made to have at most polylog dependence on $\|F\|$  if $F(Y,\bX)$ is of the form (\ref{F_dfn_intro}) for some $m \geq 2$; when $m=1$, this polylog dependence can be established under GRH for reasons similar to why GRH is assumed in Theorem \ref{thm_main} for $m=1$. To prove our main theorems, we require an unconditional  version of (\ref{Serre_bound_intro}) with at most polylog dependence on $\|F\|$, which applies to any irreducible polynomial $F(Y,\bX)$ with constant-leading-coefficient in $Y$. We prove this in Theorem \ref{thm_polylog}.

Serre also attributes  to  earlier work of Cohen \cite[Thm. 2.1]{Coh81} a bound of the strength 
 \[N(F,B) \ll_{n,D} \|F\|^{c} B^{n-1/2}\log B 
 \]
 for all $B \geq \|F\|^{3c}$ for some $c=c(n,\deg F)$.
  There is a gap in the method described in \cite[Thm. 2.1]{Coh81}; we prove in the companion paper \cite{BPW25x_per} that this gap can be repaired by using the theory of strongly $n$-genuine polynomials that we develop in the present paper.

 If no further condition is placed on $F$, the upper bound (\ref{Serre_bound_intro}) is (essentially) sharp. For example, Serre points to the trivial example $F(Y,X_1,\ldots,X_n)=Y^2-X_1$, for which 
 \beq\label{sharp_example}
 N(F,B) = \# \{\bx \in [-B,B]^n \cap \Z^n: \text{$x_1$ is a square}\} \gg B^{n-1/2}.
 \eeq
 But Serre's conjecture (\ref{Serre_M_conjecture}) for counting rational points on projective thin  sets also suggests a question for counting integral points on certain affine thin sets of type II: for
 \emph{suitable} polynomials $F(Y,X_1,...,X_n)$ with $\deg_Y F\geq 2$, is there is some $c=c(F)$ such that 
\beq\label{intro_Serre_conjecture}
N(F,B) \ll_{F} B^{n-1}( \log B)^c, \qquad \text{for all $B \geq 1$?}
\eeq
 And, can the dependence of the implicit constant on $\|F\|$ be controlled? These questions motivate our present work.

Previous works such as \cite{Mun09,Bon21,BonPie24} (or in the function field setting \cite{BCLP23})  improved on Serre's bound (\ref{Serre_bound_intro}), and approached 
 (\ref{intro_Serre_conjecture}), but only when $F(Y,\bX)$ satisfies a nonsingularity condition, and has a special shape. Additionally, Pierce and Heath-Brown \cite{HBPie12} have proved Serre's conjecture (\ref{intro_Serre_conjecture}) for $F(Y,\bX) = Y^d - f(\bX)$ when $n \geq 10$ (also obtaining stronger results than (\ref{Serre_bound_intro}) for $n \geq 6$ and than Theorem \ref{thm_cyclic_uniform} for $n \geq 8$), but only under the assumption that $f$ has nonsingular leading form.

   In the present work, we make  \emph{no assumption of nonsingularity}. Theorems \ref{thm_cyclic_uniform}, \ref{thm_main}, and \ref{thm_explicit} provide bounds for $N(F,B)$ that approach (\ref{intro_Serre_conjecture}) as $n \rightarrow \infty$. Moreover, Theorem \ref{thm_allowable_not_strongly} and Theorem \ref{thm_genuine_not_strongly} (stated below) admit the exponent predicted by (\ref{intro_Serre_conjecture}).

\subsection{Singularities, sieves, and stratification:} Our   strategy is to count solutions contributing to $N(F,B)$ via a
 polynomial sieve,  as formulated in \cite{BonPie24}. The unconditional sieve lemma we employ (Lemma \ref{lemma_poly_sieve}) imposes that $m \geq 2$ in the depiction (\ref{F_dfn_intro}) of $F(Y,\bX)$; under GRH, an analogous sieve lemma also holds true for $m = 1$ (Lemma \ref{lemma_poly_sieve_cor}). Aside from this application of the sieve lemma, the remainder of the paper does not distinguish between $m \geq 2$ and $m=1$. 

 After an application of a polynomial sieve lemma, our central concern is to understand for many primes $p$ the local counting function
\[v_p(\bx) = \#\{y\modd{p}: F(y,\bx) = 0 \modd{p}\}.\]
In particular, it is crucial to estimate for each $\bu \in \Z^n$   the associated complete sum
\[
S(\bfu,p) = \sum_{\bfa \modd{p}} (v_p(\bfa)-1)e_{p}(\bfu\cdot \bfa).
\]

A celebrated stratification result of Katz and Laumon \cite{KatLau85} (see also Fouvry \cite{Fou00}) provides a set of nested varieties that control how large $|S(\bu,p)|$ can be as $\bu$ varies. The version of the stratification we require is as follows: 
there exists a constant $C_F\geq 1$, an integer $N_F\geq 1$, and homogeneous varieties $\A^n_{\Z[1/N_F]} \supset V_1\supset V_2 \supset...\supset V_n$, where $V_j$ has codimension at least $j$, such that for any prime $p \ndiv N_F$ and for $\bfu\not\in V_j(\F_p)$, 
$$|S(\bu,p)|\leq C_F p^{\frac{n+j-1}{2}}.$$ Thus in particular if $\bu \in V_j(\F_p) \setminus V_{j+1}(\F_p)$, then 
$|S(\bu,p)|\leq C_F p^{\frac{n+j}{2}}$. 
Previous works that employed the polynomial sieve (such as \cite{Mun09, Bon21, BCLP23, BonBro23, BonPie24}) only encountered two possibilities:  $|S(\bu,p)|\leq C_F p^{\frac{n}{2}}$ ($\bu$ is good) or $|S(\bu,p)|\leq C_F p^{\frac{n+1}{2}}$ ($\bu$ is bad), since in all those works $F(Y,\bX)$ was assumed to be nonsingular. 
Here we make no such assumption, and we require the full strength of the Katz-Laumon stratification.

To obtain the explicit dependence on $\|F\|$ in Theorems \ref{thm_cyclic_uniform}, \ref{thm_main} and \ref{thm_explicit}, we require an even stronger statement of stratification, which makes explicit whether (and how) the following parameters depend on the coefficients of $F$: $C_F, N_F$, and the degree, number of irreducible components, and defining polynomials of each stratum $V_j$. This uniform stratification is new: it  is proved by two of the authors with Emmanuel Kowalski, in the companion paper \cite{BKW25x}, and recorded here in Theorem \ref{thm_strat_details}.  

Even with such a uniform stratification in hand, an essential difficulty   arises when applying the stratification to bound $N(F,B)$. This occurs if an irreducible component, say $X_{j}$, of some stratum $V_j$ is degenerate, that is, lies in a proper linear subspace of $\mathbb{A}^n(\Z)$ (and in particular is hence contained in a hyperplane, say $W_{j}$). 
We show that to handle the contribution of   a possible degenerate component, it would suffice to prove that aside from a set $\mathcal{E}$ of exceptional primes with  $|\mathcal{E}| \ll_{n,\deg F}\log \|F\|$, for each prime $p\not\in \mathcal{E}$,  
\beq\label{second_moment_W} \sum_{\bu \in W_j(\F_p)} |S(\bu,p)|^2 \ll_{n,\deg F} p^{2n-1}.
\eeq
Since there are $O(p^{n-1})$ terms in the sum, this is a statement that $|S(\bu,p)|\ll p^{n/2}$ on average over $\bu \in W_j(\F_p)$.
 The proof of the second moment bound (\ref{second_moment_W}) is delicate: it is precisely where the assumption that $F(Y,\bX)$ is strongly $(1,n)$-allowable  becomes relevant (in the proof of Theorem \ref{thm_explicit}).

Precisely, we are able to prove that (\ref{second_moment_W}) holds as long as for a certain change of variables $L \in \mathrm{GL}_n(\Q)$ with integer entries (determined by the hyperplane $W_j$), the polynomial $F_L(Y,\bX):=F(Y,L(\bX))$ has the property that aside from an exceptional set $\mathcal{E}_L$ of primes with  $|\mathcal{E}_L|\ll_{n,\deg F}\log \|F_L\|$, for each prime $p \not\in \mathcal{E}_L$,
\beq\label{B_F_intro_explanation}
\#\{(a_2,\ldots,a_n)\in \F_p^{n-1}: F_L(Y,X_1,a_2,...,a_n) \text{ is reducible over $\overline{\F_p}$}\} \ll_{n, \deg F} p^{n-2}.
\eeq
In a key step, we show that (\ref{B_F_intro_explanation}) holds   as long as $F_L(Y,\bX)$ is ``strongly $n$-genuine.'' 
Here, and throughout,  
we say an absolutely irreducible polynomial $G(Y,\bX) \in \Z[Y,X_1,\ldots,X_n]$ is an  \textbf{$n$-genuine polynomial} if  $\Q(X_1,...,X_n)[Y]/(G(Y,\bX))$ is an  $n$-genuine extension. Similarly, we say an absolutely irreducible polynomial $G(Y,\bX) \in \Z[Y,X_1,\ldots,X_n]$ is a  \textbf{strongly $n$-genuine polynomial} if  $\Q(X_1,...,X_n)[Y]/(G(Y,\bX))$ is a strongly $n$-genuine extension.  That is to say,  $F(Y,\bX)$ is $(1,n)$-allowable (resp. strongly $(1,n)$-allowable)  precisely when $F_L(Y,\bX)$ is  $n$-genuine (resp. strongly $n$-genuine) for \emph{all} $L \in \GL_n(\Q)$.

This series of ideas leads ultimately to   Theorem \ref{thm_explicit}  with at most polylog dependency on $\|F\|$: if $F(Y,\bX)$ is strongly $(1,n)$-allowable then $F_L(Y,\bX)$ is strongly $n$-genuine for all $L$, by definition. Thus we can show (\ref{B_F_intro_explanation}) holds, and hence the second moment bound (\ref{second_moment_W}) can be obtained for all hyperplanes $W$, and in particular for any hyperplane that contains at least one degenerate component of a stratum $V_j$. Here  the full strength of the new uniform stratification provided by the companion paper  \cite{BKW25x}  by Bonolis, Kowalski and Woo is also required to control the size of the  set $\mathcal{E}_L$ of exceptional primes for which (\ref{B_F_intro_explanation}) may not hold.

On the other hand, if $F_L(Y,\bX)$ is $n$-genuine but not strongly $n$-genuine, we do not prove (\ref{B_F_intro_explanation}), and so we do not prove (\ref{second_moment_W}). Instead, it is possible to exploit the property of not being strongly $n$-genuine and save a variable, without applying the polynomial sieve:
\begin{thm}\label{thm_genuine_not_strongly}
Let $F(Y,\bX)  \in \Z[Y,X_1,\ldots,X_n]$ with $\deg_Y F \geq 2$ be $n$-genuine but not strongly $n$-genuine. Then there exists an integer $e(n) \geq 1$ such that for all $B \geq 1$, 
\[ N(F,B) \ll_{n, \deg F} \log(\|F\|+2)^{e(n)} B^{n-1}(\log (B+2))^{2e(n)}.\]
\end{thm}
 This result concurs with the exponent predicted in (\ref{intro_Serre_conjecture}), in relation to Serre's conjecture. We deduce Theorem \ref{thm_allowable_not_strongly} from this result, but the polylog dependency on $\|F\|$ is degraded, as we explain in \S \ref{sec_deduce_thm_allowable_not_strongly}. We combine Theorem \ref{thm_genuine_not_strongly} with the polynomial sieve (in the form of Theorem \ref{thm_main_options}) to complete the proof of Theorem \ref{thm_main}. Here too the  polylog dependency on $\|F\|$ is degraded, as we explain in   \S \ref{sec_deduce_thm_main}.
See Example \ref{example_gen_not_strongly} for infinite classes of polynomials satisfying the hypothesis of Theorem \ref{thm_genuine_not_strongly}.

\subsection{Organization}
In \S \ref{sec_preliminary} we assemble a wide array of  counting results and algebraic facts we will call upon, with careful attention to  explicit dependence on $\|F\|$. In \S \ref{sec_genericity} we state a  theorem  that $(1,n)$-allowable and strongly $(1,n)$-allowable polynomials are generic (Theorem \ref{thm_genericity}, Corollary \ref{cor_genericity_M_Y}), also among singular polynomials (Corollary \ref{cor_genericity_sing}), and prove several useful facts about these classes of polynomials.   In \S \ref{sec_not_strongly_allowable} we focus on polynomials that are $n$-genuine but not strongly $n$-genuine, and prove Theorem \ref{thm_genuine_not_strongly_NFB}, from which we deduce Theorems \ref{thm_allowable_not_strongly} and \ref{thm_genuine_not_strongly}. In \S \ref{sec_strongly}, under the hypothesis that $F$ is strongly $n$-genuine, we prove  second moment bound related to (\ref{second_moment_W}).  In \S \ref{sec_apply_strat_to_sieve} we initiate the polynomial sieve, and the main claim is Theorem \ref{thm_main_options}. Theorem \ref{thm_main_options} immediately implies Theorem \ref{thm_explicit}, which in turn immediately implies Theorem \ref{thm_cyclic_uniform} (by Example \ref{ex: generic cyclic}). We show that Theorem \ref{thm_main_options} and Theorem \ref{thm_genuine_not_strongly_NFB} together imply Theorem \ref{thm_main}.  To prove Theorem \ref{thm_main_options}, we introduce the polynomial sieve lemma to bound $N(F,B)$, dissect the sieve terms using the Katz-Laumon stratification   in the new uniform version of Bonolis, Kowalski and Woo \cite{BKW25x}, and bound all of the resulting terms.    In \S \ref{sec_combining_estimates} we combine all the estimates for terms in the sieve lemma, and conclude the proof of Theorem \ref{thm_main_options}.  
In \S \ref{sec_examples} we construct examples necessary to verify genericity in \S \ref{sec_genericity}; examples of classes of polynomials to which each main theorem applies; examples that verify   this paper recovers nearly all the previous literature on these problems.

\section{Preliminaries}\label{sec_preliminary}
We will apply several existing bounds for counting points within our argument, and because the uniformity   of these estimates with respect to the underlying polynomial coefficients will play a role, we assemble   precise statements here.  
\begin{lem}[Trivial bound]\label{lemma_Schwartz_Zippel}
Let $F \in \Z[X_1,\ldots,X_n]$ be a   polynomial of degree $e \geq 1$, and $S \subset \Z$ a finite subset. Then 
\[ \#\{(x_1,\ldots,x_n) \in S^n : F(x_1,\ldots,x_n)=0\} \leq e |S|^{n-1}.\]
\end{lem}
The proof is by induction on dimension, and may be found in many places, such as \cite[Thm. 1]{HB02} or \cite[Lemma 10.1]{BCLP23}; it is sometimes called the Schwartz-Zippel bound. While the proof given
in \cite[Lemma 10.1]{BCLP23} is stated in the case where $F$ is homogeneous, it also applies in the nonhomogeneous case via trivial modifications.  

We will call upon a result to count the points in a ``box''   in $\F_p^n$ lying on a variety with a given codimension; we quote a special case of a lemma of Xu \cite[Lemma 1.7]{Xu20}:
 \begin{prop}\label{prop_Xu_count}
Let $X \subset \mathbb{A}_{\F_p}^n$ be a subscheme of codimension $j$, and let $d$ be the sum of the dimensions of its irreducible components. If $\{B_i\}_{i=1}^n$ are subsets of $\F_p$, consider the ``box'' $\mathcal{B} = \prod_i B_i \subset \mathbb{A}_{\F_p}^n$. Assuming $1 \leq \# B_1 \leq \#B_2 \leq \cdots \leq \# B_n < \infty$, 
\[ \# (X(\F_p) \cap \mathcal{B}) \leq d \prod_{i=j+1}^n \# B_i = d (\# \mathcal{B}) (\prod_{i=1}^j \# B_j)^{-1}.\]
 \end{prop}
The Lang-Weil bound provides both a main term and an error term with good uniformity:
\begin{lem}[Lang-Weil {\cite[Thm. 1]{LanWei54}}]\label{lemma_Lang_Weil}
Let $V \subset \mathbb{P}^n$ be a variety of dimension $m$ and degree $d$ defined over a finite field $\F_q$ and irreducible over $\overline{\F_q}$. Then
\[ \# V(\F_q) = q^m + E\]
where $E = (d-1)(d-2)q^{m-1/2} + O_{n,m,d}(q^{m-1})$, in which the implied constant only depends on $n,m,d$.  
\end{lem}
See e.g. \cite[p. 184]{Ser97} for a  formal statement. We apply its consequence in an affine setting:
\begin{lem}\label{lemma_Lang_Weil_cor}
Let $F\in \Z[X_1,\ldots,X_k]$ have degree $d$ (not necessarily homogeneous) and let $p$ be a prime such that the affine variety $V(F):=\{F(\bX)=0\}\subset \mathbb{A}_{\F_p}^k$ of dimension $k-1$ is absolutely irreducible (that is, irreducible over $\overline{\F_q}$). Then 
\[ \# \{(x_1,\ldots,x_k) \in \F_p^k: F(\bx)=0\} = p^{k-1} + O_{k,d}(p^{k-1-1/2}).\]
\end{lem}
\begin{proof} 
By homogenizing, we consider the polynomial $\tilde{F}(X_0,\ldots,X_k) = X_0^d F(X_1/X_0,\ldots,X_k/X_0)$. Then the projective closure of the affine variety defined as the vanishing set of $F$ in $\mathbb{A}^k$ is $V(\tilde{F}) \subset \mathbb{P}^k$. The projective closure is irreducible if and only if the affine variety is (see \cite[\S 4.5 Ex. 1]{Sha13}  or \cite[\S5.5 Prop.]{Rei88}), so by the hypothesis, $V(\tilde{F})$ is   irreducible over $\overline{\F_p}$. Thus by Lemma \ref{lemma_Lang_Weil}, $\# V(\tilde{F}) = p^{k-1}+O_{k,d}(p^{k-1-1/2})$. Let $U_0 = \mathbb{P}^k \setminus V(X_0)$, where $V(X_0)=\{X_0=0\}$. Then $V(\tilde{F})$ is the union of $V(\tilde{F}) \cap U_0$ (the points we wish to count) and $V(\tilde{F}) \cap V(X_0)$, which can be identified with a subset of $\mathbb{P}^{k-1}$ and hence can contribute at most $O(p^{k-2})$ points over $\F_p$, and the proposition is verified.
\end{proof}

\subsection{Variants of Noether's Lemma}
We will frequently apply Noether's Lemma, (for example in Theorem \ref{thm_genericity} and Lemma \ref{lemma_Cohen_strongly_genuine_reducible}), and for  Lemma \ref{lemma_Cohen_genuine_linear_factor} we need a refined version, which we now specify.

\begin{defin}   Let $K$ be a field, and $\overline{K}$ a given algebraic closure. Let $F(Y,\bX) \in K[Y,X_1,\ldots,X_n]$. Let $e$ denote a multi-degree $e = (e_0,e_1,\ldots,e_n)$ with non-negative integral entries, and set $|e| = e_0+e_1+ \cdots +e_n $. For a given multi-degree $e$ with  $1 \leq |e| < \deg_F$,  we say that $F(Y,\bX)$ satisfies  divisibility condition $\mathcal{D}(e)$ over $\overline{K}$ if there exists a factorization $$F(Y,\bX) = G(Y,\bX)H(Y,\bX)$$
    where $G$ and $H$ lie in $\overline{K}[Y,X_1,\ldots,X_n]$, $\deg H <\deg F$, and $\deg_Y G \leq e_0,$ $\deg_{X_j} G \leq e_j$ for $j = 1,\ldots, n$. 
\end{defin}
We will later apply this in two settings:   $F(Y,\bX) \in K[Y,X_1,\ldots,X_n]$ with total degree $\deg F >1$ is absolutely irreducible (that is, irreducible over $\overline{K}$) precisely when $F$ does not satisfy condition $\mathcal{D}(e)$ for any multi-degree $e$ with $1 \leq |e| < \deg_F$; this relates to  the classical formulation of Noether's lemma. Second, if $F(Y,\bX)$ with $\deg_Y F \geq 2$ has a linear factor in $Y$ over $\overline{\mathbb{Q}}$ (see  (\ref{dfn_Y_linear_factor})) then $F$ satisfies condition $\mathcal{D}(e)$ for some multi-degree $e=(1,e_1,\ldots,e_n)$ with $e_0=1$.  

To detect when a polynomial $F(Y,\bX)$ satisfies a divisibility condition $\mathcal{D}(e)$, we record a variant of Noether's lemma, which closely follows the presentation and proof of \cite[Ch. V Thm. 2A]{Sch76}.
\begin{lem}[Variant of Noether's Lemma]\label{lemma_Noether}
    Let $K$ be a field. Fix $D \in \Z$ with $D\geq 2$.
\begin{enumerate}[(i)]
    \item Let $\mathcal{D}(e)$ be a divisibility condition for a fixed multi-degree $e =(e_0,e_1,\ldots,e_n)$ with $1 \leq |e|<D$. Then there exist   forms $G_1,...,G_s$ in variables $(A_{i_0,...,i_n})_{i_0+\cdots+i_n\leq D}$ such that a polynomial $$F(Y,\bX) = \sum_{i_0+\cdots+i_n\leq D} a_{i_0,...,i_n} Y^{i_0}X_1^{i_1}\cdots X_n^{i_n}$$
    \beq\label{Noether_condition}
    \text{satisfies $\mathcal{D}(e)$ over $\overline{K}$ or is of degree $\deg F<D$,}
    \eeq
    if and only if $$G_j((a_{i_0,...,i_n})) = 0, \forall j=1,...,s.$$
The forms $G_1,...,G_s$ depend only on $n, D$ and $e$, and are independent of the field $K$ in the sense that if $\mathrm{char}(K)=0$ they have rational integer coefficients and if $\mathrm{char}(K) = p\neq 0$, the polynomials are obtained by reducing the integral coefficients modulo $p$. 
Moreover, $s = O_{n,D,e}(1)$, and $\deg G_j \ll_{n,D,e} 1$ for all $j=1,\ldots,s$. If $\mathrm{char}(K)=0$, 
\[ \|G_j\| \ll_{n,D,e} 1\] 
  for all $j=1,\ldots,s$. 

\item The same result as (i) holds if (\ref{Noether_condition}) is replaced by: is reducible over $\overline{K}$ or is of degree $\deg_F<D$.
\end{enumerate}
\end{lem}
\begin{proof}
The result (ii) is the classical statement proved in \cite[Ch. V Thm. 2A]{Sch76}, so we only prove (i).
Let us suppose that $F(Y,\bX)$ satisfies $\mathcal{D}(e)$ over $\overline{K}$. Let us assume that $G(Y,\bX)\mid F(Y,\bX)$ over $\overline{K}$ and denote the coefficients as $$G(Y,\bX) = \sum_{\substack{i_j \leq e_j\\ j=0,\ldots,n}}b_{i_0,...,i_n} Y^{i_0}X_1^{i_1}...X_n^{i_n}.$$
(Incidentally, to prove (ii) the only change in this argument is to specify not $i_j \leq e_j$ but $i_0+\cdots+i_n \leq D-1$ at this step.)
We also denote the quotient $F(Y,\bX)/G(Y,\bX)$ as $H(Y,\bX)$ with coefficients: 
$$H(Y,\bX) = \sum_{i_0+...+i_n \leq D-1} c_{i_0,...,i_n} Y^{i_0}X_1^{i_1}...X_n^{i_n}.$$
We note that there exists such a $G(Y,\bX)$ and $H(Y,\bX)$ in $\overline{K}[Y,\bX]$, where $G(Y,\bX)$ satisfies $\deg_Y G \leq e_0$ and $\deg_{X_j}G \leq e_j$ for each $j=1,...,n$, if and only if for some fixed choice of coefficients $\{b_{i_0,...,i_n}\}$, there is a nontrivial solution to the system of equations in the variables $c$ and $ \{c_{i_0,...,i_n}\}$ 
given by: $$\sum_{j_0+k_0=i_0}...\sum_{j_n+k_n=i_n} b_{j_0,...,j_n}c_{k_0,...,k_n} - c \cdot a_{i_0,...,i_n}=0,$$
with one equation for each multi-degree $(i_0,i_1,\ldots,i_n)$ that can appear in $F$ (of total degree $\leq D$).
The number of equations in this system is given by $\binom{D+n}{n}$, as that is the number of $(n+1)$-tuples of non-negative indices  with $i_0+\cdots+i_n\leq D$.  
Let us denote by $K$ the number variables $c$ and $\{c_{i_0,\ldots,i_n}\}$, so that
  $$K \leq  1+ \binom{D-1+n}{n},$$
which depends only on $D$ and $n$.   In particular, $K< \binom{D+n}{n}$.  
 Note that the restriction to $G(Y,\bX)$ satisfying $\mathcal{D}(e)$ does not change the number of equations, nor the number of variables, but how many of the coefficients $b_{j_0,...,j_n}$ are forced to be zero. 

Given values for the $\{b_{i_0,\ldots,i_n}\}$ and $\{a_{i_0,\ldots,i_n}\}$, for the above linear system to have a nontrivial solution in  the variables $c$ and $\{c_{i_0,\ldots,i_n}\}$, it is equivalent to require the determinant of each $K\times K$ submatrix (of the corresponding matrix) to be zero. This is equivalent to a system of (determinant) forms of degree $K$ vanishing, say $$\Delta_i(\{b_{j_0,...,j_n}\},\{a_{i_0,...,i_n}\})=0, \qquad i =1,...,r,$$
where we denote the number of such $K \times K$ determinants by  $r$, which only depends on $D$, $n$, and $e$.   These determinants can be viewed as polynomials with coefficients in terms of $\{a_{i_0,...,i_n}\}$, evaluated at variables $\{b_{i_0,\ldots,i_n}\}$. 

By \cite[Ch. V Thm. 1A]{Sch76}, there exist forms $h_1,...,h_s$ in the coefficients of $\Delta_1,...,\Delta_r$ (i.e. in polynomials of $\{a_{i_0,...,i_n}\}$) such that $\Delta_1=\Delta_2=...=\Delta_r=0$ is solvable in the $\{b_{j_0,...,j_n}\}$ if and only if $h_1=...=h_s=0$ when evaluated at $\{a_{i_0,...,i_n}\}$. These forms $h_i$ are moreover of degree $O_{D,n,e}(1)$; if $\mathrm{char}{K}=0$, $\|h_i\| = O_{D,n,e}(1)$ for $i=1,\ldots,s$ (see \cite[Ch. V Thm. 1D]{Sch76}).

Now for each $i=1,\ldots,r$ we define a form $G_i$ by substituting into $h_i$ the coefficients of the forms $\Delta_1,\ldots,\Delta_r$,  as polynomials in the $\{a_{i_0,\ldots,i_n}\}$. We recall that the coefficients of $\Delta_i$, as a function of $\{b_{i_0,\ldots,i_n}\}$, can be at most linear in $a_{i_0,...,i_n}$. Additionally, since $\Delta_1,...,\Delta_r$ range over all $K\times K$ minors, there exists a nonempty subset of the polynomials $\{\Delta_i\}$ where the coefficients have degree 1 in $\{a_{i_0,...,i_n}\}.$ So, the composition of $h_j$ with these linear polynomials in $a_{i_0,...,i_n}$ give us forms $G_j(\{a_{i_0,...,i_n}\})$ such that $F(Y,\bX) = G(Y,\bX) H(Y,\bX)$ if and only if $$G_j(\{a_{i_0,...,i_n}\})=0, j=1,...,s.$$
We observe that these forms $G_j$ cannot be uniformly zero, as that would imply that any polynomial $F(Y,\bX)$ is reducible.
Moreover, $G_j$ has degree and coefficients that are $O_{D,n,e}(1).$
Indeed, Schmidt  notes that $\deg G_j \leq (K-1)^{2^{K-1}}$ and 
$\|G_j\| \leq 4^{{K-1}^{2^{K-1}}}$ \cite[Ch. V Thm. 2A]{Sch76}.

We note that  the statement of \cite[Ch. V Thm. 1A]{Sch76} does not explicitly quantify the size of $s$ in terms of $n,D,e$. But the functions $h_1,\ldots,h_s$ are a resultant system of $r$ forms in the variables $\{a_{i_0,\ldots,i_n}\}$. A resultant system for $R$ forms in $S$ variables, when $R \geq S$, is the set of all $S \times S$ subdeterminants of an associated $R \times S$ matrix, so that the number of forms in the resultant system is $\ll_{R,S} 1$; see \cite[p. 179]{Sch76}. Thus in the statement of the lemma at hand, $s \ll_{n,D,e} 1$. 
\end{proof}
\begin{rem}
    Observe that the divisibility condition $\mathcal{D}(e)$ forces certain entries of the $K\times K$ submatrices to be zero. Nevertheless, we can still view the determinants of the $K\times K$ submatrices as polynomials in $\{b_{i_0,...,i_n}\}$ with coefficients in terms of $\{a_{i_0,...,i_n}\}$, as we have done above.
\end{rem}

One immediate corollary is the following standard fact.
\begin{cor}\label{cor_Noether_irred_generic}
For $D\geq 2$, let $\calM_n(D,k_1,...,k_n)$ denote the moduli space of polynomials in $\Q[Y,X_1,\ldots,X_n]$ that are monic in $Y$ and satisfy $\deg_Y(F) = D$ and $\deg_{X_i}(F) \leq k_i.$ In this moduli space, absolutely irreducible polynomials are generic.
\end{cor}
Indeed, by Noether's Lemma \ref{lemma_Noether} (ii), for a degree $D \geq 2$  there exist forms $G_1,\ldots, G_{s}$ such that if $F(Y,\bX)$ is reducible over $\overline{\Q}$ (or $\deg F<D$) then $G_i$ vanish when evaluated at $F$, verifying the corollary.

The following fact is a consequence of Lemma \ref{lemma_Noether}, just as  
\cite[Ch. V Cor. 2B]{Sch76} follows from \cite[Ch. V Thm. 2A]{Sch76} (as initially observed by Ostrowski).
\begin{lem}\label{lemma_Bertini_Noether}
Let $F(\bX) \in \Z[X_1,\ldots,X_n]$ have degree  $D \geq 1$. (i) If $F$ does not satisfy divisibility condition $\mathcal{D}(e)$ over $\overline{\Q}$, then there exists a finite set $\mathcal{E}_e$ of exceptional primes with $|\mathcal{E}_e| \ll_{n,D,e} \log \|F\|/ \log \log \|F\|$   such that for $p \not\in \mathcal{E}_e$, the reduction of $F$ modulo $p$ is again of degree $D$ and does not satisfy divisibility condition $\mathcal{D}(e)$ over $\overline{\F_p}$. 
(ii)  If $F$ is absolutely irreducible over $\Q$, then there exists a finite set $\mathcal{E}$ of exceptional primes with $|\mathcal{E}| \ll_{n,D} \log \|F\|/ \log \log \|F\|$  such that for $p \not\in \mathcal{E}$, the reduction of $F$ modulo $p$ is again of degree $D$ and is irreducible over $\overline{\F_p}$. 
\end{lem}

 \begin{proof}
For (i), let $F(\bX) \in \Z[X_1,\ldots,X_n]$ be given, with degree $D$, and let $\mathcal{D}(e)$ be a given divisibility condition with $|e| < D$. By hypothesis, $F$ does not satisfy $\mathcal{D}(e)$ over $\Q$ (and has degree $D$) so among the forms $G_1,\ldots,G_s$ provided by Lemma \ref{lemma_Noether} there is at least one, say $G_{j^*}$, such that when evaluated at the coefficients of $F$, $G_{j^*}(\{a_{i_0,\ldots,i_n}\}) \neq 0$. By Lemma \ref{lemma_Noether}, $F$ satisfies $\mathcal{D}(e)$ over $\overline{\F_p}$ or has degree $< D$ if and only if $p$ belongs the exceptional set $\mathcal{E}_e$ of primes such that $p |G_{j^*}(\{a_{i_0,\ldots,i_n}\})$.  
As a nonzero integer, 
\[ |G_{j^*}(\{a_{i_0,\ldots,i_n}\})| \leq \|G_{j^*}\| \|F\|^{\deg G_{j^*}}.\]
By Lemma \ref{lemma_Noether} (i), $\|G_{j^*}\| \ll_{n,D,e} 1$.
For any nonzero integer $t$, the number $\om(t)$ of distinct prime divisors of $t$ satisfies $\om(t) \ll \log t / \log \log t$. 
Consequently $|\mathcal{E}_e| \ll_{n,D,e} \log \|F\| / \log \log \|F\|$, as claimed. 

For (ii), suppose $F(\bX) \in \Z[X_1,\ldots,X_n]$ of degree $D$  is irreducible over $\overline{\Q}$. One can argue as above using the forms $G_1,\ldots, G_s$ provided by Lemma \ref{lemma_Noether} (ii), or observe that $F$ does not satisfy $\mathcal{D}(e)$ for any multi-degree $e$ with $1 \leq |e| < D$. Let $\mathcal{E} = \cup_e  \mathcal{E}_e$ denote the finite union over multi-degrees $e$ with $1 \leq |e| < D$. Then $|\mathcal{E}| \ll_{n,D} \log \|F\| / \log \log \|F\|$ and for $p \not\in \mathcal{E}$, $F$ does not satisfy condition $\mathcal{D}(e)$ over $\overline{\F_p}$ for any multi-index $e$, and is hence irreducible over $\overline{\F_p}$. 

 \end{proof}

\subsection{Hilbert's irreducibility theorem}
 We record the following version of the Hilbert Irreducibility Theorem, which we   use extensively in the next sections; a proof  can be found  in \cite[Ch. 13 and Ch. 14 \S 3]{FriJar23} or in \cite[Ch. 9, pp. 233-235]{Lan83}.

\begin{thm}[Hilbert]\label{thm_HIT}
Let $F_{1},...,F_{s}\in\mathbb{Q}(X_{1},...,X_{n},T_{1},...,T_{r})$ be irreducible polynomials in $n+r$ variables. There exists a dense subset $U\subset\mathbb{A}^{r}(\Q)$ such that for any $(t_{1},...,t_{r})\in U$, $F_{1}(X_{1},...,X_{n},t_{1},...,t_{r}),...,F_{s}(X_{1},...,X_{n},t_{1},...,t_{r})$ are irreducible over $\Q$.
\end{thm}

\subsection{Counting via the determinant method}
We will apply a result developed from the Bombieri-Pila determinant method, which has suitable uniformity, and a savings that depends on the degree of the variety:
 \begin{thm}[Pila {\cite[Thm. A]{Pil95}}]\label{thm_Pila_multidim}
Let $V$ be an irreducible affine variety defined over $\R$ of dimension $m$ and degree $d$ in $\mathbb{A}^n$. The number $N(H)$ of integral points $(x_1,\ldots,x_n) \in \Z^n$ lying on $V$ with $\max_i|x_i| \leq H$ satisfies 
\[  N(H) \ll_{n,m,d} H^{m-1+1/d}\exp(12\sqrt{d\log H\log \log H}) \]
for all $H \geq 1$. Note that given $\ep>0$, for all $H \gg_{d,\ep} 1$, $\exp(12\sqrt{d\log H\log \log H})\leq H^\ep$. 
\end{thm}

\subsection{Counting with a smooth weight}
A smoothed count  preserves the upper bound from a sharp cut-off count for ``points in a box,'' up to a negligible factor.
\begin{lem}\label{lemma_weight_count}
Let $W(\bu) = w(\bu/B)$ where $B \geq 1$ and $w$ is a non-negative infinitely differentiable function that is compactly supported in $[-2,2]^n$. 
\begin{enumerate}[(i)]
\item For any $M \geq 1$, 
$$\hat{W}(\bfu) \ll_n B^n \prod_{i=1}^n \left(1+|u_i|B\right)^{-M}.$$
\item  Suppose an affine variety $V \subset \mathbb{A}_\Z^n$ has the property that for each $H\geq 1$, 
\beq\label{hypothesis_in_box}
\#\{\bu \in \Z^n \cap V \cap [-H,H]^n\} \ll H^\kappa,
\eeq
for a given $\kappa = \kappa(V) \in [1,n]$. Let $A\geq 0$ be fixed. Then for any $Y >0$ and any $\ep>0$,   by taking $M \gg_{n,\kappa,\ep} 1$ sufficiently large,
\[  \sum_{\bu \in \Z^n \cap V} \prod_{i=1}^n \left(1+\frac{|u_i|}{Y}\right)^{-M}(\log(1+ |\bu|))^A\ll_{n,A,\ep} \max\{ 1,Y^{\kappa + \ep}\} .\]
\item  Suppose an affine variety $V \subset \mathbb{A}_{\F_p}^n$ has the property that for any subset $I \subset \F_p$ and corresponding ``box'' $I^n \subset \F_p^n$, 
\beq\label{hypothesis_in_box_p}
\#(V(\F_p)\cap I^n) \ll |I|^\kappa,
\eeq
for a given $\kappa = \kappa(V) \in [1,n]$. Let $A \geq 0$ be fixed. Then for $Y \geq 1$, for any sufficiently small $\ep>0$, by taking $M \gg_{n,\kappa,\ep
} 1$ sufficiently large,
\[  \sum_{\substack{\bu \in \Z^n \\ \bu_p \in V(\F_p)}}\prod_{i=1}^n \left(1+\frac{|u_i|}{Y}\right)^{-M}(\log(1+ |\bu|))^A\ll_{n,A,\ep}  \begin{cases}
    Y^{\kappa + \ep}, & \text{$Y \ll p^\sig,$ $\sig < 1$}\\
    (Y/p)^n p^\kappa Y^{\ep}& \text{$Y \gg p$}
\end{cases}
 ,\]
in which $\bu_p \in V(\F_p)$ indicates that the reduction of $\bu \modd{p}$ is identified with a point in $V(\F_p)$.
\end{enumerate}
\end{lem}
\begin{proof}
The bound for $\hat{W}(\bu)$ holds because $\hat{W}(\bu) = B^n \hat{w}(B\bu )$, and integrating $\hat{w}(\bx)$ by parts $M$ times in each coordinate proves the rapidly decaying upper bound $\hat{w}(\bx) \ll \prod_i (1+|x_i|)^{-M}$, for any $M \geq 1$. 

For the second claim, if $0<Y<1$  it suffices to observe
the left-hand side is bounded above, for sufficiently large $M$, by
\[  \sum_{\bu \in \Z^n } \prod_{i=1}^n (1+ |u_i|)^{-M}(\log(1+ |\bu|))^A \ll_{n,A}
\sum_{\bu \in \Z^n } \prod_{i=1}^n (1+|u_i|)^{-(M-1)} 
\ll 1.\]
Thus from now on we assume $Y \geq 1$.  We divide the sum over $\bu$ into $\mathrm{I}+\mathrm{II}$, which denote the contributions from $\bu \in C_\ep$ and $\bu \not\in C_\ep$, respectively, in which we define 
\[ C_\ep= [-Y^{1+\ep},Y^{1+\ep}]^n\]
for a fixed small $\ep>0$ of our choice. 
Then for any $M \geq 1$, by the hypothesis
(\ref{hypothesis_in_box}), 
\[ \mathrm{I} \ll_{\ep,M} (\log Y)^A \sum_{\bu \in C_\ep \cap V} 1  \ll (\log Y)^A (Y^{1+\ep})^\kappa \ll_{A,\ep} Y^{\kappa + (\kappa +1) \ep},\]
say, using $(\log Y)^A \ll_{A,\ep} Y^\ep$.
On the other hand,
\beq\label{C_ep_estimate}  \mathrm{II} \leq \sum_{\bu \in \Z^n \setminus C_\ep} \prod_{i=1}^n \left(1+\frac{|u_i|}{Y}\right)^{-M}
(\log |u_i|)^A \ll_{n,A,\ep} Y^n Y^{-\ep (M-2)}.
\eeq
(This used the fact that $(\log |u_i|)^A \ll_{A,\ep} \log Y$ if $|u_i| \leq Y^{1+\ep}$, and that $(\log |u_i|)^A(1+|u_i|/Y)^{-1} \leq 1$ if $|u_i| \geq Y^{1+\ep}$, for all $Y \gg_{A,\ep} 1$.)
By taking $M\gg_{n,\ep,A} 1$ sufficiently large, this is dominated by as small a power of $Y$ as we like, and in particular by the contribution of $\mathrm{I}$. We then obtain (ii) by taking $\ep>0$ as small as we like.

For the third claim, we again define the box $C_\ep$ as above (for a small $\ep>0$ of our choice) and divide the sum over $\bu$ into $\mathrm{I}+\mathrm{II}$, denoting the contributions from $\bu \in C_\ep$ and $\bu \not\in C_\ep$, respectively. Suppose first $Y \gg p$. (We ultimately do not require this case, but include it for completeness.) Then the integer points in the box $C_\ep$ may be identified as belonging to at most $\ll (Y^{1+\ep}/p + 1)^n$ copies of $\F_p^n$. Thus by the hypothesis (\ref{hypothesis_in_box_p}) and positivity, 
\[\mathrm{I} \ll_{\ep} (\log Y)^A (Y^{1+\ep}/p + 1)^n p^\kappa \ll_{n,A,\ep}  (Y/p)^n p^\kappa Y^{(n+1)\ep}  \qquad (Y \gg p) .\]
If on the other hand $Y \ll p^\sig$ for some $\sig<1$,  by  choosing $\ep$ sufficiently small with respect to $\sig$, we may ensure that  with $C_\ep$ defined as above, $C_\ep \cap \Z^n$  can be identified with a subset, say $I^n \subset \F_p^n$, to which the hypothesis (\ref{hypothesis_in_box_p}) applies directly. Consequently $\mathrm{I} \ll_A (\log Y)^A(Y^{1+\ep})^\kappa \ll_{A,\ep} Y^{\kappa + (\kappa+1)\ep}$. 
In either case, by positivity we may bound 
 $\mathrm{II}$ by summing over all integer tuples in the complement of $C_\ep$ as in (\ref{C_ep_estimate}).  By taking $M$ sufficiently large this is dominated by the contribution of $\mathrm{I}$. 
\end{proof}

\subsection{An estimate with polylog dependency}
 
We prove an unconditional    estimate for $N(F,B)$ with polylog dependence on $\|F\|$. This provides an explicit bound with essentially the same exponent as Serre's bound (\ref{Serre_bound_intro}), which we require later.
\begin{thm}\label{thm_polylog}
Fix an integer $n \geq 1$. Let $F(Y,X_{1},...,X_{n})\in\mathbb{Z}[Y,X_{1},...,X_{n}]$ be irreducible in $\mathbb{Q}[Y,X_{1},...,X_{n}]$ of total degree $D$, with   $\deg_{Y}(F)\geq 2$, and assume that  $F$ has constant-leading-coefficient in $Y$. Then  
\[
N(F,B)\ll_{n,D} (\log (\|F\|+2))^{e(n)} B^{n-\frac{1}{2}}(\log B)^{e(n)}\qquad \text{for all $B \geq 1$,}
\]
  for a constant $e(n)>0$ depending only on $n$.
\end{thm}
The key tool to prove this is a recent result of Cluckers et al, on effective Hilbert's Irreducibility Theorem, via the determinant method.
\begin{thm}[{\cite[Theorem 1.10]{CDHNV23x}}]\label{thm_CDHNV}
Let $F(T,X_{1},...,X_{m})\in\mathbb{Z}[T,X_{1},...,X_{m}]$ be irreducible in $\mathbb{Q}[T,X_{1},...,X_{m}]$ with  $d=\deg_{\bX}(F)\geq 1$. 
 For each integer $B \geq 1$ define
 \[ R_T(F,B) = \{ t \in \Z \cap [-B,B]: \text{$F(t,X_1,\ldots,X_m)$ is reducible in $\Q[X_1,\ldots,X_m]$}\}.\]
 Then there is a constant $h(m)>0$ depending only on $m$ such that 
 for all $B \geq 1$, 
 \[ R_T(F,B) \ll_m 2^{h(m)d}d^{h(m)}(\log (\|F\|+2))^{h(m)}B^{1/2}.
 \]
\end{thm}

\begin{rem}\label{remark_apply_CDHNV}
Note that the hypotheses of the theorem do not require that $\deg_T(F) \geq 1$. Indeed, if $\deg_T(F) =0$, then the hypothesis that $F(T,X_1,\ldots,X_m)$ is irreducible over $\Q$ requires that $F(X_1,\ldots,X_m)$ is irreducible over $\Q$, independent of specializing $T=t$ to any value $t$. Thus the set $R_T(F,B)$ is empty for all $B$, and the theorem is vacuously true. The theorem also does not require $F$ to be monic in $T$.
\end{rem}

To deduce Theorem \ref{thm_polylog}, we start by observing that, since $\deg_{Y}(F)\geq 2$,
\begin{multline}\label{F_solvable_reducible}
\#\{\bfx\in[-B,B]^n \cap \mathbb{Z}^{n}:\text{$F(Y,\bfx)=0$ is solvable over $\Z$}\}\\
\leq \#\{\bfx\in [-B,B]^n \cap \mathbb{Z}^{n}: \text{$F(Y,\bfx)=0$ is reducible in $\Q[Y]$}\}.
\end{multline}
Observe that Theorem \ref{thm_polylog} holds immediately if $\deg_Y(F) \geq 2$ but $\deg_{\{X_1,\ldots,X_n\}}(F)=0$, since then the set on the right-hand side of (\ref{F_solvable_reducible}) must be empty. For in such a case, $F(Y,\bX)=F(Y,\bx)$ for any fixed $\bx \in \Z^n$, so if $F(Y,\bx)$ is reducible then $F(Y,\bX)$ is reducible over $\Q$, which contradicts the hypothesis that $F(Y,\bX)$ is irreducible. 
Thus from now on we may assume $\deg_{\{X_1,\ldots,X_n\}}(F) \geq 1$. 

 If $n=1$, then by Theorem \ref{thm_CDHNV} applied to $F(Y,X_1)$ with $t$ playing the role of $x_1$,
 then 
 \[ \#\{x_1\in [-B,B] \cap \mathbb{Z}: \text{$F(Y,x_1)=0$ is reducible in $\Q[Y]$}\} \ll_{D}( \log (\|F\|+2))^{h(1)}B^{1/2}.\]
 From now on, let $n \geq 2$, and we proceed by induction on the number  $n$ of variables. Let us now assume that   for any polynomial $G\in\mathbb{Z}[Y,X_{1},...,X_{n-1}]$ that is irreducible in $\mathbb{Q}[Y,X_{1},...,X_{n-1}]$ of degree $\deg_Y(G) \geq 2$ and $\deg_{\{X_1,\ldots,X_{n-1}\}}(G) \geq 1$, the following inductive hypothesis holds:
  \begin{multline*}
       \#\{\bx' \in [-B,B]^{n-1} \cap \mathbb{Z}^{n-1}: \text{$G(Y,\bx')=0$ is reducible in $\Q[Y]$}\}\\
       \ll_{n,\deg G}( \log (\|G\|+2))^{e(n-1)}B^{(n-1)-1/2},
       \end{multline*}
  for some  $e(n-1)$. 
 
 Suppose $F(Y,X_1,\ldots,X_n)$ satisfies the hypotheses of Theorem \ref{thm_polylog}. We split up our count according to those $x_n$ such that $F(Y,X_1,...,X_{n-1},x_n)$ is reducible/irreducible in $\Q[Y,X_1,\ldots,X_{n-1}]$.
We now apply Theorem \ref{thm_CDHNV} with $t$ playing the role of $x_n$. (It may be that $\deg_{X_n}F =0$, but note that by Remark \ref{remark_apply_CDHNV}, Theorem \ref{thm_CDHNV} is also applicable in this case.) This leads to:
\begin{multline*}
\#\{x_{n}\in [-B,B]\cap \Z:\text{$F(Y,X_{1},...,X_{n-1},x_{n})$ is reducible in $\Q[Y,X_1,\ldots,X_{n-1}]$}\}
    \\
    \ll_{n,D}(\log (\|F\|+2))^{h(n)} B^{\frac{1}{2}}.
\end{multline*}
Now if $F(Y,X_{1},...,X_{n-1},x_{n})$ is reducible  in $\Q[Y,X_1,\ldots,X_{n-1}]$, then   $F(Y,x_{1},...,x_{n-1},x_{n})$ is   reducible in $\Q[Y]$ for any choice of $x_{1},...,x_{n-1}$, so we count these trivially. Hence the overall contribution to (\ref{F_solvable_reducible}) of the $x_n \in [-B,B] \cap \Z$ such that $F(Y,X_{1},...,X_{n-1},x_{n})$ is reducible is bounded by 
\[  \ll_{n,D}(\log (\|F\|+2))^{h(n)} B^{(n-1)+\frac{1}{2}}
\ll_{n,D}(\log (\|F\|+2))^{h(n)} B^{n-\frac{1}{2}},\]
which is acceptable.

On the other hand, suppose $x_n \in [-B,B] \cap \Z$ is fixed and $F(Y,X_{1},...,X_{n-1},x_{n})$ is irreducible in $\Q[Y,X_1,\ldots,X_{n-1}]$.  To apply the inductive step, we must ensure that 
\beq\label{degs_hold} \deg_YF(Y,X_{1},...,X_{n-1},x_{n}) \geq 2, \qquad \deg_{\{X_1,\ldots,X_{n-1}\}}F(Y,X_{1},...,X_{n-1},x_{n}) \geq 1.
\eeq
The first property holds since $F$ has constant-leading-coefficient in $Y$, and we assumed initially that $\deg_Y(F) \geq 2$. Second, suppose that $\deg_{\{X_1,\ldots,X_{n-1}\}}F(Y,X_{1},...,X_{n-1},x_{n}) =0,$ in which case $F(Y,X_{1},...,X_{n-1},x_{n})\in \Z[Y]$ while also being irreducible (in the case we currently consider).  Then for all choices of $\bX'=\bx'$, $F(Y,\bX',x_{n}) = F(Y,\bx',x_n)$ is irreducible in $\Q[Y]$, and so no contribution to (\ref{F_solvable_reducible})  occurs.  

Now we may suppose $x_n \in [-B,B] \cap \Z$ is fixed, and that $F(Y,X_{1},...,X_{n-1},x_{n})$ is irreducible in $\Q[Y,X_1,\ldots,X_{n-1}]$, and (\ref{degs_hold}) holds.
We can apply the inductive hypothesis, obtaining
\[
\#\{\bfx'\in [-B,B]^{n-1} \cap \Z^{n-1}: \text{$F(Y,\bfx',x_{n})$ is reducible in $\Q[Y]$}\}\ll_{n, D}(\log (\|F_{x_n}\|+2))^{e(n-1)} B^{n-1-\frac{1}{2}}.
\]
Here $\|F_{x_n}\|$ denotes the norm of the polynomial $F(Y,X_1,...,X_{n-1},x_n)$ and satisfies that $\log(\|F_{x_n}\|+2) \ll_{n,D} \log(\|F\|+2)+\log(B)$.
The number of $x_n$ which lead to this case is trivially at most $\ll B$, 
so the overall contribution from all $x_{n}$ for which $F(Y,X_{1},...,X_{n-1},x_{n})$ is irreducible in $\Q[Y,X_1,\ldots,X_{n-1}]$ is $\ll_{n, D}(\log (\|F\|+2)+\log(B))^{e(n-1)} B^{n-\frac{1}{2}}$.

In total,
this argument has proved that for all $B \geq 1$,
\[ 
\#\{\bfx\in [-B,B]^n\cap \mathbb{Z}^{n}:  \text{$F(Y,\bfx)=0$ is solvable over $\Z$}\}\ll_{n,D}(\log (\|F\|+2)+\log(B))^{e(n)}  B^{n-\frac{1}{2}} 
\]
with $e(n) = \max\{h(n),e(n-1)\}.$ This completes the proof of Theorem \ref{thm_polylog}.

\section{Genericity  and key properties}\label{sec_genericity}
 
 In this section we begin with a theorem that $(1,n)$-allowable and strongly $(1,n)$-allowable polynomials are generic. 
 We then prove several key properties of these classes of polynomials, which we will call upon later. 
 
 \emph{Notation:} To work with these classes of polynomials, we define the following notational conventions. For a given subset $I \subseteq \{1,\ldots,n\}$, we denote $\bX_I = (X_i)_{i \in I}$. For example, if $I=\{2,\ldots,n\}$ then $\bX_I = (X_2,\ldots,X_n)$, so that $\Q(\bX_I)=\Q(X_2,\ldots,X_n)$, and a polynomial $G(Y,\bX_I)$ can depend on $Y$ and $X_2,\ldots,X_n$ but not $X_1$. With a slight abuse of notation with respect to the ordering of variables, we indicate by $F(Y,\bX_{I^c},\bx_{I})$, for a given $\bx_I \in \Q^{|I|}$, that $X_i$ is evaluated at $x_i$ for each $i \in I$. We will also denote $(X_2,\ldots,X_n)$ by $\bX'$.   

In this section, we denote by $\calM_n(D,k_1,...,k_n)$ the closed irreducible subset of the moduli space of polynomials in  $\Q[Y,X_1,\ldots,X_n]$ that are monic in $Y$ and satisfy $\deg_Y(F) = D$ and $\deg_{X_i}(F) \leq k_i$. For a given integer $m|D$, we denote by $ \mathcal{Y}_n^m(D,k_1,...,k_n)\subset \calM_n(D,k_1,...,k_n)$ the subset of polynomials that take the form (\ref{F_dfn_intro}). 
 
\begin{thm}[Genericity]\label{thm_genericity}
 Fix integers $D \geq 2$ and $k_1,\ldots, k_n \geq 1$.   There exists a finite collection of polynomials  $\tilde{h}_1,\ldots,\tilde{h}_s$ in the coefficients of $F \in \calM_n(D,k_1,...,k_n)$, not all identically zero on $\calM_n(D,k_1,...,k_n)$, such that if $F\in \calM_n(D,k_1,...,k_n)$  is not strongly $(1,n)$-allowable, then
 $\tilde{h}_1(F) = \cdots = \tilde{h}_s(F)=0$.

 Let $\calM\subset\calM_n(D,k_1,...,k_n)$ be a closed irreducible subset. Moreover, assume that there exists $P\in\calM$ such that $P$ is strongly $(1,n)$-allowable. Then strongly $(1,n)$-allowable polynomials are generic  in $\calM$.

 Fix  an integer $m \geq 1$ with $m|D$. The polynomials  $\tilde{h}_1,\ldots,\tilde{h}_s$ are not all identically zero on $\mathcal{Y}^m_n(D,k_1,...,k_n)$. Let $\mathcal{Y}\subset\mathcal{Y}_n^m(D,k_1,...,k_n)$ be a closed irreducible subset, and assume that there exists $P\in \mathcal{Y}$ which is strongly $(1,n)$-allowable. Then   strongly $(1,n)$-allowable polynomials are generic in $\mathcal{Y}$.
\end{thm}
The key conclusions of the theorem are contingent on the existence of at least one explicit example in $\mathcal{M}_n(D,k_1,\ldots,k_n)$, $\mathcal{M}$, $\mathcal{Y}_n^m(D,k_1,\ldots,k_n)$, or $\mathcal{Y}$ that confirms the polynomials  $\tilde{h}_1,\ldots,\tilde{h}_s$ are not all identically zero on the relevant set. We reserve the construction of these examples to \S \ref{sec_examples}, and  summarize here the consequences, as two corollaries.
 
 \begin{cor}\label{cor_genericity_M_Y}
Fix integers $D \geq 2$ and $k_1,\ldots, k_n \geq 1$. Then strongly $(1,n)$-allowable polynomials are generic in $\calM_n(D,k_1,...,k_n)$. 
Furthermore, polynomials that are not strongly $(1,n)$-allowable have density zero in $\calM_n(D,k_1,...,k_n)$, in the sense that  
 \[ \frac{\#\{ F(Y,\bX)\in \calM_n(D,k_1,...,k_n), \|F\| \leq T: F\text{ is not strongly $(1,n)$-allowable})\}}{\#\{ F(Y,\bX) \in \calM_n(D,k_1,...,k_n), \|F\| \leq T\}}\ll \frac{1}{T}
 \]
  for all $T \geq 1$. 
For each $m|D$, the above results hold for $\mathcal{Y}_n^{m}(D,k_1,...,k_n)$ as well.
\end{cor} 
\begin{proof}[Proof of Corollary \ref{cor_genericity_M_Y}]
 
 For the first claim, it is enough to observe that  $\calM_n(D,k_1,...,k_n)$ is irreducible and that the polynomial
\[
P=Y^{D}-\sum_{i=1}^{n}X_{i}^{k_{i}},
\]
is strongly $(1,n)$-allowable, which is verified   in Example \ref{ex: generic diag} of \S \ref{sec_examples}. Then by Theorem \ref{thm_genericity}, it follows that strongly $(1,n)$-allowable polynomials are generic in $\calM_n(D,k_1,...,k_n)$.

For the next claim of the corollary, recall that at least one among the polynomials  $\tilde{h}_1,...,\tilde{h}_s$ (provided by Theorem \ref{thm_genericity}) is not identically zero  over $\calM_n(D,k_1,...,k_n)$. Let us suppose $\tilde{h}_1$ is not identically zero over $\calM_n(D,k_1,...,k_n)$. Then the left-hand side of the claim is bounded above by  
\[\frac{\#\{ F(Y,\bX)\in \calM_n(D,k_1,...,k_n), \|F\| \leq T: \tilde{h}_1(F) = 0\}}{\#\{ F(Y,\bX) \in \calM_n(D,k_1,...,k_n), \|F\| \leq T\}}\ll \frac{1}{T}.
\]
This is because $\mathcal{M}_n(D,k_1,\ldots,k_n)$ is isomorphic to affine space $\mathbb{A}^M$ (for an appropriate dimension $M$ depending on $n,D,k_1,\ldots,k_n$), so that a trivial estimate (Lemma \ref{lemma_Schwartz_Zippel}) saves at least one variable in the count recorded in the numerator.
 The argument for $\mathcal{Y}_n^{m}(D,k_1,...,k_n)$ is analogous, since if $D=md$ then the   example $P$ lies in $\mathcal{Y}_n^{m}(D,k_1,...,k_n)$, and again $\mathcal{Y}_n^m(D,k_1,\ldots,k_n)$ is isomorphic to affine space (of an appropriate dimension).
\end{proof}
\begin{cor}\label{cor_genericity_sing}
Fix integers $D \geq 2$ and $k_1,\ldots, k_n \geq 1$ and denote
\[
\mathcal{S}_n(D,k_1,...,k_n)=\{P\in\calM_n(D,k_1,...,k_n)\text{: }P\text{ is singular}  \}.
\]
Then strongly $(1,n)$-allowable polynomials are generic in $\mathcal{S}_n(D,k_1,...,k_n)$.
\end{cor}
\begin{proof}[Proof of Corollary \ref{cor_genericity_sing}]
  By Theorem $\ref{thm_genericity}$, it is enough to show that $\mathcal{S}_n(D,k_1,...,k_n)$ is irreducible and that there exists $P\in\mathcal{S}_n(D,k_1,...,k_n)$ that is strongly $(1,n)$-allowable. The fact that $\mathcal{S}_n(D,k_1,...,k_n)\subset\mathcal{M}_n(D,k_1,...,k_n)$ is a closed irreducible subset follows from \cite[Proposition $7.1$, pages $245-246$]{EisHar16}. By  Example \ref{ex: genericsin} of \S \ref{sec_examples} there is a polynomial  $P\in\mathcal{S}_n(D,k_1,...,k_n)$ which is strongly $(1,n)$-allowable.
\end{proof}

 \subsection{Proof of   Theorem \ref{thm_genericity}}
Consider $F(Y,\bX) \in \calM_n(D,k_1,...,k_n)$.
 Assume that $F(Y,\bX)$ is not strongly $(1,n)$-allowable; this occurs if and only if there exists a linear change of variables $\sigma\in \GL_n(\Q)$, some subset $I \subsetneq \{1,...,n\}$, and an irreducible polynomial $G(Y,\bX_I)$ with $\deg_Y G \geq 2$, such that 
\[\Q(\bX)[Y]/F_\sig(Y,\bX)  \supset \Q(\bX)[Y]/G(Y,\bX_I)\supsetneq\Q(\bfX),\]
in which $F_{\sigma}(Y,\bX) :=F(Y,\sigma(\bX))$.
 Without loss of generality, we may suppose that $1 \not\in I$, so that $G$ is independent of $X_1$. In what follows, we denote $M=\Q(\bX)[Y]/F_\sig(Y,\bX)$ and $N=\Q(\bX)[Y]/G(Y,\bX_I)$. Note  that $[N:\Q(\bX)] \geq 2$.   Furthermore,
\beq\label{F_reducible_over_intermediate_field}
\deg_{Y}F_{\sigma}=[M:\Q (\bfX)]=[M:N][N:\Q (\bfX)].
\eeq
Since $[N:\Q (\bfX)] \geq 2$ this implies that $[M:N]<\deg_{Y}F_{\sigma}$, so $F_{\sig}$ cannot be irreducible over $N$. (For later reference, observe that $N = N'(X_1)$ for $N' = \Q(X_2,...,X_n)[Y]/(G(Y,\bX_I))$, so $F_\sigma$ factorizes over $N'$ as well.) Now let us decompose $F_{\sigma}(Y,\bX)$ as 
 \beq\label{genericity_expansion}
F_{\sigma}(Y,\bX) = \sum_{j,k} f_{1,j,k}(X_2,...,X_n;\sigma) Y^j X_1^k.
\eeq

 By Noether's Lemma \ref{lemma_Noether} (ii), there exists a form $G_R$ in the coefficients of the expansion (\ref{genericity_expansion}) such that  $F_\sig(Y, X_{1},x_{2},...,x_{n})$ (for a fixed choice of $x_2,\ldots,x_n$) is reducible over $\overline{\Q}$ or has strictly lower degree than $F_\sig(Y,\bX)$, if and only if $$  G_R(\{(f_{1,j,k}(x_2,...,x_n;\sigma)\}_{j,k})= 0.$$
(While the lemma produces a system of forms, by taking a sum of appropriate powers of these forms, one can construct one such a form $G_R$.)
Since $F_{\sig}(Y,\bX)$ is reducible over $N$ (and hence $N'$ as observed above), it follows that for any $x_2,...,x_n$, $F_{\sig}(Y,X_1,x_2,...,x_n)$ is reducible over $\Q[Y]/(G(Y,x_2,...,x_n))\subset \bar{\Q}$.
 Thus it follows that
$$  G_R(\{(f_{1,j,k}(X_2,...,X_n;\sigma)\}_{j,k})\con 0.$$
Note that for each $j,k$ the coefficients $f_{1,j,k}(X_2,...X_n;\sigma)$   depend polynomially on the entries of $\sigma$, so that this function is itself a polynomial in $X_2,\ldots,X_n$, and the entries of $\sig$.
Thus we may write 
$$G_R(\{(f_{1,j,k}(X_2,...,X_n;\sigma)\}_{j,k}) = \sum_{\bfk \in K} g_{\bfk,F}(\sigma) \prod_{j\neq 1} X_j^{k_j},$$ defining a system of polynomials $g_{\bfk,F}(\sigma)$ indexed by a finite set $K$ of multi-indices $\bfk \in \Z_{\geq 0}^{n-1}$ of bounded order (depending only on $n,D,k_1,\ldots,k_n$). The argument above has shown that if $F(Y,\bX)$ is not  strongly $(1,n)$-allowable, then there must exist a linear change of variables $\sigma\in \GL_n(\Q)$ such that $g_{\bfk,F}(\sigma)=0$ for each $\bfk \in K$.  By \cite[Chapter V Thm. 1A]{Sch76}, there exists a resultant system of  forms $h_1,...,h_s$ in the coefficients of the polynomials $g_{\bfk,F}$ such that $\{g_{\bfk,F}(\sigma)=0 ,\forall \bfk \in K\}$ has a solution in $\sigma\in \GL_n(\Q)$ if and only if $h_1=...=h_s=0.$ Moreover, the coefficients of $g_{\bfk,F}$ depend polynomially on the coefficients of $F(Y,\bX)$. Thus, we find that   there exists a system of polynomials $\tilde{h}_1,\ldots,\tilde{h}_s$ (functions of the coefficients of $F$) such that  if $F(Y,\bX)$ is not strongly  $(1,n)$-allowable then $$\tilde{h}_1(F)=...=\tilde{h}_s(F)=0.$$
By Example \ref{ex: generic cyclic} of \S \ref{sec_examples} there is a polynomial in $\mathcal{M}_n(D,k_1,\ldots,k_n)$ that is strongly $(1,n)$-allowable, so we know that the polynomials $\tilde{h}_1,...,\tilde{h}_s$ are not all identically zero on $\calM_n(D,k_1,\ldots,k_n)$. 

For the next claim of the theorem, since we are assuming that there exists $P\in\calM$ which is strongly $(1,n)$-allowable, it follows that $\tilde{h}_1,...,\tilde{h}_s$ are not identically zero on $\calM$. Then the result that strongly $(1,n)$-allowable polynomials are generic in $\calM$ follows, since we are assuming $\calM$ to be closed and irreducible.
 
 Finally, again by Example \ref{ex: generic cyclic} (when $m | D$) there is a polynomial in $\mathcal{Y}^m_n(D,k_1,\ldots,k_n)$ that is strongly $(1,n)$-allowable, so we know that the polynomials $\tilde{h}_1,...,\tilde{h}_s$ are not all identically zero on $\mathcal{Y}^m_n(D,k_1,\ldots,k_n)$. In particular, if $\mathcal{Y} \subset \mathcal{Y}^m_n(D,k_1,\ldots,k_n)$ 
is a closed irreducible subset, the claim of the theorem for $\mathcal{Y}$ follows. This completes the proof of the theorem, aside from the explicit construction of Example \ref{ex: generic cyclic}, which we verify in \S \ref{sec_examples}.

\subsection{Properties: $n$-genuine and strongly $n$-genuine polynomials}

 There is an equivalent characterization for being strongly $n$-genuine; this in particular plays a role in \S \ref{sec_strongly} as well as \cite{BPW25x_per}.
\begin{lem}\label{lemma_strongly_n_gen_equiv_alg_closed}
       Let $F\in\mathbb{Z}[Y,X_1,\ldots,X_n]$ be an absolutely irreducible polynomial that is monic in $Y$ with $\deg_Y F \geq 2$ and denote $\mathcal{L}_F=  \mathbb{Q}(X_1,\ldots,X_n)[Y]/F$. Then $F$ is a strongly $n$-genuine polynomial if and only if for every $j\in\{1,\ldots,n\}$,
       \[
       \overline{\mathbb{Q}(X_{i})_{i\neq j}}\cap \mathcal{L}_F=\mathbb{Q}(X_{i})_{i\neq j}.
       \]
   \end{lem}
    \begin{proof}
    Here, for example, the notation $\Q(X_i)_{i \neq 1}$ indicates $\Q(X_2,\ldots,X_n)$.
    Recall $F$ is a strongly $n$-genuine polynomial if and only if $\mathcal{L}_F$ is a strongly $n$-genuine extension of $\mathbb{Q}(X_1,\ldots,X_n)$. For each  $j\in\{1,...,n\}$ denote $N_{j}= (\overline{\mathbb{Q}(X_{i})_{i\neq j}}\cap \mathcal{L_F})(X_{j})$. By construction, $N_j$ is an extension of $\mathbb{Q}(X_1,\ldots,X_n)$ that is not  $n$-genuine.  Now assume $\mathcal{L}_F$ is strongly $n$-genuine: if for some  $j$,  $\mathbb{Q}(X_1,\ldots,X_n)\subsetneq N_{j}\subset \mathcal{L}_F$, this would be a contradiction, since $N_j$ is not an $n$-genuine extension. Thus for each $j \in \{1,\ldots,n\}$, $\mathbb{Q}(X_1,\ldots,X_n) = N_j$, and this implies $ \overline{\mathbb{Q}(X_{i})_{i\neq j}}\cap \mathcal{L}_F=\mathbb{Q}(X_{i})_{i\neq j}.$ In the other direction, suppose $\mathcal{L}_F$ is not strongly $n$-genuine: then there exists some intermediate extension $M'$ with $\mathbb{Q}(X_1,\ldots,X_n)\subsetneq M'\subset \mathcal{L}_F$, which is not $n$-genuine. Since $M'\subset N_{j}$ for some $j$, we have $N_{j}\neq \mathbb{Q}(X_1,\ldots,X_n)$, and this suffices to show $ \overline{\mathbb{Q}(X_{i})_{i\neq j}}\cap \mathcal{L}_F\neq \mathbb{Q}(X_{i})_{i\neq j}.$
    \end{proof}

The following lemma is useful to establish whether  a polynomial is  $n$-genuine or strongly $n$-genuine, and this will play a key role in the study of $(1,n)$-allowable and strongly $(1,n)$-allowable polynomials as well.
\begin{lem}\label{lem : spectoglobal}
Let $F(Y,\bX)\in\ZZ[Y,\bX]$ be an irreducible polynomial, and assume $F(Y,\bX)$ is not strongly $n$-genuine (resp. not $n$-genuine). Then we can find a non-empty subset $I\subsetneq\{1,...,n\}$ such that for every $\bfx_{I}=(x_{i})_{i\in I} \in \Q^{|I|}$,   the polynomial $F_{\bfx_{I}}(Y, (X_i)_{i \not\in I})=F(Y,\bX_{I^c},\bx_I)$ (obtained by the specialization $X_{i}=x_{i}$ for every $i\in I$) is reducible over $\overline{\Q}$ (resp. has a linear factor in $Y$ over $\overline{\Q}$).
\end{lem}
 Here we say that a polynomial $F(Y,\bX) \in \Q[Y,\bX]$ has  a linear factor in $Y$ over $\overline{\Q}$ if $F(Y,\bX) = (Y - Q(\bX)) \Tilde{H}(Y,\bX),$
 in which $Q(\bX)\in \overline{\Q}[\bX]$ and $\Tilde{H}(Y,\bX)\in \overline{\Q}[Y,\bX].$ (Suppose a non-empty subset $I \subsetneq \{1,\ldots,n\}$ has been found such that the conclusion of the lemma holds. Then for every superset $S$ with $I \subseteq S \subseteq \{1,\ldots,n\}$, for every $\bx_S \in \Q^{|S|}$, the polynomial $F_{\bx_S}(Y,(X_j)_{j \notin S})$ has the respective property in the conclusion.)
\begin{proof}
We first prove the lemma in the case that the polynomial is not strongly $n$-genuine. Let $F(Y,\bX)\in\ZZ[Y,\bX]$ be an irreducible polynomial that is not strongly $n$-genuine. This means that we can find $I\subsetneq \{1,...,n\}$ and a polynomial $G\in\ZZ[Y,\bX_i]$ such that
\beq\label{G_sequence}
\Q(\bX) \subsetneq \Q(\bX)[Y]/G\subset \Q(\bX)[Y]/F.
\eeq
In particular, $F(Y,\bX)$ is reducible over $\Q(\bX)[Y]/G$  (by the  reasoning applied earlier in (\ref{F_reducible_over_intermediate_field})). For each $\bx_I \in \Z^{|I|}$, let $G_{\bx_I}(Y)$ denote $G(Y,\bx_I)$.  The result now follows since for every $\bfx_{I}\in\Q^{|I|}$   one has that $\Q(\bX_{I^c})[Y]/G_{\bfx_{I}}\subset\overline{\Q}(\bX_{I^c})[Y]$ and $F_{\bfx_{I}}$ is reducible over $\Q(\bX_{I^c})[Y]/G_{\bfx_{I}}$.

 Let us assume now  that $F(Y,\bX)\in\ZZ[Y,\bX]$ is an irreducible polynomial that is not $n$-genuine. This means that we can find $I\subsetneq \{1,...,n\}$ and a polynomial $G\in\ZZ[Y,\bX_I]$ such that
\beq\label{F_sequence}
\Q(\bX)[Y]/G= \Q(\bX)[Y]/F.
\eeq
In particular, $F(Y,\bX)$ has a linear  factor in $Y$ over $\Q(\bX)[Y]/G$. Let $Z$ be a root of $G$, such that
\[
F(Y,\bfX)=(Y-H(\bfX,Z))R(Y,\bfX,Z).
\]
 For any $\bfx_{I}\in\Q^{|I|}$, $G_{\bfx_{I}}(Y):=G(Y,\bx_I)\in\Q [Y]$, thus a solution $z$ of $G_{\bfx_{I}}(Y)=0$ is in $\overline{\Q}$. 
Hence, we obtain that for this $z \in \overline{\Q}$,
\[
F(Y,\bX_{I^c},\bx_I)=(Y-H(\bX_{I^c},\bx_I,z)R(\bX_{I^c},\bx_I,z).
\]
Now the result follows since $H(\bX_{I^c},\bx_I,z)$ and $R(\bX_{I^c},\bx_I,z)\in\overline{\Q}[\bX_{I^c}]$. 

\end{proof}

\begin{rem}\label{remark_G_sequence}
Note in particular that if $F(Y,\bX)$ is not strongly $n$-genuine (respectively, not $n$-genuine), then we may take the non-empty subset $I \subsetneq \{1,\ldots,n\}$ to be the set of variables that $G$ depends on, where $G$ is any polynomial with the property (\ref{G_sequence}) (respectively, with the property (\ref{F_sequence})).
\end{rem}

The previous lemma indicates that the property of having a linear factor in $Y$ over $\bar{\Q}$ will later play a role; we will require the following lemma.
\begin{lem}\label{lemma_F_linear_factor_splits_completely}
    If $F(Y,X_1) \in \Z[Y,X_1]$  has a linear factor in $Y$ over $\overline{\Q}$ and is irreducible over $\Q(X_1)$, then $$F(Y,X_1) = \prod_{j} (Y-Q_j(X_1))$$
    for $Q_j(X_1)\in \overline{\Q}[X_1]$ for all $j$. In other words, $F(Y,X_1)$ splits completely. 
\end{lem}
\begin{proof}
    Since $F(Y,X_1)$ is irreducible over $\Q(X_1)$,  the automorphisms of $\Q(X_1)[Y]/(F(Y,X_1))$  act transitively on the roots of $F(Y,X_1)$. Since $F(Y,X_1)$ has a linear factor over $\overline{\Q}$,   $$F(Y,X_1) = \prod_{\sigma_j} (Y-\sigma(Q(X_1)))$$
    for some polynomial $Q(X_1)\in \overline{\Q}[X_1],$ and $\sigma_j$ varying over 
    the group of embeddings of the extension $\mathbb{Q}(X_{1})[Y]/F(Y,X_{1})$ in its Galois closure.  We observe that $\sigma_j(Q(X_1))\in \overline{\Q}[X_1]$ to complete the proof.
\end{proof}

Being $n$-genuine is preserved by simple scaling transformations, as we will later require.
 \begin{lem}[Preservation under scaling]\label{lemma_preservation_scaling}
 Let $S \in \GL_n(\Q)$ denote a (diagonal) matrix that scales $(x_1,\ldots,x_n)$ to $(s_1x_1,\ldots,s_nx_n)$, with  $s_i\neq 0$ for each $i$. Then $F(Y,\bX)$ is   $n$-genuine (resp.  strongly $n$-genuine) if and only if $F(Y,S\bX)$ is   $n$-genuine (resp.  strongly $n$-genuine). 
\end{lem}
\begin{proof}
    Let $S$ denote a matrix that scales $x_i\mapsto s_i x_i$ for $s_i\neq 0$ for each $i$. Consider the extensions defined by $F(Y,\bX)$ and $F(Y,S\bX)$, namely
    $$\frac{\Q(\bX)[Y]}{F(Y,\bX)} \cong \frac{\Q(\bX)[Y]}{F(Y,S\bX)},$$
  and  these fields are isomorphic under map $S$. Now assume that there exists a $G(Y,\bX)$ such that $G(Y,X_2,...,X_n)$ is a polynomial in the variables $Y,X_2,\ldots,X_n$ and that $G(Y,\bX)$ determines a subfield:
    $$\frac{\Q(\bX)[Y]}{G(Y,\bX)}\subset \frac{\Q(\bX)[Y]}{F(Y,\bX)}.$$
    Then since subfields are preserved under isomorphism, we know that 
    $$\frac{\Q(\bX)[Y]}{G(Y,S\bX)} \subset 
    \frac{\Q(\bX)[Y]}{F(Y,S\bX)}.$$
    However, since $S$ is a scaling matrix, $G(Y,S\bX)$ is also a polynomial in $Y, X_2,...,X_n$. Thus, we know that $F(Y,\bX)$ is not $n$-genuine (resp. not strongly $n$-genuine) if and only if $F(Y,S\bX)$ is not $n$-genuine (resp. not strongly $n$-genuine). 
\end{proof}

\section{$n$-genuine but not strongly $n$-genuine polynomials}\label{sec_not_strongly_allowable}

The main goal of this section is to prove Theorems \ref{thm_allowable_not_strongly} and \ref{thm_genuine_not_strongly}. The key input is the following result:
\begin{thm}\label{thm_genuine_not_strongly_NFB}
Let $R(Y,\bX) \in \Z[Y,X_1,\ldots,X_n]$   be a polynomial that is $n$-genuine   but not strongly $n$-genuine. Then there exists an integer $e=e(n)$ such that for every integer $d\geq 1$ and every $B\geq 1$,
\[ \#\{ \bx \in [-B,B] \cap d^{-1} \Z^n : \exists y \in \Z : R(y,\bx)=0\} \ll_{n,\deg R} \log(\|R\|+2)^{e} (dB)^{n-1}(\log d(B+2))^{2e}.\]
\end{thm}

\begin{proof}[Deduction of Theorem \ref{thm_genuine_not_strongly}]
This immediately implies Theorem \ref{thm_genuine_not_strongly}, by taking $d=1$. 
\end{proof}
 As explained in Remark \ref{remark_gen_deg4} below,   a polynomial $R$ that is $n$-genuine but not strongly $n$-genuine must have $\deg_Y R \geq 4$; similarly for a polynomial that is $(1,n)$-allowable but not strongly $(1,n)$-allowable.
To deduce Theorem \ref{thm_allowable_not_strongly} (and later Theorem \ref{thm_main}), it is useful to record a corollary.
\begin{cor}\label{cor_thm_genuine_not_strongly_NFB}
Let $F \in \Z[Y,X_1,\ldots,X_n]$   be given, along with a transformation $L\in\GL_{n}(\mathbb{Q})$ with all integral entries such that $F_{L}(Y,\bX):=F(Y,L(\bX))\in \mathbb{Z}[Y,\bX]$  is $n$-genuine but not strongly $n$-genuine.  Denote by $\|L\| = \max_{i,j}\|L_{i,j}\|$, where for a rational number in lowest terms $\|a/b\| := \max(|a|,|b|)$. Then there exist integers $e=e(n)$ and $e'=e'(n)$ such that 
\[
N(F,B) \ll_{n,\deg F} \|L\|^{e'} (\log (\|F\|+2))^e B^{n-1} (\log (B+2))^{2e}.
\]
\end{cor}

\begin{proof}[Proof of Corollary \ref{cor_thm_genuine_not_strongly_NFB}]
For simplicity we may assume $\|F\| \geq 3$, $\|F_L\| \geq 3$, $B \geq 3$.   By Theorem \ref{thm_genuine_not_strongly_NFB}, for all integers $d \geq 1$, and for all $B \geq 3$,
\[ \#\{ \bx \in [-B,B] \cap d^{-1} \Z^n : \exists y \in \Z : F(y,L(\bx))=0\} \ll_{n,\deg F} (\log\|F_L\|)^e (dB)^{n-1} (\log dB)^{2e}.\]
The set on the left-hand side identifies with the set on the left-hand side of the equivalent bound
\beq\label{z_box_denoms}
\#\{ \bz \in L([-B,B] \cap d^{-1} \Z^n) : \exists y \in \Z : F(y,\bz)=0\} \ll_{n,\deg F}(\log \|F_L\|)^e (dB)^{n-1}(\log dB)^{2e}.
\eeq
 We will next choose $d$ (and modify $B$) so that the values of $\bz$ considered in this set include all integral points in a box.
\begin{lem}\label{lemma_image_box_denominators}
    Given $L \in \GL_n(\Q)$,
  there exist constants $1 \leq c  \ll_n \|L\|^{n^2}$ and $1\leq d\leq \|L\|^{n^2}$ such that
\[
[-B,B]^n \cap \Z^n \subset L([-c B, c B]^n \cap d^{-1}\Z^n).
\]
\end{lem}
\begin{proof}
It is equivalent to show that there exist such $c,d$ with the property that $L^{-1}([-B,B]^n \cap \Z^n) \subset [-c B, c B]^n \cap d^{-1}\Z^n.$ Write $L^{-1}=(\det L)^{-1}M$ in which $M$ is the adjugate matrix to $L$. By taking $d=\det L$ (so that $1\leq d\leq \|L\|^{n^2}$), it then suffices to find $c$ such that $M([-B,B]^n) \subset [-cB,cB]^n$. Such a $c$ with $1 \leq c  \ll_n \|L\|^{n^2}$ can be obtained by noting that  each entry in the adjugate matrix is $\ll_n \|L\|^{(n-1)^2}.$  
\end{proof}
By the lemma, there exist integers $1 \leq c_L  \ll_n \|L\|^{n^2}$ and $1\leq d_L\leq \|L\|^{n^2}$ such that
\[
[-B,B]^n \cap \Z^n \subset L([-c_L B, c_L B]^n \cap d_L^{-1}\Z^n).
\]
We then learn from applying (\ref{z_box_denoms}) (with $B$ replaced by $c_LB$) that 
\begin{align*}
N(F,B) &\leq \#\{ \bz \in L([-c_LB,c_LB] \cap d_L^{-1} \Z^n) : \exists y \in \Z : F(y,\bz)=0\}\\
&\ll_{n,\deg F} (\log \|F_L\|)^e(c_Ld_LB)^{n-1} (\log (c_Ld_LB))^{2e} \\
&\ll_{n,\deg F} \|L\|^{e'} (\log \|F\|)^e B^{n-1} (\log B)^{2e},
\end{align*}
for some $e'=e'(n)$. Here we have applied that $\log \|F_L\| \ll_{n,\deg F} \log \|L\| + \log \|F\|.$
\end{proof}

\subsection{Deduction of Theorem \ref{thm_allowable_not_strongly}}\label{sec_deduce_thm_allowable_not_strongly}
Under the hypothesis of Theorem \ref{thm_allowable_not_strongly}, a given polynomial $F(Y,\bX)$ is $(1,n)$-allowable but not strongly $(1,n)$-allowable. 
\begin{lem}\label{lemma_make_L_integral}
Let $F\in \mathbb{Z}[Y,X_1,\ldots,X_n]$ be a $(1,n)$-allowable polynomial which is not strongly $(1,n)$-allowable. Then there exists a transformation $L\in\GL_{n}(\mathbb{Q})$ with all integral entries such that $F_{L}(Y,\bX):=F(Y,L(\bX))\in \mathbb{Z}[Y,\bX]$ and $\mathbb{Q}(\bfX)[Y]/(F_{L}(Y,\bfX))$ is $n$-genuine but not strongly $n$-genuine.
\end{lem}

\begin{proof}
  By hypothesis $F$ is not strongly $(1,n)$-allowable, so there exists a transformation $L'\in\GL_{n}(\mathbb{Q})$ such that $F_{L'}(Y,\bX):=F(Y,L'(\bX))$ is $n$-genuine but not strongly $n$-genuine. Let $N$ be a positive integer such that $NL'$ has all integral entries, and let $S=NI_n$ denote the corresponding scaling matrix; define the matrix $L=SL'$. Then by  Lemma \ref{lemma_preservation_scaling}, $F_{L}(Y,\bX)$ is $n$-genuine but not strongly $n$-genuine.
 
\end{proof}
Now, given $F(Y,\bX)$ under the hypothesis of  Theorem \ref{thm_allowable_not_strongly},   we set $L$ to be a transformation as provided by Lemma \ref{lemma_make_L_integral}, so that $F_{L}(Y,\bX):=F(Y,L(\bX))$ is $n$-genuine but not strongly $n$-genuine. An application of Corollary \ref{cor_thm_genuine_not_strongly_NFB} shows that $N(F,B)\ll_{n,\deg F} \|L\|^{e'}(\log \|F\|)^e B^{n-1}(\log B)^{2e}$ for certain $e=e(n),e'=e'(n)$. The factor $\|L\|$ depends on $F$ but in an unspecified way. Thus we conclude that $N(F,B)\ll_{n,F} B^{n-1} (\log B)^{e''}$ for some $e''=e''(n),$
with unspecified dependence on $F$ in the implicit constant. This concludes the proof of Theorem \ref{thm_allowable_not_strongly}.

\subsection{Preliminaries: $n$-genuine polynomials and linear factors}
We turn to proving Theorem \ref{thm_genuine_not_strongly_NFB}.
Our principal results in this section are two-fold. First, we bound a count for the number of specializations of an $n$-genuine polynomial that have a linear factor in $Y$ over $\overline{\Q}$. Precisely, we say that $F(Y,X_1) \in \Q[Y,X_1]$ has a linear factor in $Y$ over $\overline{\Q}$ if 
    \beq\label{dfn_Y_linear_factor} 
    F(Y,X_1) = (Y - Q(X_1)) \Tilde{H}(Y,X_1),
    \eeq
    where $Q(X_1)\in \overline{\Q}[X_1]$ and $\Tilde{H}(Y,X_1)\in \overline{\Q}[Y,X_1].$

\begin{lem}\label{lemma_Cohen_genuine_linear_factor}
    Let  $F(Y,\bX) \in \Z[Y,X_1,\ldots,X_n]$ be an  $n$-genuine polynomial. Then for all $B \geq 1$,   
    $$\#\{\bx' \in [-B,B]^{n-1}: F(Y,X_1,\bx')\textrm{ has a linear factor in $Y$ over }\overline{\Q}\} \ll_{n,\deg F} B^{n-2}.$$
\end{lem}

Second, we make a variable-reducing observation that applies to polynomials that are $n$-genuine but not strongly $n$-genuine. 
\begin{lem}\label{lem_shared_zeros_not_strongly_genuine}
    Let $F(Y,\bX) \in \Z[Y,X_1,\ldots,X_n]$ with $\deg_Y F \geq 2$ be absolutely irreducible, but not a strongly $n$-genuine polynomial.  Then  there exists a polynomial $G(Y,\bX) \in  \Z[Y,\bX]$ depending only $Y, \{X_i\}_{i \in I}$ for an index set $1 \leq |I|<n$, with the following properties:
    \beq\label{QGF} \Q(\bX) \subsetneq \Q(\bX)[Y]/(G(Y,\bX)) \subseteq \Q(\bX)[Y]/(F(Y,\bX));
    \eeq
   $G$ is irreducible over $\Q$,  monic in $Y$,  $\deg_Y G \geq 2$, and   $ 1\leq  \deg_{X_i} G \ll_{n,\deg F} 1$ for each $i \in I$. Moreover,  
    if $\bx\in \Z^n$ satisfies that $F(Y,\bx)=0$ is solvable over $\Z$, then $G(Z,(x_i)_{i \in I})=0$ is also solvable over $\Z$. Additionally, we can choose $G(Y,\bX)=G(Y,\bX_I)$ such that $\log \|G(Y,\bX_I)\|\ll_{n,\deg F} \log \|F(Y,\bX)\|.$

\end{lem}

\begin{rem}\label{remark_gen_deg4}
    In particular, if $F$ is $n$-genuine but not strongly $n$-genuine, then the right-most inclusion in (\ref{QGF}) is strict, and  $2 \leq \deg_Y G$ is a proper divisor of $\deg_Y F$  (and thus such $F$ can only exist if $\deg_Y F\geq 4$).
\end{rem}

\subsection{Preparation to prove Lemma \ref{lemma_Cohen_genuine_linear_factor}}

The following lemma considers the integral closure of $\Q(X_2,\ldots,X_n)$ (omitting the variable $X_1$) in $\Q(\bX)[Y]/F$, and constructs a nonempty open set that will be useful for an application of Noether's lemma. 

\begin{lem}\label{lemma_F_open_set}
 Let  $F(Y,\bX) \in \Z[Y,X_1,\ldots,X_n]$   be absolutely irreducible and monic in $Y$.  Consider the integral closure of $\Q(X_2,...,X_n)$ in $\Q(\bX)[Y]/F$, which we denote by $K_{\hat{1}}.$ Suppose  $G(Y,X_2,...,X_n) \in \Z[Y,X_2,\ldots,X_n]$ is an irreducible polynomial such that $$K_{\hat{1}} = \Q(X_2,...,X_n)[Y]/G(Y,X_2,...,X_n).$$ 
Then there exists a nonempty open set $V \subset \Q^{n-1}$ such that for all $\bx' \in V$, 
 $\Q[Y]/G(Y,\bx')$ is integrally closed in $\Q(X_1)[Y]/(F(Y,X_1,\bx'))$.
\end{lem}

\begin{proof}
If $\deg_{X_1} F = 0$, the lemma is trivially true, since we may simply take $G$ to be the polynomial $F$, and  $\Q(\bX)[Y]/(F(Y,\bX)) = K_{\hat{1}}(X_1)$.
In general, let $$W_F = \bSpec(\Q(\bX)[Y]/F(Y,\bX)), \qquad W_G = \bSpec(\Q(X_2,...,X_n)[Y]/G(Y,X_2,...,X_n)).$$
The inclusion of fields gives us a morphism $$f: W_F\rightarrow W_G.$$
Let $V'$ be the set of points $u\in W_G$ where the fiber $$f_u:W_{F,u}\rightarrow W_{G,u}$$
is geometrically integral. We claim that this contains a nonempty open set. 

By  \cite[EGA IV part 3, Theorem 9.7.7]{EGAIV_part3},  $V'$ is locally constructible. Since $G$ is defined so that $\Q(X_2,...,X_n)[Y]/G(Y,X_2,...,X_n)$ is algebraically closed in $\Q(\bX)[Y]/F(Y,\bX)$, we get that the generic fiber is geometrically integral and therefore contained in $V'$. Hence, $V'$ contains a nonempty open subset $V$ such that for all $\bx' \in V$, $k(W_{G,\bx'})$   is integrally closed in $k(W_{F,\bx'})$. (Here we use the standard notation in algebraic geometry that $k(u)$ is the residue field at $u$.) That is to say, for all $\bx' \in V$, $\Q[Y]/G(Y,\bx')$ is integrally closed in $\Q(X_1)[Y]/(F(Y,X_1,\bx')).$  
\end{proof}

\subsection{Proof of Lemma \ref{lemma_Cohen_genuine_linear_factor}}

We use Noether's Lemma \ref{lemma_Noether} (i)  in combination with Lemma \ref{lemma_F_open_set} and Lemma \ref{lemma_F_linear_factor_splits_completely} to count the number of specializations $\bx'$ for which $F(Y,X_1,\bx')$  has a linear factor in $Y$ over $\overline{\Q}$. 
Let $F$ be given as in Lemma \ref{lemma_Cohen_genuine_linear_factor}, with $\deg_Y F \geq 2$ and $\deg_{X_j}F=k_j$ for $j=1,\ldots,n$. In particular, each $k_j \geq 1$ since $F$ is $n$-genuine.
Define the multi-index $e=(1,k_1)$. 
 Let $G_L$ denote a form provided by Lemma \ref{lemma_Noether} for $K=\Q$, which tests for any $H(Y,X_1) \in \Q[Y,X_1]$ the divisibility condition $\mathcal{D}(e)$ to hold for $H(Y,X_1)$ over $\overline{\Q}$.
 (Lemma \ref{lemma_Noether} provides a collection of forms $\{G_1,...,G_s\}$ which determine when $H(Y,X_1)$ has a linear factor, rather than a single form. We can take $G_L = G_1^{2\ell_1} + ... + G_s^{2\ell_s}$ for appropriate powers $\ell_i$ to obtain a single form of degree $2\ell$ where $\ell$ is the $\mathrm{lcm}$ of the degrees of the $G_i$.)
 Note in particular, if $H(Y,X_1)$ has a linear factor in $Y$ over $\overline{\Q}$ then divisibility condition $\mathcal{D}(e)$ holds for $H(Y,X_1)$ over $\overline{\Q}$.
 
  Suppose now that  $F(Y,\bX) \in \Z[Y,X_1,\ldots,X_n]$ is $n$-genuine.   Denote the coefficients of $F(Y,\bX)$ as a polynomial in $Y,X_1$ by $$F(Y,\bX) = \sum_{i_0, i_1} a_{i_0,i_1}(X_2,...,X_n) Y^{i_0}X_1^{i_1}.$$ 
 Here $a_{i_0,i_1}(X_2,...,X_n) \in \Z[X_2,\ldots,X_n]$; in particular $a_{\deg_Y F,k_j}(\bX')\con 1$ since $F$ is monic in $Y$. 
  Let $G_L((a_{i_0,i_1}(\bX'))) \in \Z[\bX']$ denote the form specified above, so that $G_L((a_{i_0,i_1}(\bx')))=0$ if and only if $F(Y,X_1,\bx')$ satisfies divisibility condition $\mathcal{D}(e)$ over $\overline{\Q}$ (or the degree of $F$, as a polynomial in $Y,X_1$, drops).  In particular, if $F(Y,X_1,\bx')$ has a linear factor in $Y$ over $\overline{\Q}$, then $G_L((a_{i_0,i_1}(\bx')))=0$.
  
     We aim to show that the form $G_L$ is not identically zero as a function of $\bx'=(x_2,\ldots,x_n)$. Once this has been verified, it follows from the trivial bound (Lemma \ref{lemma_Schwartz_Zippel}) that at most $\ll_{\deg G_L} B^{n-2}$ choices of $\bx' \in [-B,B]^n$ lead to $G_L((a_{i_0,i_1}(\bx')))=0$. Since by Lemma \ref{lemma_Noether} (i), $\deg G_L \ll_{n, \deg F} 1$, this suffices for the claim in Lemma \ref{lemma_Cohen_genuine_linear_factor}.
     
   Now we show that $G_L$ is not identically zero. We first consider the integral closure of $\Q(X_2,...,X_n)$ in $\Q(\bX)[Y]/F$, which we denote by $K_{\hat{1}}.$ Then we can find an irreducible polynomial $G(Y,X_2,...,X_n) \in \Z[Y,X_2,\ldots,X_n]$ such that $$K_{\hat{1}} = \Q(X_2,...,X_n)[Y]/G(Y,X_2,...,X_n).$$ 
Since $F(Y,\bX)$ is an $n$-genuine polynomial, we must have $\deg_Y(F) > \deg_Y(G)$, otherwise $G$ would provide a polynomial in $Y$ and $(X_i)_{i \in \{2,\ldots,n\}}$ that violates the definition of $F$ being $n$-genuine.
 By Hilbert's irreducibility theorem (Theorem \ref{thm_HIT}), we can find a dense set $U_{F,G}'\subset \Q^{n-1}$ such that $F(Y,X_1,\bfx')$ and $G(Y,\bx')$ are both irreducible over $\Q$ for all $\bx'\in U_{F,G}'$. 
 Moreover, take the open subset 
\[
V=\{(x_{2},...,x_{n})\in\mathbb{Q}^{n-1}: a_{i_{0},i_{1}}(x_{2},...,x_{n})\neq 0\text{ for all $i_{0},i_{1}$ such that $a_{i_0,i_1}(\bX') \not\con 0$} \}.
\]
Then the set $U_{F,G}=U_{F,G}'\cap V$ is dense in $\Q^{n-1}$. Now assume for contradiction that $G_L=0$ on $U_{F,G}$.
 Under this hypothesis, observe that for every $\bx' \in U_{F,G}$,  $F(Y,X_1,\bx')$ has a linear factor in $Y$ over $\overline{\Q}$ and is irreducible over $\Q(X_1)$: indeed for every $\bfx'\in U_{F,G}$ we have $\deg_{X_{1}}(F(Y,X_{1},\bfx'))=k_{1}$ and $\deg_{Y}(F(Y,X_{1},\bfx'))=\deg_{Y}F$, since $U_{F,G}\subset V$ (so the degree does not drop). We also have by the assumption that $G_{L}=0$ on $U_{F,G}$ that $F=HG$, where $\deg_{Y}(G)\leq 1$. On the other hand, since $F(Y,X_{1},\bfx)$ is monic in $Y$, the only possibility is $\deg_{Y}(G)=1$. 
  
Thus fix any  $\bx' \in U_{F,G}$, and then   $F(Y,X_1,\bx')$ has a linear factor in $Y$ over $\overline{\Q}$ and is irreducible over $\Q(X_1)$.  Then  by Lemma \ref{lemma_F_linear_factor_splits_completely}, $F(Y,X_1,\bx')$ splits completely. Consequently, the splitting field of $F(Y,X_1,\bx')$ over $\Q(X_1)$ is of the form $R_{\bx'}(X_1)$ for $R_{\bx'}$ a finite extension of $\Q$. Hence we can find $K_{\bfx'}\subset R_{\bfx'}$ such that $K_{\bfx'}(X_{1})=\Q(X_1)[Y]/(F(Y,X_1,\bx'))$. Observe that this implies that $[K_{\bx'}:\Q]= \deg_Y F$. Next, we consider $$S_{\bx'}:=\Q[Y]/G(Y,\bx')\subset K_{\bx'}.$$
We know that $[S_{\bx'}:\Q]=\deg_Y(G)<\deg_Y(F) =[K_{\bx'}:\Q]$. So, for the fixed $\bx'\in U_{F,G}$ we considered, $S_{\bx'}$ cannot be integrally closed in $K_{\bx'}$ since $S_{\bfx'}\subset K_{\bfx'}$ are both finite extensions of $\mathbb{Q}$.   
In conclusion, if  $G_L=0$ on $U_{F,G}$, then for all $\bx'\in U_{F,G}$, $S_{\bx'}$ is not integrally closed in $K_{\bx'}$.

On the other hand, Lemma \ref{lemma_F_open_set} provides a nonempty open set $V \subset \Q^{n-1}$ such that for every $\bx' \in V$, $S_{\bx'}=\Q[Y]/G(Y,\bx')$ is integrally closed in $K_{\bfx'}(X_{1}) =\Q(X_1)[Y]/(F(Y,X_1,\bx'))$.  Since $X_1$ is transcendental over $K_{\bx'}$, this implies that $S_{\bx'}$ is integrally closed in $K_{\bx'}$.
  This directly leads to a contradiction, since $U_{F,G}\cap V$ is a nonempty   set, and for any $\bx'\in U_{F,G}\cap V$, it would occur both that $S_{\bx'}$ is not integrally closed in $K_{\bx'}$ and $S_{\bx'}$ is   integrally closed in $K_{\bfx'}$. In conclusion, $G_L$ is not identically zero on $U_{F,G}$, and the lemma is proved.

\subsection{Proof of Lemma \ref{lem_shared_zeros_not_strongly_genuine}}
 We turn to   Lemma \ref{lem_shared_zeros_not_strongly_genuine}, on the  variable-reducing property implied by a polynomial $F(Y,X_1,\ldots,X_n)$ being  not strongly $n$-genuine. Let $E$ denote the total degree of $F(Y,\bX)$, and let $D=\deg_Y F$. Then by definition, there exists some subextension $M_1$ such that 
$$\Q(\bX) \subsetneq M_1\subset \Q(\bX)[Y]/(F(Y,\bX))$$
and $M_1$ is not $n$-genuine. Let us first establish some properties of any such $M_1$. 

Since $M_1$ is not $n$-genuine, we can find a subset $I\subsetneq \{1,\ldots, n\}$ and $G_1(Y,\bX_I) \in \Q[Y,\bX_I]$ such that $M_1 = \Q(\bX)[Y]/(G(Y,\bX_I))$. Note that $I$ is nonempty since the left-most inclusion $\Q(\bX)\subsetneq M_1$ is strict (however, the right-most inclusion need not be strict, since we do not assume that $F$ is $n$-genuine); additionally we can take $|I| = n-1$. For notational simplicity, assume that $I = \{2,\ldots, n\}$ and that $1\not\in I$. Since $G_1(Y,\bX_I)$ is independent of $X_1$, we can write $M_1 = M_{I}(X_1)$, where $M_{I} = \Q(\bX_I)[Y]/(G(Y,\bX_I))$. Furthermore, for any element $S$ of $M_{I}$, we note that the minimal polynomial of $S$ over $\Q(\bX)$ will not be $n$-genuine (as it will be independent of $X_1)$. 

Now, we return to our original extension $\Q(\bX)[Y]/(F(Y,\bX))$ that is not strongly $n$-genuine. Our aim is to construct a polynomial $G(Y,\bX_I)$ such that $\log\|G(Y,\bX_I)\|\ll_{n,E,D} \log\|F(Y,\bX)\|$ and such that $\Q(\bX)\subsetneq M' = \Q(\bX)[Y]/(G(Y,\bX_I)) \subset \Q(\bX)[Y]/(F(Y,\bX))$ (observe that we may not construct the $M_1$ chosen above -- $M'$ may be a different subextension that is also not $n$-genuine.) 

Let $T$ denote a root of $F(Y,\bX)$ such that $\Q(\bX)[Y]/(F(Y,\bX)) = \Q(\bX)(T).$ Let $b(Z) \in \Z[\bX,Z]$ denote a polynomial in $Z$, with coefficients $b_i(\bX) \in \Z[\bX]$ that have total degree at most $B$, of the following form: 
$$b(Z) = b_1(\bX) Z + \hdots + b_{D-1}(\bX) Z^{D-1}.$$
For each $b_i(\bX)$, we can expand the expression as 
$$b_i(\bX) = \sum_{e_1+...+e_n\leq B} b_{i,\bfe} X_1^{e_1}...X_n^{e_n}.$$
We note that for $T$ as chosen above, $\Q(\bX)(b(T)) \subset \Q(\bX)(T)$ for any such polynomial $b(T)$. Moreover, we know that if $b(Z)$ is a nonzero polynomial, then $b(T) \not\in \Q(\bX)$ since $F(Y,\bX)$ is irreducible and no polynomial of degree $<\deg_Y F = D$ in $\Q(\bX)[Y]$ can have $T$ as a root; consequently, if $b(Z)$ is a nontrivial polynomial, then $\Q(\bX)\subsetneq \Q(\bX)(b(T)) \subset \Q(\bX)(T)$. We will show that there exists a polynomial, called $a(Z)\in \Z[\bX,Z]$, with sufficiently small coefficient polynomials $a_1(\bX),...,a_{D-1}(\bX)$ satisfying that $\Q(\bX)(a(T))$ is not $n$-genuine -- in particular, the minimal polynomial of $a(T)$ will be independent of $X_1$. 

Consider for any polynomial $b(Z)\in \Z[\bX,Z]$ the minimal polynomial $m_b(Y,\bX) \in \Z[Y,\bX]$ of the element $b(T)$ over $\Q(\bX)$. If we write $$m_b(Y,\bX) = c_0(\bX) + c_1(\bX) Y + \hdots + c_{\deg_Y m_b}(\bX) Y^{\deg_Y m_b},$$
then each coefficient polynomial $c_{j}(\bX)$ is a symmetric polynomial of degree $j$ in the conjugates of $b(T)$. Note that $\deg_Y m_b \leq D.$ Hence, the coefficients of $c_j(\bX)$ for each $j$ (viewed as a polynomial in $\Z[\bX]$) will be linear combinations of $(b_{i,\bfe})_{i=1,...,D-1, |\bfe|\leq B}$; specifically, if $$c_j(\bX) = \sum c_{j,\mathbf{e}'} X_1^{e_1'}... X_n^{e_n'},$$
then for each $j$ and $\mathbf{e'}$, we can write 
$$c_{j,\bfe'} = \sum_{i=0}^{D-1} \sum_{|\bfe|\leq B} \alpha_{i,j,\bfe',\bfe} b_{i,\bfe},$$
where $\alpha_{i,j,\bfe',\bfe}$ are polynomials, in the coefficients of $F(Y,\bX)$, of degree $\leq j$.

Let us ask for $c_{j,\bfe'}$ to be zero for any $\bfe'=(e_1',\ldots,e_n')$ with $e_1'\geq 1$; in other words, $m_b(Y,\bX)$ is independent of $X_1$. We can see that this defines a linear system of equations in the variables $(b_{i,\bfe})_{i=1,...,D-1, |\bfe|\leq B}$ and the corresponding matrix has entries bounded by $\|F\|^D$. We can now input our assumption that $\Q(\bX)[Y]/(F(Y,\bX))$ is not $n$-genuine and thus a subextension $M_1= M_I(X_1)$ exists, as described at the start of the proof; recall that the minimal polynomial of any element of $M_I$ over $\Q(\bX)$ will be independent of $X_1$. Further, since \begin{align*}   
M_1\subset \Q(\bX)[Y]/(F(Y,\bX)) &= \Q(X_2,...,X_n,Y)[X_1]/(F(Y,\bX))\\
&= \ldots = \Q(X_1,...,X_{n-1},Y)[X_n]/(F(Y,\bX)),
\end{align*}
we can take any such minimal polynomial to have total degree $\leq E$. This implies that for $B:=E$ (as we fix from now on), our linear system of equations (defined above) has a nullspace with rank $\geq 1$, as there is a nontrivial solution given by the minimal polynomial of any element of $M_I \setminus \Q(\bX).$ 

 By Lemma \ref{lemma_matrix_nullspace} (below), we can now construct a nontrivial solution to our linear system of equations, denoted as $(a_{i,\bfe})_{i=1,...,D-1,|\bfe|\leq E}$ satisfying that $\|\bfa\|=\max_{i,\bfe}|a_{i,\bfe}| \ll_{n,E} \|F\|^{C_{n,E}}$, where $C_{n,E}$ is a constant only depending on $n$ and $E$, the total degree of $F$. This set of coefficients then gives us a polynomial $a(Z)\in \Z[\bX,Z]$ satisfying that the minimal polynomial of $a(T)$ over $\Q(\bX)$, denoted as $G(Y,\bX_I) := m_a(Y,\bX_I)$, is independent of $X_1$. Further, since $\|\ba\|\ll_{n,E} \|F\|^{C_{n,E}}$ and the coefficients of $G(Y,\bX_I)$ are given by polynomials of degree $\leq D$ in $(a_{i,\bfe})$ and the coefficients of $F(Y,\bX)$, we know that $\log\|G(Y,\bX_I)\|\ll_{n,\deg F} \log\|F\|.$ Further, since $G(Y,\bX_I)$ is a minimal polynomial, we know that $G$ is irreducible over $\Q$ and $1\leq \deg_{X_i} G \leq E =\deg F$ for each $i$. Additionally, $a(Z)$ is a nonzero polynomial and hence $\Q(\bX) \subsetneq \Q(\bX)(a(T))$, so we know that $2\leq \deg_Y G.$ 

Finally, we must check that if $\bx\in \Z^n$ satisfies that $F(Y,\bx)$ is solvable over $\Z$, then $G(Y,\bX_I)=0$ is also solvable over $\Z$. Let $\sig_i$ be an element of the Galois group of $F(Y,\bX)$ and set $T_i = \sig_i(T)$; it follows that 
$$\sig_i(a(T)) = a_1(\bX)T_i+ \hdots + a_{D-1}(\bX) T_i^{D-1}.$$
Moreover, since $\sig_i$ preserves the elements of $\Q(\bX)$, $\sig_i(a(T))$ is also a root of $G(Y,\bX_I)$. 

Now, let $\bx\in \Z^n$ be such that $F(Y,\bx)=0$ is solvable over $\Z$. Consequently, for some $i\in \{1,\hdots ,D\}$, it must be that $t_i:= T_i \bmod (X_1-x_1,...,X_n-x_n)$ (that is, specializing each $X_i$ to $x_i$) is in $\Z$. For this $i$, $\sig_i(a(T)) \bmod (X_1-x_1,..., X_n-x_n)$ is a root of $G(Y,\bx_I)$ and 
$$\sig_i(a(T)) \bmod (X_1-x_1,...,X_n-x_n) = a_1(\bx) t_i+...+a_{D-1}(\bx)t_i^{D-1} \in \Z.$$
Thus, $G(Y,\bx_I)=0$ is solvable over $\Z$. This completes the proof of Lemma \ref{lem_shared_zeros_not_strongly_genuine}.

\begin{lem}\label{lemma_matrix_nullspace}
Suppose $A$ is an $m_1 \times m_2$ singular   matrix with rank $s < \min\{m_1,m_2\}$ and with each entry $a_{ij}$ an integer satisfying $|a_{ij}| \leq M$. Then there exists a nonzero integral vector $\bfb$ in the nullspace of $A$ with $\|\bfb\| = \max_i |b_i| \ll_s M^{s}$.
\end{lem}

\begin{proof}

 Since $A$ has rank $s$, all minors of size $(s+1)\times (s+1)$ vanish, while there is at least one $s \times s$ minor that is nonzero; we may assume it arises within the top left $(s+1)\times (s+1)$ submatrix of $A$ by omitting the first row and first column. Let $A'$ denote the $m_1 \times (s+1)$ submatrix of $A$ that takes the first $(s+1)$ columns of $A$; by hypothesis every $(s+1)\times(s+1)$ minor of $A'$ vanishes. It suffices to find a nonzero vector $\bfb'$ of length $(s+1)$ in the nullspace of $A'$, and then padding $\bfb'$ with zeroes for the last $m_2-(s+1)$ entries will suffice to construct a nonzero vector $\bfb$ of length $m_2$ such that $A\bfb= \mathbf{0}$.

Define $\bfb'=(b_1,\ldots,b_{s+1})$ so that for each $1 \leq j \leq s+1$, $b_j$ is the cofactor obtained within the top $(s+1) \times (s+1)$ block  of $A'$ by deleting the first row and $j$-th column. By construction, $\|\bfb'\| \ll_s M^s$. Note also that $b_1 \neq 0$, by hypothesis. Now observe that for each row $\mathbf{r}_i$ of the matrix $A'$, 
\[ \mathbf{r}_i \cdot \bfb' = \det \left( \begin{array}{c} \mathbf{r}_i \\ \mathbf{r}_2\\ \vdots \\ \mathbf{r}_{s+1} \end{array} \right)=0
\]
so that $A' \bfb'= \mathbf{0}$. 
Indeed, if $2 \leq i \leq s+1$ this determinant vanishes since the matrix depicted has two identical rows. If $i=1$ or if $i > s+1$, this determinant vanishes since (up to sign) it is an $(s+1)\times (s+1)$ minor of $A'$, which by hypothesis must vanish.

\end{proof}

\subsection{Proof of Theorem  \ref{thm_genuine_not_strongly_NFB}}
 
The main inputs to prove   Theorem \ref{thm_genuine_not_strongly_NFB} are Lemma \ref{lemma_Cohen_genuine_linear_factor},   Lemma \ref{lem_shared_zeros_not_strongly_genuine}, and  Theorem \ref{thm_polylog}, which saves essentially $B^{1/2}$ off the trivial bound for $N(F,B)$.
For simplicity we may assume in what follows that $B \geq 3$.

As in Theorem \ref{thm_genuine_not_strongly_NFB}, suppose that $R(Y,\bX) \in \Z[Y,\bX]$    is an $n$-genuine polynomial but not a strongly $n$-genuine polynomial. By Lemma \ref{lem_shared_zeros_not_strongly_genuine}, there exists a subset $I \subsetneq \{1,\ldots,n\}$ and a  polynomial $G(Y,\bX_I) \in \Z[Y,\bX_I]$  satisfying 
$$\Q(X_1,\ldots,X_n) \subsetneq \Q(\bX)[Y]/(G(Y,\bX_I)) \subsetneq \Q(\bX)[Y]/(R(Y,\bX)),
$$ and such that if $\bx\in \Z^n$ satisfies that there exists $y\in \Z$ such that $R(y,\bx)=0$ then there also exists $z\in \Z$ such that $G(z,\bx_I)=0$. As stated in the lemma, $G$ is irreducible, monic in $Y$, $\deg_Y G \geq 2$, and for each $i \in I$, $1 \leq \deg_{X_i}(G)\leq [\Q(\bX)[Y]/(R(Y,\bX)): \Q(\bX)] \ll_{n,\deg R} 1.$ Additionally, $\log \|G\| \ll_{n,\deg R} \log \|R\|.$ Without loss of generality, we can assume that $I \subset \{2,\ldots,n\}$ so that $G(Y,\bX)$ can be written as $G(Y,X_2,...,X_n) \in \Z[Y,X_2,\ldots,X_n]$ (independent of $X_1$). Let $\bx'$ denote $(x_2,\ldots,x_n)$.

In the context of Theorem \ref{thm_genuine_not_strongly_NFB}, it suffices to bound the count:
\begin{align}
   \sum_{\substack{\bx'\in d^{-1}\Z^{n-1} \\ \bx' \in [-B,B]^{n-1} }}
   & \#\{x_1 \in [-B,B]\cap d^{-1}\Z, \exists y \in \Z, R(y,x_1,\bx')=0\}\nonumber \\
    &=\sum_{\substack{\bx'\in d^{-1}\Z^{n-1} \\ \bx'\in [-B,B]^{n-1} \\ \exists z \in \Z: G(z,\bx')=0}} \#\{x_1 \in [-B,B]\cap d^{-1}\Z, \exists y\in \Z, R(y,x_1,\bx')=0\} \nonumber \\
      &= \sum_{\substack{\bx'\in \Z^{n-1} \\ \bx'\in [-dB,dB]^{n-1}\\ \exists z\in \Z: G_d(z,\bx')=0}} \#\{x_1 \in [-dB,dB]\cap  \Z: \exists y\in \Z, R_d(y,x_1,\bx')=0\},\label{NGR_sum}
      \end{align}
 in which the last identity holds with the definitions
\[R_d(Y,X_1,\bx'):= d^{\deg(R)}R(Y,X_1/d, \bx'/d) \in \Z[Y,X_1], \quad G_d(Z,\bX') := d^{\deg G}G(Y,\bX'/d) \in \Z[Y,\bX'].\]
  In particular, $R_d(Y,X_1,\bx')$ is a polynomial with integral coefficients, and with constant-leading-coefficient in $Y$. (Precisely, since $R$ is monic in $Y$, $R_d$ has leading coefficient $d^{\deg R}$ in $Y$.) Moreover, the polynomial $R_d(Y,X_1,\bX')$ is $n$-genuine but not strongly $n$-genuine, by Lemma \ref{lemma_preservation_scaling}.  For each    $\bx' \in \Z^{n-1}$,  $R_d(Y,X_1,\bx')=0$ defines a curve in $\mathbb{P}^1(\Q)$.
  
 Let $\bx' \in \Z^{n-1}$ be fixed. Observe that since the original $R(Y,X_1,\bx')$ is monic in $Y$, if we decompose it into irreducible factors over $\Q$, say 
\[ R(Y,X_1,\bx') = \prod_i R^{(i)}(Y,X_1,\bx'),\]
then by Gauss's lemma each factor $R^{(i)}(Y,X_1,\bx')$ has integral coefficients. 
For each $i$, define  $R_d^{(i)}(Y,X_1,\bX') := d^{\deg(R^{(i)})}R^{(i)}(Y,X_1/d,\bX'/d)$. Then the above decomposition  into irreducible factors over $\Q$ yields a decomposition 
$$R_d(Y,X_1,\bx') = \prod_{i} R_d^{(i)}(Y,X_1,\bx'),$$
with each $R_d^{(i)} \in \Z[Y,X_1]$ irreducible over $\Q$.  In particular each $R_d^{(i)}$ has constant-leading-coefficient in $Y$,  and there are at most $\deg R$ terms in the indexing set.
If  for at least one index $i$, $R_d^{(i)}(Y,X_1,\bx')$ is linear in $Y$, we will say below that ``$R_d(Y,X_1,\bx')$ has a linear factor.''
In any case, by the factorization, for each fixed $\bx' \in \Z^{n-1}$, 
\begin{multline*}
\#\{x_1 \in [-dB,dB]\cap  \Z: \exists y \in \Z, R_d(y,x_1,\bx')=0\}\\
\leq \sum_{i} \#\{x_1 \in [-dB,dB]\cap  \Z: \exists y \in \Z, R_d^{(i)}(y,x_1,\bx')=0\}.
\end{multline*}
So we proceed to estimate the size of this set for each irreducible factor $R_d^{(i)}(Y,X_1,\bx')$  over $\Q$. 

First, suppose that $\deg_Y(R_d^{(i)}(Y,X_1,\bx'))=0$ so it can be written as $R_d^{(i)}(X_1,\bx')$. In this case 
\[\#\{x_1 \in [-dB,dB]\cap  \Z: R_d^{(i)}(x_1,\bx')=0\} \ll \deg(R).\]
Second, suppose that  $R_d^{(i)}(Y,X_1,\bx')$ is linear in $Y$.  Then the best bound available is the trivial
\beq\label{bad_linear_bound}
\#\{x_1 \in [-dB,dB]\cap  \Z: \exists y \in \Z, R_d^{(i)}(y,x_1,\bx')=0\} \ll  dB.
\eeq
Finally, suppose that
$\deg_{Y}(R_d^{(i)}(Y,X_1,\bx'))\geq 2$.    We now apply 
Theorem \ref{thm_polylog} to $R_d^{(i)}(Y,X_1,\bx')$ in the variables $Y,X_1$, which has a constant-leading-coefficient in $Y$. This yields:
\begin{multline*}\#\{x_1 \in [-dB,dB]\cap  \Z: \exists y \in \Z, R_d^{(i)}(y,x_1,\bx')=0\} \\ \ll_{\deg R} (\log(\|R_d^{(i)}(Y,X_1,\bx')\|+2))^{e(1)}(dB)^{1/2} ( \log (dB))^{e(1)}\\
\ll_{\deg R} (\log(\|R\|+2))^{e(1)} (dB)^{1/2}(\log(dB))^{2e(1)}.
\end{multline*}
for a constant $e(1)$.  Here, we have applied the fact that for  $\bx'\in [-dB,dB]^{n-1} \cap \Z^{n-1}$, 
\[\|R_{d}^{(i)}(Y,X_1,\bx')\|\ll_{\deg R} d^{\deg(R^{(i)})}\cdot \|R\|\cdot (dB/d)^{\deg(R^{(i)})}
\ll \|R\|(dB)^{\deg(R^{(i)})}.\]
 
 Combining the above observations shows that (\ref{NGR_sum}) is bounded by:
\begin{multline*}
\ll_{\deg R}  \log(\|R\|+2)^{e(1)} (dB)^{1/2}(\log (dB))^{2e(1)} \#\{\bx'\in [-dB,dB]^{n-1} \cap   \Z^{n-1} : \exists z\in \Z: G_d(z,\bx')=0\}\\
 + dB\cdot  \#\{\bx'\in [-dB,dB]^{n-1} \cap \Z^{n-1}: R_d(Y,X_1,\bx') \text{ has a linear factor}\}.
\end{multline*}
The second term is bounded by 
\[dB\cdot  \#\{\bx'\in [-dB,dB]^{n-1} \cap \Z^{n-1}: R_d(Y,X_1,\bx') \text{ has a linear factor}\}\ll_{n,\deg R} dB \cdot (dB)^{n-2},\]
by Lemma \ref{lemma_Cohen_genuine_linear_factor}, and this suffices.
To bound the first term, recall the definition $G_d(Z,\bX') := d^{\deg G} G(Z,\bX'/d).$  Note that $G_d(Z,\bX')$ has integral coefficients and a constant-leading-coefficient in $Z$ (that is a power of $d$).
We apply Theorem \ref{thm_polylog} again, now   to $G_d(Z,\bX')$ in variables $Z,X_2,\ldots,X_n$. Thus an application of  Theorem \ref{thm_polylog} shows
\begin{multline*}
    \#\{\bx'\in [-dB,dB]^{n-1} \cap   \Z^{n-1} : \exists z\in \Z: G_d(z,\bx')=0\}\\
\ll_{n,\deg G} \log(\|G_d(Z,\bX')\|+2)^{e(n-1)}(dB)^{(n-1)-1/2}(\log (dB))^{e(n-1)}
\end{multline*}
for all $dB \geq 3$, for  a certain power $e(n-1)$.  
Finally, since $\log \|G\| \ll_{n,\deg R}\|R\|$ by Lemma \ref{lem_shared_zeros_not_strongly_genuine}, 
\[
\log \|G_d(Z,\bX')\| \ll \log (  d^{\deg G} \|G\|)\ll_{n,\deg R} \log (d\|R\|).
\]
We summarize by concluding 
$$\#\{\bx \in [-B,B]^n \cap d^{-1}\Z^n, \exists y\in \Z, R(y,\bx')=0\}\ll_{n,\deg R}\log(\|R\|+2)^{e(n)} (dB)^{n-1} (\log dB)^{2e(n)},$$
 with the constant $e(n)$ depending only on $n$.
This completes the proof of Theorem \ref{thm_genuine_not_strongly_NFB}.

\section{Strongly $n$-genuine polynomials: second moment bound}\label{sec_strongly}
The main goal of this section is to prove   a second moment bound we will apply when treating certain terms in the polynomial sieve, namely terms that arise from irreducible components in the stratification that are degenerate (that is, lie in a proper linear subspace, see \S \ref{sec_degenerate_primes}). In full generality, this relies on the complete theory of strongly $n$-genuine polynomials. In order to make the ideas of the proof clear, we first briefly illustrate the core strategy in the setting of Theorem \ref{thm_cyclic_uniform}, which considers a special class of strongly $(1,n)$-allowable polynomials that exhibit cyclic structure. With Proposition \ref{prop_cyclic_second_moment} (below) in hand, a reader could directly prove Theorem \ref{thm_cyclic_uniform} via the ``power sieve'' for multiplicative characters as formulated in \cite{Mun09,HBPie12}, without the full theory of strongly $n$-genuine polynomials, which is much more technical. We handle the general case of the second moment bound in \S \ref{sec_second_moment_general}, which is applied to prove  Theorem \ref{thm_explicit} in full, of which Theorem \ref{thm_cyclic_uniform} is a special case.

\subsection{Illustrative special case: cyclic structure}

\begin{prop}[Second moment bound: cyclic case]\label{prop_cyclic_second_moment}
Let $H(\bX)\in \Z[X_1,\hdots,X_n]$ be an absolutely irreducible polynomial such that $\deg_{X_1}(H)\geq 1$. Let $d\geq 2$ and $\chi$ denote a nonprincipal character of $\F_p^\times$ of order $d$. Define 
$$S(\bu,p) = \sum_{\bx\in \F_p^n}\chi(H(\bx)) e_p(\bu\cdot \bx).$$
Then there exists a finite set of exceptional primes $\calE$ satisfying $|\calE|\ll_{n,\deg H} \log\|H\|/\log\log\|H\|$ such that for all $p\not\in \calE,$
$$\sum_{\bu\in \{u_1=0\}\subset \F_p^n} |S(\bu,p)|^2 \ll_{n,d,\deg H} p^{2n-1}.$$

\end{prop}

 \begin{cor}
Under the hypotheses of the proposition, suppose furthermore that there is no $L\in \GL_n(\Q)$ such that $H(L(\bX))\in \Z[(X_i)_{i\in I}]$ for a subset $I\subsetneq \{1,...,n\}.$ Then for any hyperplane $W = \{\bX: W \cdot \bX=0\}$ (defined by a fixed $W \in \Z^{n}$) there exists a finite set of exceptional primes $\calE$ satisfying $|\calE|\ll_{n,\deg H} \log (\|H\|\|W\|)/\log\log(\|H\|\|W\|)$ such that for all $p\not\in \calE,$
$$\sum_{\bu\in W\subset \F_p^n} |S(\bu,p)|^2 \ll_{n,d,\deg H} p^{2n-1}.$$
\end{cor}
The corollary follows by applying the proposition to $H_{L}(\bX):=H(L\bX)$, where $L \in \GL_n(\Q)$ is a linear transformation with all integral entries such that $\{(L^T\bu)_1 =0\} = W$. By the hypothesis of the corollary, $\deg_{X_1} H_L \geq 1$, as required.  (We may assure that $L$ is invertible over $\F_p$ for each $p \not\in \mathcal{E}$ by ensuring that $\mathcal{E}$ includes all $p | \det L$, of which there are $\ll_{n} \log \|W\| / \log \log \|W\|$.)

\begin{proof}[Proof of Proposition \ref{prop_cyclic_second_moment}]
Suppose a prime $p$ is fixed.    We start by opening up the square, getting 
    \begin{align*}
        \sum_{\bu\in \{u_1=0\}\subset \F_p^n} |S(\bu,p)|^2 &= \sum_{\bx,\by\in \F_p^n}\chi(H(\bx))\overline{\chi(H(\by))} \sum_{\bu\in \{u_1=0\}} e_p(\bu\cdot (\bx-\by)) \\
        &= p^{n-1} \sum_{\substack{\bx,\by\in \F_p^{n} \\ \bx-\by \in \{u_1=0\}^\perp}} \chi(H(\bx))\overline{\chi(H(\by))}.
    \end{align*}
    Since $\{u_1=0\}^\perp = \{u_2=\hdots = u_n=0\}$, we can rewrite 
    \begin{align*}
        \sum_{\bu\in \{u_1=0\}\subset \F_p^n} |S(\bu,p)|^2 &= p^{n-1} \sum_{\bt\in \F_p^{n-1}} \sum_{x,y\in \F_p} \chi(H(x,\bt))\overline{\chi(H(y,\bt))} \\
        &= p^{n-1} \sum_{\bt\in \F_p^{n-1}} \left|\sum_{x\in \F_p} \chi(H(x,\bt))\right|^2.
    \end{align*}
  By the Weil bound as in \cite[Thm. 11.23]{IK04},
   \[
    \left|\sum_{x\in\mathbb{F}_{p}}\chi(H(x,\bft))\right|=\begin{cases}
        O_{\deg H}(\sqrt{p})     &\text{if $H(X,\bft)$ is not a $d$th-power over $\F_p$,}\\
        p+O_{\deg H}(1)   &\text{otherwise.}
    \end{cases}
    \]
    (The latter case follows since if $H(X,\bft)=(R_{\bft}(X))^d$ over $\F_p$, say, then the sum evaluates to $p - \#\{x \in \F_p: R_{\bft}(x)=0\}$.)
Thus in order to complete the proof of the proposition it suffices to construct an appropriately small set $\mathcal{E}$ such that  for all $p\not\in \calE$, the polynomial $H(X, \bft)$ is not a $d$th-power over $\F_p$ for all but $O_{n,d,\deg H}(p^{n-2})$ choices of $\bft\in \F_p^{n-1}.$

    Write the polynomial $$H(\bX) = f_0(X_2,\hdots,X_n) X_1^k + f_1(X_2,\hdots, X_n)X_1^{k-1} +\hdots + f_k(X_2,\hdots, X_n).$$
    Since $k\geq 1$ by the hypothesis of the proposition, there is a nonzero polynomial $\Delta(X_2,\hdots,X_n)$ with $\deg \Delta \ll_{n,\deg H} 1$ (the discriminant polynomial applied at the coefficient polynomials $f_0,...,f_k$) satisfying  
    $$H(X,\bft) \textrm{ is square-free}\iff \Delta(\bft)\neq 0.$$
    Moreover, the above relation also holds over $\F_p$. Observe that if $H(X,\bft)$ is a $d$th-power, then it is not square-free. So it suffices to show that $p \ndiv \Delta(\bft)$ for all but $O_{n,d,\deg H}(p^{n-2})$ choices of $\bft \in \F_p^{n-1}$, as long as $p$ is not in an appropriately small exceptional set $\mathcal{E}$.

    First, we establish that $\Delta(X_2,\hdots,X_n)\not\equiv 0$ as a polynomial over $\Z$. Since we are assuming that $H(\bX)$ is irreducible over $\Q$, the Hilbert Irreducibility Theorem \ref{thm_HIT} guarantees that we can find $t_2,\hdots,t_{n}\in \Z$ where $H(X_1,t_2,\hdots,t_n)$ is irreducible over $\Q$. For this choice of $(t_2,\hdots,t_n)\in \Z^{n-1}$, $\Delta(t_2,\hdots, t_n)\neq 0$ and hence $\Delta(X_2,\hdots,X_n)\not\equiv 0.$ With this fact in hand, as a result of the trivial Schwartz-Zippel bound (Lemma \ref{lemma_Schwartz_Zippel}), we can moreover choose a particular tuple $\bft$ with $\Delta(t_2,\hdots, t_n)\neq 0$ and $\|\bft\|\ll_{\deg \Delta} 1 \ll_{n,\deg H} 1.$ Given this tuple, define the exceptional set $\calE:=\{p: p\mid \Delta(t_2,\hdots,t_n)\}$. By the previous bound on $\bft$, we know that $\log|\Delta(\bft)|\ll_{n,\deg H} \log\|H\|.$ Thus, $|\calE|\ll_{n,\deg H} \log\|H\|/\log\log\|H\|$. Moreover, for each $p \not\in \mathcal{E}$, the property $p \ndiv \Delta(\bft)$ shows that $\Delta(X_2,\ldots,X_n) \not\con 0 \modd{p}$, so that a second application of the trivial   bound (Lemma \ref{lemma_Schwartz_Zippel}) shows $\#\{\bx' \in \F_p^{n-1}: \Delta(\bx')=0\} \ll_{\deg H} p^{n-2}$, as desired.
\end{proof}

\subsection{General case: strongly $n$-genuine polynomials}\label{sec_second_moment_general}

The notion of a strongly $n$-genuine polynomial plays a key role in our strategy to prove a second moment bound in full generality. 
We first state and prove a second moment bound over the hyperplane $\{u_1=0\}$, and then deduce from it a second moment bound over any hyperplane. Each of these bounds requires some additional hypothesis on $F(Y,\bX)$. 

In preparation, we recall from Lemma \ref{lemma_strongly_n_gen_equiv_alg_closed} that  $F$ is a strongly $n$-genuine polynomial if and only if for every $j\in\{1,\ldots,n\}$,
       \beq\label{alg_closed_assumption}
       \overline{\mathbb{Q}(X_{i})_{i\neq j}}\cap \mathcal{L}_F=\mathbb{Q}(X_{i})_{i\neq j},
       \eeq
       in which $\mathcal{L}_F=\Q(X_1,\ldots,X_n)[Y]/F$.
 For example, for $j=1$,  (\ref{alg_closed_assumption}) is the statement that $\Q(X_{2},...,X_{n})$ is integrally closed in $\Q(X_{1},...,X_{n})[Y]/F$; equivalently, if $\theta\in \mathcal{L}_F$ and $\theta$ is algebraic over $\Q(X_{2},...,X_{n})$, then $\theta\in \Q(X_{2},...,X_{n})$; equivalently, for any algebraic extension $K'$ of $\Q(X_2,...,X_n)$, $K'\cap \mathcal{L_F} = \Q(X_2,...,X_n)$. The next second moment bound, tailored to the hyperplane $\{u_1=0\}$, is proved under the hypothesis that (\ref{alg_closed_assumption}) holds for $j=1$. 
   
\begin{prop}[Second moment bound]\label{prop_strongly_genuine_second_moment}
Let $F(Y,\bX) \in \Z[Y,X_1,\ldots,X_n]$ with $\deg_Y F \geq 2$ and $F$ monic in $Y$ be absolutely irreducible. Suppose that  \[
       \overline{\mathbb{Q}(X_2,\ldots,X_n)}\cap \mathcal{L}_F=\mathbb{Q}(X_2,\ldots,X_n),
       \]
       in which $\mathcal{L}_F=\Q(X_1,\ldots,X_n)[Y]/F$.
Define for a given prime $p$,
\[ S(\bu,p) = \sum_{\bx \in \F_p^n} (v_p(\bx)-1)e_p(\bu \cdot \bx)\]
where $v_p(\bx) = \# \{ y \in \F_p : F(y,\bx)=0\}$. Then there exists a finite set of exceptional primes $\calE$ satisfying $|\calE| \ll_{n,\deg F}\log\|F\|/\log\log\|F\|$ such that for all $p\not\in \calE$, 
    $$\sum_{\bu \in \{u_1=0\}\subset \F_p^n} |S(\bu,p)|^2 \ll_{n,\deg F} p^{2n-1}.$$
    (In particular, the conclusion holds if $F$ is strongly $n$-genuine.)
\end{prop}

This proposition is the main result of the section, and to prepare for its proof we gather several key lemmas.

 \begin{lem}\label{lemma_F_strongly_n_geniune_open_V}
    Let $F(Y,\bX) \in \Z[Y,X_1,\ldots,X_n]$ with $\deg_Y F \geq 2$ be absolutely irreducible. Suppose that  
    \beq\label{int_closed_open_V}
       \overline{\mathbb{Q}(X_2,\ldots,X_n)}\cap \mathcal{L}_F=\mathbb{Q}(X_2,\ldots,X_n),
       \eeq
       in which $\mathcal{L}_F=\Q(X_1,\ldots,X_n)[Y]/F$. Then, there exists a nonempty open set $V\subset \Q^{n-1}$ such that for all $\bx'\in V$, $\Q$ is integrally closed in $\Q(X_{1})[Y]/F(Y,X_{1},\bfx ')$. (In particular, the conclusion holds if $F$ is strongly $n$-genuine.)
    \end{lem}
    We remark that the condition of $F$ being absolutely irreducible is actually implied by the fact that we are assuming
    \[
       \overline{\mathbb{Q}(X_2,\ldots,X_n)}\cap \mathcal{L}_F=\mathbb{Q}(X_2,\ldots,X_n),
       \]
       Indeed, if $F$ is irreducible over $\Q$ but not absolutely irreducible, then we can find a polynomial $H$ with coefficients in $\overline{\mathbb{Q}}$ which is a proper divisor of $F$ and is a minimal polynomial for some root of $F$. Denote by $K$ the finite extension of $\mathbb{Q}$ which contains all coefficients of $H$. Now let $\eta$ be a root of $H$ (and hence a root of $F$), $\Omega$ be the splitting field of $F$, and consider $\tau\in\text{Gal}(\Omega/\mathbb{Q}(\bfX))$ such that $\tau(\eta)=\eta$; then $\tau$ needs to fix also the coefficients of $H$ since $H$ is the minimal polynomial of $\eta$ over $\overline{\mathbb{Q}}(\bfX)$. Hence all the coefficients of $H$ are contained in $\mathcal{L}_{F}$, i.e. $K\subset\mathcal{L}_{F}$ which implies
       \[
       \overline{\mathbb{Q}(X_2,\ldots,X_n)}\cap \mathcal{L}_F\supset K(X_2,\ldots,X_n),
       \]
       a contradiction unless $H$ has coefficients in $\Q$, which cannot occur since $F$ is irreducible over $\Q$.
    \begin{proof}
    Consider $W=\text{Spec}\left(\Q
    [Y,X_{1},...,X_{n}]/F\right)$ and the morphism induced by $\Q[X_{2},....,X_{n}]\subset \Q[Y,X_{1},...,X_{n}]/F$
    \[
    f:W\longrightarrow\mathbb{A}_{\Q}^{n-1}.
    \]
    Let $V'\subset\mathbb{A}_{\Q}^{n-1}$ be the set of  $u\in\mathbb{A}_{\Q}^{n-1}$ such that the fiber
    \[
    f_{u}:W_{u}\longrightarrow\mathbb{A}_{k(u)}^{n-1}
    \]
    is geometrically integral (i.e. $k(u)$ is integrally closed in $K(W_{u})$). (Here we again use the standard notation in algebraic geometry that $k(u)$ is the residue field at $u$ and $K(W_{u})$ is the field of rational functions on $W_{u}$.) By  \cite[EGA IV part 3, Theorem $9.7.7$]{EGAIV_part3},   the set $V'$ is locally constructible. On the other hand, the fiber on the generic point $\eta$, say $f_{\eta}$, is geometrically integral; this follows from the hypothesis (\ref{int_closed_open_V}) that $\Q(X_{2},...,X_{n})$ is integrally closed in $\Q(X_{1},...,X_{n})[Y]/F=K(W)$.   Consequently, the generic point $\eta$ lies in  $V'$. Hence, since it contains the generic point, $V'$ contains an open subset $U'$: indeed, since $V'$ is locally constructible we can find an open covering $\mathbb{A}_{\Q}^{n-1}=\bigcup_{i}V_{i}$ such that for each $i$, $V_{i}\cap V'$ is a constructible set, i.e. it is a finite union of sets of the type $S\cap T^{c}$ for $T,S$ open sets. Since $\eta\in V'$, we can find $i$ such that $\eta\in V_{i}\cap V'$, but then we can find $S,T$ open sets such that $\eta\in S\cap T^{c}$. On the other hand, since $\{\eta\}$ is dense in $\mathbb{A}_{\Q}^{n-1}$, it follows that $T=\emptyset$. Then it suffices to set $U' = S \subset (V_i \cap V') \subset V'$. Finally, defining the open subset $V=U'\cap \Q^{n-1}$ in $\Q^{n-1}$, for every $(x_{2},...,x_{n})\in V$,  we have that $k(u)=\Q$ is algebraically closed in $K(W_{u})=\text{Frac}\left(\Q[Y,X_{1},x_{2},...,x_{n}]/F(Y,X_{1},x_{2},...,x_{n})\right)$. 
\end{proof}
 
Next, we apply Noether's Lemma \ref{lemma_Noether} (ii) in combination with Lemma \ref{lemma_F_strongly_n_geniune_open_V},  to control how often $F(Y,X_1,x_2,...,x_n)$ is reducible, still under the hypothesis (\ref{alg_closed_assumption}) for $j=1$. The following is an alternative to \cite[Lemma 4.2(ii)]{Coh81} whose proof contains a gap (see \cite{BPW25x_per} for details, as well as a full correction of the gap).
\begin{lem}\label{lemma_Cohen_strongly_genuine_reducible}
    Let $F(Y,\bX) \in \Z[Y,X_1,\ldots,X_n]$ with $\deg_Y F \geq 2$ be absolutely irreducible.  Suppose that  \[
       \overline{\mathbb{Q}(X_2,\ldots,X_n)}\cap \mathcal{L}_F=\mathbb{Q}(X_2,\ldots,X_n),
       \]
       in which $\mathcal{L}_F=\Q(X_1,\ldots,X_n)[Y]/F$. Then there exists a finite set of exceptional primes $\calE$ satisfying $|\calE|\ll_{n,\deg F} \log\|F\|/\log\log\|F\|$ such that for all $p\not\in \calE$, 
$$\#\{(x_2,...,x_n)\in \F_p^{n-1}: F(Y,X_1,x_2,...,x_n) \textrm{ is reducible over }\overline{\F_p}\} \ll_{n,\deg F} p^{n-2}.$$
(In particular, the conclusion holds if $F$ is strongly $n$-genuine.)
\end{lem}

\begin{proof}

     By Hilbert's irreducibility theorem (Theorem \ref{thm_HIT}), since $F(Y,\bX)$ is irreducible over $\Q$, there exists a dense set $U' \subset \Q^{n-1}$    such that 
    \beq\label{conclusion_F_irred_on_U}
    F(Y,X_1,x_2,...,x_n)\text{ is irreducible over }\Q, \forall (x_2,...,x_n)\in U'.
    \eeq
 Denote the coefficients of $F(Y,\bX)$ as a polynomial in $Y,X_1$ by $$F(Y,\bX) = \sum_{i_0, i_1} a_{i_0,i_1}(X_2,...,X_n) Y^{i_0}X_1^{i_1}.$$
 Take the open subset 
\[
V=\{(x_{2},...,x_{n})\in\mathbb{Q}^{n-1}: a_{i_{0},i_{1}}(x_{2},...,x_{n})\neq 0\text{ for all $i_{0},i_{1}$ such that $a_{i_0,i_1}(\bX') \not\con 0$} \}.
\]
(In particular, for all $\bx' \in V,$ $\deg_{Y}F(Y,X_1,\bx')=\deg_{Y}F(Y,X_1,\bX')$ and $\deg_{X_1}F(Y,X_1,\bx')=\deg_{X_1}F(Y,X_1,\bX')$.)
Then the set $U=U'\cap V$ is dense in $\Q^{n-1}$. 

Let $G_R((a_{i_0,i_1}(\bX')))$ denote a form provided by Noether's Lemma \ref{lemma_Noether} (ii) for polynomials in $Y,X_1$ of degree $\leq \max\{\deg_Y F, \deg_{X_1}F\}$ over $K=\Q$   and the divisibility condition ``reducible'', so that $G_R((a_{i_0,i_1}(\bx')))=0$ if and only if $F(Y,X_1,\bx')$ is reducible over $\overline{\Q}$ (or the degree of $F$, as a polynomial in $Y,X_1$, drops). 
  (Lemma \ref{lemma_Noether} (ii) provides a collection of forms $\{G_1,...,G_s\}$ which determine when $F(Y,\bX)$ is reducible, rather than a single form. We can, however,  take $G_R = G_1^{2 \ell_1} + \cdots + G_s^{2\ell_s}$ for appropriate $\ell_i \geq 1$ to obtain a single form whose degree is $2r$, where $r$ is the $\mathrm{lcm}$ of the degrees of $G_1,\ldots,G_s$.)

    We will use Lemma \ref{lemma_F_strongly_n_geniune_open_V} to show that there exists $\bx'\in U$ such that $G_R((a_{i_0,i_1}(\bx'))) \neq 0$.  
  Assume for contradiction that $G_R = 0$ on $U$. Then $F(Y,X_1,\bx')$ is reducible over $\overline{\Q}$  for all $\bx'\in U$ (since by construction the degree of $F$ does not drop, as a polynomial of $Y,X_1$, on $U$).  However, by the construction in (\ref{conclusion_F_irred_on_U}), for each $\bx' \in U$, $F(Y,X_1,\bx')$ is irreducible over $\Q$. So, for all $\bx'  \in U$, $\Q$ cannot be integrally closed in $\Q(X_{1})[Y]/F(Y,X_{1},\bfx ')$. 
  On the other hand, by Lemma \ref{lemma_F_strongly_n_geniune_open_V}, there exists a nonempty open set $V \subset \Q^{n-1}$ such that for all $\bx' \in V$, $\Q$ is integrally closed in $\Q(X_{1})[Y]/F(Y,X_{1},\bfx ')$. Since $U$ is dense and $V$ is open, $U\cap V \neq \emptyset.$ Thus, we get a contradiction: if an element $\bx'$  exists in $ U\cap V$, $\Q$ would be both integrally closed and not integrally closed in $\Q(X_{1})[Y]/F(Y,X_{1},\bfx ')$. Thus the supposition that $G_R = 0$ on $U$ must be false.
  
 Now we can complete the proof of the lemma. Since there exists a choice of $\bx'\in U$ such that $G_R((a_{i_0,i_1}(\bx'))) \neq 0$, then $G_R((a_{i_0,i_1}(\bX')))$ is a polynomial in $\Z[\bX']$ that is not identically zero. Consequently, for all rational primes $p$ that do not divide the gcd of the coefficients of $G_R((a_{i_0,i_1}(\bX')))$, the reduction of $G_R((a_{i_0,i_1}(\bX')))$ modulo $p$ is a   polynomial in $\F_p[\bX']$, not identically zero. Denote this gcd temporarily by $g$,  and denote by $\|G_R\|$ the maximum absolute value of the coefficient of any monomial in $\bX'$ in the polynomial $G_R((a_{i_0,i_1}(\bX')))$ (which will depend polynomially on the fixed coefficients defining $F$). Note that $g \ll \|G_R\|$. By Lemma \ref{lemma_Noether}, $\log \|G_R\| \ll_{n,\deg F} \log \|F\|$. Upon defining the exceptional set $\mathcal{E}$ to be all $p|g$, it follows that $|\mathcal{E}| \ll_{n,\deg F} \log \|F\| / \log \log \|F\|$, as claimed.  Now suppose $p \not\in \mathcal{E}$ is fixed. Again by Lemma \ref{lemma_Noether} (ii), for a given tuple $(x_2,\ldots,x_n) \in \F_p^{n-1}$, $F(Y,X_1,x_2,\ldots,x_n)$ is reducible over $\overline{\F_p}$ if and only if $G_R((a_{i_0,i_1}(\bx')))=0$ in $\F_p$.  Since this is not the zero polynomial in $\F_p$, the trivial bound (Lemma \ref{lemma_Schwartz_Zippel}) shows that the number of $\bx' \in \F_p^{n-1}$ that satisfy $G_R((a_{i_0,i_1}(\bx'))) \neq 0$ is $\ll \deg_{G_R} p^{n-2}$. By Lemma \ref{lemma_Noether} (ii), $\deg G_R \ll_{n,\deg F} 1$, so that this is sufficient for the claim of Lemma \ref{lemma_Cohen_strongly_genuine_reducible}.

\end{proof}

\subsection{Proof of Proposition \ref{prop_strongly_genuine_second_moment}}
The main input   is Lemma \ref{lemma_Cohen_strongly_genuine_reducible}.
Expanding the square,
    \begin{align*}
        \sum_{\bu\in \{u_1=0\} \subset \F_p^n} |S(\bu,p)|^2 &= \sum_{\bx,\bx'\in \F_p^n} (v_p(\bx)-1)(v_p(\bx')-1) \sum_{\bu\in \{u_1=0\}\subset \F_p^n} e(\bu\cdot (\bx-\bx'))\\
        &= p^{n-1} \sum_{\ba \in \F_p^{n-1}} \left|\sum_{x_1\in \F_p} (v_p((x_1,\ba))-1)\right|^2\\
        &= p^{n-1} \sum_{\ba\in \F_p^{n-1}}\left|\#\{(y,x_1)\in \F_p^2: F(y,x_1,\ba)=0\} - p\right|^2.
    \end{align*}
Under the assumption that $F(Y,\bX)$ is monic in $Y$, note that   for all $p$ and all $\ba \in \F_p^{n-1}$, $F(Y,X_1,\ba) \not\con 0$ over $\F_p$.   By the Lang-Weil bound (Lemma \ref{lemma_Lang_Weil_cor}), if $F(Y,X_1,\ba) \in \F_p[Y,X_1]$ is  irreducible over $\overline{\F_p}$, then 
    $$|\#\{(y,x_1)\in \F_p^2: F(y,x_1,\ba)=0\} - p|\ll_{n,\deg F} p^{1/2}.$$
    On the other hand, if $F(Y,X_1,\ba)$ is not  irreducible over $\overline{\F_p}$, the best bound available is 
    $$|\#\{(y,x_1)\in \F_p^2: F(y,x_1,\ba)=0\} - p|\ll_{\deg F} p,$$
by Lemma \ref{lemma_Schwartz_Zippel}.
    By Lemma \ref{lemma_Cohen_strongly_genuine_reducible},   there exists a set of exceptional primes $\calE$ satisfying $|\calE|\ll \log\|F\|/\log\log\|F\|$ such that for all $p\not\in \calE$,
    $$\#\{\ba\in \F_p^{n-1}: F(Y,X_1,\ba) \textrm{ is reducible over }\overline{\F_p}\} \ll_{n,\deg F} p^{n-2}.$$
    Thus, for each such $p \not\in \mathcal{E}$,
    $$\sum_{\bu\in \{u_1=0\}} |S(\bu,p)|^2 \ll_{n,\deg F} p^{n-1} (\sum_{\ba \in \F_p^{n-1}} O(p^{1/2})^2 + \sum_{\substack{\ba \in \F_p^{n-1}\\ F(Y,X_1,\ba) \textrm{ reducible/}\overline{\F_p}}} O(p)^2 )\ll_{n,\deg F} p^{2n-1},$$
    and the proposition is proved.

\subsection{Linear changes of variable and second moment bounds}
Proposition \ref{prop_strongly_genuine_second_moment} proves a second moment bound when $|S(\bu,p)|^2$ is summed over $\bu$ lying in the hyperplane $\{u_1=0\}$. We will ultimately require a more general result that sums over $\bu$ lying in a fixed arbitrary hyperplane $V \subset \mathbb{A}^n(\Z)$. Of course any hyperplane $V$ can be mapped to the special case $\{u_1=0\}$ by an appropriate linear change of variables, but because in bounding $N(F,B)$ we consider integral points corresponding to a certain box $[-B,B]^n \cap \Z^n$, we must consider carefully how this set of integral points is affected by  a   change of variables.  
Given $L \in \GL_n(\Q)$, denote the rows of the matrix by $L_1,\ldots,L_n$ and correspondingly define $n$ linear forms by $L_i(\bx) = L_i \cdot \bx$.

\begin{cor}\label{cor_second_moment}
Let $F(Y,\bX) \in \Z[Y,X_1,\ldots,X_n]$ with $\deg_YF \geq 2$ and $F$ monic in $Y$ be absolutely irreducible. Define for a given prime $p$,
\[ S(\bu,p) = \sum_{\bx \in \F_p^n} (v_p(\bx)-1)e_p(\bu \cdot \bx)\]
where $v_p(\bx) = \# \{ y \in \F_p : F(y,\bx)=0\}$.
Let $L \in \GL_n(\Q)$ be given, with all integer entries, and let $\|L\|$ denote the maximum of the absolute value of the entries in $L$. If the polynomial $F_L(Y,\bX) :=F(Y,L_1(\bx),\ldots,L_n(\bx))$ satisfies the condition (\ref{alg_closed_assumption}) for $j=1$, then
there exists a set $\mathcal{E}_L$ with 
\[|\mathcal{E}_L| \ll \log (\|F\| \|L \|)/ \log \log (\|F\| \|L\|)\]
such that 
for all $p \not\in \mathcal{E}_L,$
\[ \sum_{\bu \in \{ (L^T\bu)_1=0\} \subset \F_p^n} | S(\bu,p)|^2 \ll_{n,\deg F} p^{2n-1}.\]
(In particular, the conclusion holds if   $F(Y,L_1(\bx),\ldots,L_n(\bx))$ is a strongly  $n$-genuine polynomial.) \end{cor}
\begin{proof}
    Let $L$ be given. By assumption,   $F_L(Y,\bX)$ satisfies the hypotheses of Proposition \ref{prop_strongly_genuine_second_moment} and $\|F_L\| \leq \|F\| \|L\|$, so
    there exists a set $\mathcal{E}'_L$ with $|\mathcal{E}'_L| \ll \log (\|F\| \|L \|)/ \log \log (\|F\| \|L\|)$ such that 
for all $p \not\in \mathcal{E}'_L,$
\[ \sum_{\bu \in \{ u_1=0\} \subset \F_p^n} | S_L(\bu,p)|^2 \ll_{n,\deg F} p^{2n-1},\]
in which 
\[ S_L(\bu,p) :=  \sum_{\bx \in \F_p^n} (v^L_p(\bx)-1)e_p(\bu \cdot \bx)\]
where $v^L_p(\bx) = \# \{ y \in \F_p : F(y,L_1(\bx), \ldots, L_n(\bx))=0\}$.   We will require that $p \ndiv \det L$, so that furthermore $L \in \GL_n(\F_p)$; this excludes at most $\ll_n \log \|L\|/\log \log \|L\|$ further primes, and we enlarge the set $\mathcal{E}'_L$ to a set $\mathcal{E}_L$ that includes these further primes as well. Now suppose $p \not\in \mathcal{E}_L$ and write the matrix inverse as $M$ so that $ML =LM= \mathrm{Id}$ over $\F_p$. Note that for each $\bx \in \F_p^n$,
\begin{align*}
    v^L_p(M\bx) &= \# \{ y \in \F_p : F(y,L_1(M\bx), \ldots, L_n(M\bx))=0\} \\
    &= \# \{ y \in \F_p : F(y,x_1, \ldots, x_n)=0\} = v_p(\bx).
    \end{align*}
Consequently   
\[ S_L(\bu,p) = \sum_{M\bx \in \F_p^n} (v^L_p(M\bx)-1)e_p(\bu \cdot M\bx) = \sum_{M\bx \in \F_p^n} (v_p(\bx)-1)e_p((M^T\bu )\cdot \bx), \]
where we have applied $(\bu \cdot M \bx) = \bu^T (M\bx) = (M^T \bu)^T \bx$. Since $\bx$ ranges over $\F_p^n$ as $M\bx$ ranges over $\F_p^n$, that is to say 
\[ S_L(\bu,p) = \sum_{\bx \in \F_p^n} (v_p(\bx)-1)e_p(M^T\bu \cdot \bx) = S( M^T \bu,p).\]
We can conclude that for all $p \not\in \mathcal{E}_L,$
\[ \sum_{\bu \in \{ u_1=0\} \subset \F_p^n} | S(M^T\bu,p)|^2 \ll_{n,\deg F} p^{2n-1}.\]
Finally, we set $\mathbf{v} = M^T \bu$ (and so $\bu = (M^T)^{-1} \mathbf{v} = L^T \mathbf{v}$), so that this statement is equivalent to 
\[  \sum_{\mathbf{v} \in\{ \mathbf{v} \in \F_p^n :  (L^T \mathbf{v})_1=0\} } | S(\mathbf{v},p)|^2 \ll_{n, \deg F} p^{2n-1}.\]
\end{proof}

\section{Application of the stratification to the polynomial sieve}\label{sec_apply_strat_to_sieve}
We turn to the proof of Theorem \ref{thm_main} and Theorem \ref{thm_explicit} (which we recall immediately implies Theorem \ref{thm_cyclic_uniform} by Example \ref{ex: generic cyclic}). The main claim in this section is Theorem \ref{thm_main_options}, which is the main outcome of the polynomial sieve in this paper. For simplicity we may assume throughout that $\|F\|\geq 3$.
\subsection{Initiation of the polynomial sieve}
 The  strategy is to use a polynomial sieve to control a smoothed version of the counting function $N(F,B)$. Let $w:\R^n\rightarrow \R_{\geq 0}$ be an infinitely differentiable function that is $\geq 1$ on $[-1,1]^n$ and compactly supported in $[-2,2]^n$. For each fixed $B\geq 1$, let $W(\bfx) = w(\bfx/B)$. By construction, 
\[N(F,B) \leq S(F,B):= \sum_{\substack{\bx\in \Z^n\\ F(Y,\bx)=0 \text{ is solvable in $\Z$}}}W(\bfx).\]
The polynomial sieve relies on understanding the local counting function
\[v_p(\bx) = \#\{y\modd{p}: F(y,\bx) = 0 \modd{p}\}\]
for sufficiently many primes $p$ in a sieving set $\mathcal{P}$. The version we employ is as follows:
\begin{lem}[{\cite[Lemma 1.2]{BonPie24}}]
   \label{lemma_poly_sieve}
   Let $m\geq 2$ and $d \geq 1$ be integers. Consider a polynomial of the form $$F(Y,\bX) = Y^{md} + f_1(\bX) Y^{m(d-1)} + ... + f_d(\bX),$$ in which   $f_d(\bX)\not\equiv 0$ and $f_j \in \Z[X_1,\ldots,X_n]$ have degrees $k_j$ (possibly zero), respectively.  Let $\calP \subset \{p\text{ prime}: p\equiv 1 \modd{m}\}\cap [P,2P]$ be a set of primes. Suppose that $|\calP|\gg P/\log P$ and that 
    \beq\label{P_size_lower_bound}
    |\calP| \gg_{\deg f_d} \max\{\log(\|f_d\|,\log B)\}.
    \eeq
    Then  
    $$S(F,B) \ll_{\deg F} \sum_{f_d(\bx)=0}W(\bx) + \frac{1}{|\calP|} \sum_{\bx\in \Z^n}W(\bx) + \frac{1}{|\calP|^2} \sum_{\substack{p,q\in \calP \\ p\neq q}}\left|\sum_{\bx\in \Z^n} (v_p(\bx)-1)(v_q(\bx)-1)W(\bx)\right|.$$
\end{lem}
This formulation of the polynomial sieve introduces the requirement that $m \geq 2$; this guarantees that for a positive proportion of primes (namely those $p \con 1 \modd{m}$), $v_p(\bfx) >0$ implies that $v_p(\bfx) >1$ so that $(v_p(\bfx)-1) \neq 0$. The version stated here is a small modification of \cite[Lemma 1.2]{BonPie24}, which was stated in the setting that the polynomials $f_j$ are homogeneous, with degree constraints of the form $\deg f_j = m \cdot e \cdot j$; then $F(Y,\bX)=0$ defines a hypersurface in weighted projective space $\mathbb{P}(e,1,\ldots,1)$ for a given integer $e \geq 1$. The proof given for \cite[Lemma 1.2]{BonPie24} did not depend on homogeneity of the polynomials $f_j$, and repeating it line-by-line immediately leads to the result stated above, so we do not repeat the proof.

In our present work, we also allow $F(Y,\bX)$ to have the form (\ref{F_dfn_intro})  with $\deg_Y F \geq 2$ and $m=1$, and for this case we apply the conditional version of the sieve lemma:
\begin{lem}[{\cite[Remark 1.3, \S3.2]{BonPie24}}]\label{lemma_poly_sieve_cor}
Lemma \ref{lemma_poly_sieve} holds for $m=1$ and  $d \geq 2$, conditional on GRH, with $\mathcal{P}$ a set of primes in $[P,2P]$ with $|\mathcal{P}| \gg P(\log P)^{-1}$, and (\ref{P_size_lower_bound}) replaced by 
 $
P \gg_{m,d, \deg f_d} \max \{ (\log \|F\|)^{\al_0}, (\log B)^{\al_0}\}
$ for any fixed $\al_0>2$.
\end{lem} 
For clarity, we note that \cite[Remark 1.3, \S3.2]{BonPie24} was also stated in the case where the polynomials $f_j$ are homogeneous, but the proof given there did not depend on homogeneity, and repeating it line-by-line immediately leads to the result stated above. (The dependence of Lemma \ref{lemma_poly_sieve_cor} on GRH can be weakened somewhat; see  Remark \ref{remark_GRH_number_of_apps}.)

Assume that $F(Y,\bX)$ with $\deg_Y F \geq 2$ is any absolutely irreducible polynomial of the form (\ref{F_dfn_intro}) for a fixed $m \geq 1$. (In particular, since $F(Y,\bX)$ is absolutely irreducible, it must be that $f_d(\bX) \not\equiv 0$.) 
By an application of Lemma \ref{lemma_poly_sieve} (when $m\geq 2$) or Lemma \ref{lemma_poly_sieve_cor} (when $m=1$, assuming GRH),
   $$N(F,B) \ll_{\deg F} \sum_{f_d(\bx)=0}W(\bx) + \frac{1}{|\calP|} \sum_{\bx\in \Z^n}W(\bx) + \frac{1}{|\calP|^2} \sum_{\substack{p,q\in \calP \\ p\neq q}}\left|\sum_{\bx\in \Z^n} (v_p(\bx)-1)(v_q(\bx)-1)W(\bx)\right|.$$
  Here the sieving set $\mathcal{P}$ is assumed to satisfy  $\mathcal{P} \subset \{p\text{ prime}: p\equiv 1 \modd{m}\}\cap [P,2P]$ and we assume $P \gg_{\deg F} \max \{(\log \|F\|)^{\al_0}, (\log B)^{\al_0}\}$ for a fixed $\al_0>2$.  (The condition $p \con 1 \modd{m}$ denotes primes in an arithmetic progression modulo $m$ if $m \geq 2$; if $m=1$ it  imposes no restriction.) We will state the precise choice of $\mathcal{P}$ in \S \ref{sec_sieving_set},  and  confirm that  $|\calP|\gg P/\log P$.
   For the remainder of the paper, whether $m \geq 2$ or $m =1$ will no longer play any role in the argument. 
   
   An application of the trivial bound (Lemma \ref{lemma_Schwartz_Zippel}) shows that
   $$\sum_{f_d(\bx) =0} W(\bx) \ll_{\deg f_d} B^{n-1}.$$
Trivially,  $$\frac{1}{|\calP|}\sum_{\bx\in \Z^n}W(\bx) \ll_n \frac{B^n}{|\calP|}.$$
Hence to bound $N(F,B)$ from above, it remains to understand the third, main sieve, term.

To bound the main sieve term, we first fix a pair $(p,q)\in \calP^2$ such that $p\neq q$. After an application of Poisson summation,  
\begin{equation*}
    \sum_{\bx\in \Z^n} W(\bx) (v_p(\bx)-1)(v_q(\bx)-1) = \frac{1}{(pq)^n} \sum_{\bfu\in \Z^n} \hat{W}(\bfu/pq)S(\bfu,pq), 
\end{equation*}
where 
$$S(\bfu,pq) = \sum_{\bfa \modd{pq}} (v_p(\bfa)-1)(v_q(\bfa)-1)e_{pq}(\bfu\cdot \bfa).$$
Accordingly, for a prime $p$ define 
\beq\label{S_prime_dfn}
S(\bfu,p) = \sum_{\bfa \modd{p}} (v_p(\bfa)-1)e_{p}(\bfu\cdot \bfa).
\eeq
The sum is multiplicative:  for primes $p \neq q$,
\begin{equation*}
    S(\bfu,pq) = S(\overline{q}\bfu,p) S(\overline{p}\bfu, q), 
\end{equation*}
in which $\overline{p}p \con 1 \modd{q}$ and $\overline{q}q \con 1 \modd{q}$. 
The standard result of Lemma \ref{lemma_weight_count}(i)  applies to bound the function $\hat{W}(\bu)$. We define the main focus of study: for each pair of primes $p \neq q$ (for any fixed large $M \geq 1$ of our choice, to be determined later), set
\begin{equation}\label{eq: def of T(p,q)}
    T(p,q;B) := \frac{B^n}{(pq)^n} \sum_{\bfu\in \Z^n} \prod_{i=1}^n \left(1+ \frac{|u_i|B}{pq}\right)^{-M} |S(\overline{q}\bfu,p)||S(\overline{p}\bfu,q)|.
\end{equation}
In particular, 
\beq\label{N_bounded_by_T}
N(F,B) \ll_{n,\deg F} B^{n-1} + B^n |\mathcal{P}|^{-1}  + \frac{1}{|\mathcal{P}|^2}\sum_{p \neq q \in \mathcal{P}} |T(p,q;B)|,
\eeq
 for $P \gg_{\deg F} \max \{(\log \|F\|)^{\al_0}, (\log B)^{\al_0}\}$ for a fixed $\al_0>2$.
To control the terms $T(p,q;B)$, the key tool we will use is a stratification result for additive character sums over a variety, with strong uniformity with respect to the variety.

\subsection{Introducing the uniform stratification}
  Let $X$ be an affine scheme over $\Z$ of dimension $n$ and let $ \varphi:X\rightarrow \A^n_{\Z}$ be a finite morphism. Let $p$ be a prime and $\psi:\F_p\rightarrow \C$ be a nontrivial additive character. Define the following exponential sum for $\bfu\in \Z^n$:
  \beq
  S_{p,\psi}(\bfu; \varphi) = \sum_{x\in X(\F_p)} \psi(\bfu\cdot \varphi(x)) = \sum_{\by\in \F_p^n}\psi(\bu\cdot \by) r_\varphi(\by),
  \eeq
  where we define $r_{\varphi}(\by) = \#\{x\in X(\F_p): \varphi(x) = \by\}.$ The function $r_\varphi(\by)$ is a trace function and fits into the stratification framework.

  Loosely speaking, a stratification associated to this sum is a nested collection of varieties $\A^n =: Z_0\supset Z_1 \supset ...$ such that the tuples  $\bfu$ that lie progressively deeper in this nesting sequence, have progressively worse upper bounds on $|S_{p,\psi}(\bfu;\varphi)|$, in terms of the power of $p$. The key is to produce lower bounds on the codimension of $Z_j$ for each $j$, so that ``progressively worse'' $\bfu$ become ``progressively sparser.''
  The sums $S(\bfu,p)$ we consider in (\ref{S_prime_dfn}) fit this paradigm: first write
$$S(\bu,p) = \sum_{\ba \modd{p}} (v_p(\ba)-1) e_{p}(\bu\cdot \ba) = \sum_{\substack{(y,\bx) \in \F_p^{n+1}\\F(y,\bx) = 0\modd{p}}} e_{p}(\bu\cdot \bx) - p^n \mathbf{1}_{\bu=\mathbf{0}}.$$
Now let 
\[ X = \{F(Y,\bX)=0\} \subset \mathbb{A}_\Z^n\]
and $\varphi: \{F(Y,\bX)=0\} \rightarrow \A^n$, where $\varphi((y,\bx))= \bx$, and define the additive character $\psi(a) =e_p(a)$. Then  for each $\bu \neq \mathbf{0}$, $S(\bu,p)$ is  in the form of the sum $S_{p,\psi}(\bu;\varphi)$ defined above; for this particular map and choice of additive character, we define 
\begin{equation}\label{S_recap}
\mathbf{S}_{p}(\bu;F)  := S_{p,\psi}(\bu,\varphi).
\end{equation}

The existence of such a stratification has been proved by Katz and Laumon \cite{KatLau85} (further explained in Fouvry \cite{Fou00}). 
We require an enhanced version of the stratification, which tracks the dependence of all parameters on the scheme $X$. We have developed such an enhanced stratification in a companion paper \cite{BKW25x} with E. Kowalski, and here we record the relevant result from that work, in the form of \cite[Thm. 2.3]{BKW25x}.

\begin{thm}[Uniform stratification]\label{thm_strat_details}
Fix integers $n \geq 1$ and $D \geq 1$, and let $F\in \Z[Y,X_1,\ldots,X_n]$ be a polynomial of total degree at most $D$.  Then there exist integers $C(n,D)$ and $N(n,D)$ and a stratification 
\[ \A^n_{\Z} = V_0\supset V_1 \supset ... \supset V_n\]
such that $V_i$ is a homogeneous subvariety of codimension $\geq i$ and  for all primes $p \ndiv N(n,D)$,   for $\bu\in V_i(\F_p)\backslash V_{i+1}(\F_p)$,
$$|\mathbf{S}_{p}(\bu;F)|\leq C(n,D) p^{\frac{n+i}{2}}.$$
Moreover, 
\begin{enumerate}[(i)]
     \item \label{thm_strat_C} $C(n,D)\ll_{n,D} 1$;
    \item \label{thm_strat_deg} $\deg(V_i) \ll_{n,D} 1$;
    \item \label{thm_strat_num_comp} the number of irreducible components of $V_i$ is $\ll_{n,D} 1;$
      \item \label{thm_strat_N} $N(n,D) \ll_{n,D} 1$;
    \item \label{thm_strat_coeff} each  $V_i = V(G_{i,1},...,G_{i,k})$ is the common vanishing set of homogeneous polynomials $G_{i,s}$ with the property that $\log\|G_{i,s}\|\ll_{n,D} \log\|F\|$ for every pair $i,s.$
\end{enumerate}
\end{thm}
We will say that such a stratification is uniform; in particular, this result   follows from a statement of algebraic uniformity of the Katz--Laumon stratification, as proved in \cite{BKW25x}.
 Note that the first four statements are   uniform in $F$; the impact of the last   statement is that the homogeneous polynomials $G_{i,s}$  are defined over $\Z$ and the dependence of $\|G_{i,s}\|$ on $\|F\|$ is at most polynomial.

 Let us then denote by 
\beq\label{nesting_V} 
\mathbb{A}_{\mathbb{Z}}^n =:V_0 \supset V_1\supset V_2\supset \cdots \supset V_n
\eeq
the stratification provided for the exponential sum $\mathbf{S}_{p}(\bu;F)$ in  (\ref{S_recap}).
Since $V_n$ is a homogeneous variety of codimension at least $n$ in $\mathbb{A}_{\Z}^n$, it must be that $V_n$ is only $V_n = \{\mathbf{0}\}.$ 
Now fix some prime $p$ in the sieving set $\mathcal{P}$. 
By hypothesis $\{F(Y,\bX) = 0\}$ is absolutely irreducible over $\Q$, and by the choice of sieving set made momentarily in \S \ref{sec_sieving_set} we may assume that for all $p \in \mathcal{P}$   the reduction remains absolutely irreducible over $\F_p$.
Consequently, for $\bu\equiv \mathbf{0} \modd{p}$,  by an application of the Lang-Weil bound (Lemma \ref{lemma_Lang_Weil_cor}),  
\beq\label{eq: zero freq}
S(\mathbf{0},p) = \#X(\F_p) - p^n \ll_{n,\deg F}p^{n-1/2}.
\eeq
Thus for each $p \in \mathcal{P}$, as a consequence of Theorem \ref{thm_strat_details} for $j=n-1$ applied to $\bu \in V_{n-1}(\F_p)\setminus \{\boldsymbol{0}\}$,   and (\ref{eq: zero freq})  applied to $\bu \con \boldsymbol{0} \modd{p}$,  
\beq\label{S_bound_all_V_n_1}
\text{$|S(\bu,p)|\ll p^{n-1/2}$ holds for \emph{all} $\bu \in V_{n-1}(\F_p)$}.
\eeq
 In what follows, we thus formally re-define 
$V_n = \emptyset$, so that any restriction of the form $\bu \in V_{n-1} \setminus V_n$ can instead be read as $\bu \in V_{n-1}$, and this will lead to valid upper bounds due to the uniformity in (\ref{S_bound_all_V_n_1}).

\subsection{Degenerate irreducible components of the stratification}\label{sec_degenerate_primes}
   For each $1 \leq j \leq n-1$ decompose the stratum $V_j$ into irreducible components over $\Q$ as 
\[ V_j = \bigcup_{s \in I_j} X_{j,s}.\] 
By Theorem \ref{thm_strat_details}(\ref{thm_strat_num_comp}), $|I_j| \ll_{n,\deg F} 1$.   We must   consider whether any stratum $V_j$ contains an irreducible component  $X_{j,s} \subset \mathbb{A}^n(\Z)$ which is contained in a proper linear subspace of $\mathbb{A}^n(\Z)$; in this case we will say the component $X_{j,s}$ is degenerate. (Recall that a variety $X \subset \mathbb{P}^{n-1}$ is said to be nondegenerate if it spans $\mathbb{P}^{n-1}$, or equivalently, does not lie in any hyperplane \cite[p. 145]{Har92}.) Fix a degenerate irreducible component $X_{j,s}$; we may in particular assume it lies in a hyperplane, say $W_{j,s}$, defined by $\{ \bu: W^{(j,s)} \cdot \bu =0\} \subset \mathbb{A}^n(\Z)$ for a nonzero vector $W^{(j,s)} \in \Z^n$.  By Theorem \ref{thm_strat_details}(\ref{thm_strat_coeff}), $\log \|W^{(j,s)}\| \ll_{n,\deg F} \log \|F\|.$
We define $L_{j,s} \in \GL_n(\Q)$ to be a matrix with first column $W^{(j,s)}$ and all integer entries; once the first column is fixed we may choose the other entries appropriately to make the matrix in $\GL_n(\Q)$, with $ \log \|L_{j,s} \|\ll_{n,\deg F} \log \|F\|$.   Then in particular,
\beq\label{parameterize_W_plane}
\{ \bu : (L_{j,s}^T \bu)_1 =0\}=\{\bu  : W^{(j,s)} \cdot \bu =0\}  .
\eeq

  Cumulatively, let $\{W_{j,s}\}_{(j,s)}$ be the finite set of (distinct) hyperplanes that contain a degenerate irreducible component $X_{j,s}$ of any stratum $V_j$ in the stratification $V_0\supset V_1 \supset \ldots \supset V_n$. By Theorem \ref{thm_strat_details}(\ref{thm_strat_num_comp}), there are $\ll_{n,\deg F}1$ hyperplanes in this set (and it may be empty).   Let $\mathcal{L}(F) = \{L_{j,s}\}_{(j,s)}$ denote the finite (possibly empty) set of corresponding (distinct)   transformations, with $|\mathcal{L}(F)| \ll_{n,\deg F} 1$. Both Theorem \ref{thm_main} and Theorem \ref{thm_explicit} will be deduced from the following theorem, whose proof occupies the remainder of the paper.
 \begin{thm}\label{thm_main_options}
 Fix $n \geq 2$. Let $F \in \Z[Y,X_1,\ldots,X_n]$ be an absolutely irreducible polynomial of the form  (\ref{F_dfn_intro}) for an integer $m \geq 2$. 
 Let  $\mathcal{L}(F) \subset \GL_n(\Q)$ be the finite (possibly empty) set of linear transformations with integer entries, as defined above, with $|\mathcal{L}(F)| \ll_{n,\deg F} 1$ and with $\log \|L\| \ll_{n,\deg F}\log \|F\|$ for each $L \in \mathcal{L}(F)$.  Assume that $F(Y,L(\bX))$  is strongly $n$-genuine for each  $L \in \mathcal{L}(F)$.
 Then for $B \gg_{n,\deg F} (\log \|F\|)^6$,
\[ N(F,B) \ll_{n,\deg F, \ep} (\log \|F\|) B^{n-1+\frac{1}{n+1} + \ep} \qquad \text{for every $\ep>0$}.\] 
 If $\deg_Y F \geq 2$ and the same hypotheses hold with  $m =1$, the same conclusions hold, conditional on GRH.
\end{thm}
\begin{rem}
Suppose that an absolutely irreducible polynomial $F$ of the form (\ref{F_dfn_intro}) has the property $(*)$ that no stratum $V_j$ with $j \leq n-1$ in the associated stratification $\A^n = V_0 \supset V_1 \supset \cdots \supset V_{n-1}\supset V_n  $ contains an irreducible component that lies in a linear subspace. Then  $\mathcal{L}(F)$ is empty, and the bound for $N(F,B)$ stated above holds (under GRH if $m=1$). By  homogeneity, $V_n= \{ \mathbf{0}\}$ (as remarked above), and the irreducible components of $V_{n-1}$ can be seen to be lines (as argued below in \S \ref{sec_glo_glo_nondegenerate}).  So in particular  property $(*)$ would also imply that $V_{n-1} = \{ \mathbf{0}\}$. It would be interesting to illuminate what is characterized about a polynomial $F$ if the strata satisfy  property $(*)$.
\end{rem}

\subsection{Deduction of Theorem \ref{thm_explicit}}\label{sec_deduce_thm_explicit}
Under the hypothesis of Theorem \ref{thm_explicit} that $F(Y,\bX)$ is strongly $(1,n)$-allowable, $F_{L}(Y,\bX):= F(Y,L(\bX))$ is strongly $n$-genuine for all $L \in \GL_n(\Q)$ so certainly for each $L \in \mathcal{L}(F)$. Thus  Theorem \ref{thm_main_options} proves the claim in Theorem \ref{thm_explicit} for all $B \gg_{n,\deg F}(\log \|F\|)^6$.  For 
all $B \ll_{n,\deg F} (\log \|F\|)^6$, we can apply the trivial bound 
\[ N(F,B) \ll_n B^n \ll_{n,\deg F} (\log \|F\|)^{6n} \ll_{n,\deg F,\ep} (\log \|F\|)^{6n} B^{n-1+\frac{1}{n+1}+\ep},\]
completing the proof of the theorem for all $B \geq 1$.

\subsection{Deduction of Theorem \ref{thm_main}}\label{sec_deduce_thm_main}
For Theorem \ref{thm_main}, suppose that $F(Y,\bX)$ is a $(1,n)$-allowable polynomial. If $F_{L}(Y,\bX)=F(Y,L(\bX))$ is strongly $n$-genuine for each $L \in \mathcal{L}(F)$, then Theorem \ref{thm_main_options} applies directly, yielding the claim of Theorem \ref{thm_main} (by the reasoning immediately above). On the other hand, suppose that for some $L \in\mathcal{L}(F)$, say $L_*$, the polynomial $F_{L_*}(Y,\bX):=F(Y,L_*(\bX))$ is $n$-genuine but not strongly $n$-genuine. Then we do not apply the polynomial sieve to $F$. Instead, as a consequence of   Theorem \ref{thm_genuine_not_strongly_NFB}, in the form of Corollary \ref{cor_thm_genuine_not_strongly_NFB}, we know 
that there exist integers $e=e(n)$ and $e'=e'(n)$ such that for all $B \geq 1$,
\[
N(F,B) \ll_{n,\deg F} \|L_*\|^{e'} (\log (\|F\|+2))^e B^{n-1} (\log (B+2))^{2e}.
\]
By Theorem \ref{thm_strat_details}(v),  $\log\|L_*\|\ll_{n,\deg F} \log\|F\|$, so that there exists $h=h(n,\deg F)$ for which $\|L_*\|^{e'} \ll_{n,\deg F} \|F\|^{h}$.  This verifies the remaining case of Theorem \ref{thm_main}.

\subsection{Adapting the second moment bound to the setting of Theorem \ref{thm_main_options}}\label{sec_adapting_second_moment}
In the setting of Theorem \ref{thm_main_options}, we will apply the second moment bound proved in \S \ref{sec_strongly} in the following formulation. 
As in \S \ref{sec_degenerate_primes}, let  $X_{j,s}$ be a degenerate irreducible component of a stratum $V_j$, with associated $L_{j,s} \in \mathcal{L}(F)$   with integer entries, which defines a hyperplane $W^{(j,s)}$ containing $X_{j,s}$ as in (\ref{parameterize_W_plane}). By Theorem \ref{thm_strat_details}(v),  $\log\|L_{j,s}\|\ll_{n,\deg F} \log\|F\|$.
By Corollary \ref{cor_second_moment}, there exists an exceptional set $\mathcal{E}_{j,s}$ of  primes with 
\[|\mathcal{E}_{j,s}| \ll_{n,\deg F}\log ( \|F\|\|L_{j,s}\|)/ \log \log (\|F\| \|L_{j,s}\|)\ll_{n,\deg F} \log \|F\|/ \log \log \|F\|,\]
such that
for all $p \not\in \mathcal{E}_{j,s},$
\[ \sum_{\bu \in X_{j,s}(\F_p)  \subset \F_p^n} | S(\bu,p)|^2 \leq \sum_{\bu \in  W_{j,s}(\F_p) \subset \F_p^n} | S(\bu,p)|^2  \ll_{n,\deg F} p^{2n-1}.\]
The first inequality is by positivity, and the second is by Corollary \ref{cor_second_moment}, via the parameterization (\ref{parameterize_W_plane}). It immediately follows that
for any integer $M \geq 1$, $$\#\{\bu\in X_{j,s}(\F_p): |S(\bu,p)|\geq M\} \ll_{n, \deg F} \frac{p^{2n-1}}{M^2}.$$
 Moreover, by the homogeneity of $X_{j,s}$, we can deduce from this that for any $\al \in \F_p^\times$,
 \beq\label{second_moment_apply}
 \#\{\bu\in X_{j,s}(\F_p): |S(\al \bu,p)|\geq M\} \ll_{n, \deg F} \frac{p^{2n-1}}{M^2}.
 \eeq
 This is the main consequence of Corollary \ref{cor_second_moment} we will apply in \S \ref{sec_glo_glo_degenerate} (and again in \S \ref{sec_glo_loc}).

\subsection{Choice of the sieving set}\label{sec_sieving_set}
We now turn to the proof of Theorem \ref{thm_main_options}, which is the main outcome of the polynomial sieve in this paper.
We  choose the sieving set $\mathcal{P}$, adapted to the stratification of Theorem \ref{thm_strat_details}.
Fix $N(n,\deg F)$ as in Theorem \ref{thm_strat_details} (\ref{thm_strat_N}) and let $\mathcal{E}_0$ denote the primes $p|N(n,\deg F)$  so that $|\mathcal{E}_0|\ll_{n,\deg F} 1$.  
Let $\mathcal{E}_1$ denote the finitely many exceptional primes for which the reduction of $F(Y,\bX)$ is not   irreducible over $\overline{\F}_p$; by Lemma \ref{lemma_Bertini_Noether}, $|\mathcal{E}_1| \ll_{n,\deg F} \log \|F\| / \log \log \|F\|$. 

In \S \ref{sec_adapting_second_moment}, for each $1 \leq j \leq n-1$, for each $s \in I_j$ such that $X_{j,s}$ is a degenerate irreducible component of the stratum $V_j$, we have isolated a set $\mathcal{E}_{j,s}$ of exceptional  primes with $|\mathcal{E}_{j,s}| \ll_{n,\deg F} \log \|F\| / \log \log \|F\|$.   If (for a given $j,s$) $X_{j,s}$ is nondegenerate, set $\mathcal{E}_{j,s} = \emptyset.$ Define 
\[ \mathcal{E}_{\mathrm{deg}}  = \bigcup_{1 \leq j \leq n-1} \bigcup_{s \in I_j} \mathcal{E}_{j,s}.\] 
By an application of Theorem \ref{thm_strat_details}(\ref{thm_strat_num_comp}),  $|\mathcal{E}_{\mathrm{deg}}|\ll_{n,\deg F} \log \|F\|/ \log \log \|F\|$.

  By the Siegel-Walfisz theorem for primes in arithmetic progressions ($m \geq 2$) or the prime number theorem ($m=1$),  $\#\{p \in [P,2P]: p\con 1 \modd{m}\} \gg_m P/\log P$ for all $P \gg_m 1$.
  Thus as long as 
  \beq\label{P_big_enough}
  P \gg_{n, \deg F} \log \|F\|/\log \log \|F\|,
  \eeq
by taking (either for $m \geq 2$ or for $m=1$),
\[ \mathcal{P} = \{p \in [P,2P]: p\con 1 \modd{m}\} \setminus (\mathcal{E}_0 \cup \mathcal{E}_1 \cup \mathcal{E}_{\mathrm{deg}}),\]
we still obtain $|\mathcal{P}| \gg P/\log P$.
We also assume from now on that 
\beq\label{P_size_half}
P=B^{\rho}, \qquad 1/2 < \rho<1.\eeq

\subsection{Separation into global and local contributions}

To prove Theorem \ref{thm_main_options}, we now bound $N(F,B)$ via the polynomial sieve; recall the main consequence of the sieve lemma stated in (\ref{N_bounded_by_T}) with $T(p,q;B)$ defined in (\ref{eq: def of T(p,q)}).
We  apply the stratification to the evaluation of $T(p,q;B)$ for a fixed pair $p \neq q \in \mathcal{P}$. First, write
\begin{align*}
    T(p,q;B) &\leq   \frac{B^n}{(pq)^n} \sum_{j=0}^{n-1} \sum_{k=0}^{n-1} \sum_{\substack{\bu \in \Z^n \\ \overline{q}\bu \in V_{j}(\F_p)\backslash V_{j+1}(\F_p) \\ \overline{p}\bu \in V_{k}(\F_q)\backslash V_{k+1}(\F_q)}} |S(\overline{q}\bu,p)| |S(\overline{p}\bu,q)|\prod_{i=1}^n \left(1+\frac{|u_i|B}{pq}\right)^{-M}\\
    &= \frac{B^n}{(pq)^n} \sum_{j=0}^{n-1} \sum_{k=0}^{n-1} \sum_{\substack{\bu \in \Z^n \\ \bu \in V_{j}(\F_p)\backslash V_{j+1}(\F_p) \\ \bu \in V_{k}(\F_q)\backslash V_{k+1}(\F_q)}} |S(\overline{q}\bu,p)| |S(\overline{p}\bu,q)|\prod_{i=1}^n \left(1+\frac{|u_i|B}{pq}\right)^{-M} .
    \end{align*}
    The sums need only index over $j \leq n-1$ and $k \leq n-1$ by the remark (\ref{S_bound_all_V_n_1}) and the corresponding convention $V_n :=\emptyset$.
  The second line is equivalent to the first since the strata $V_j$ are homogeneous varieties.  Throughout what follows, we will repeatedly use   the homogeneity of the varieties $V_i$ in order to impose a constraint on $\bu$ rather than on $\overline{q}\bu$ or $\overline{p}\bu$. When we next average $T(p,q;B)$ over $p\neq q\in \calP$, this will enable us to bring the sum over $p,q$ inside the sum over $\bu$. Indeed, 
  
\begin{multline}\label{T_average_first_step}
\frac{1}{|\calP|^2}\sum_{\substack{p,q\in \calP \\ p\neq q}} T(p,q;B) \\ 
\ll \frac{B^n}{P^{2n}} \sum_{j=0}^{n-1}\sum_{k=0}^{n-1}    \frac{1}{|\calP|^2} \sum_{\bu\in \Z^n} \sum_{\substack{p \neq q \in \mathcal{P} \\ \bu \in V_{j}(\F_p)\backslash V_{j+1}(\F_p) \\ \bu \in V_{k}(\F_q)\backslash V_{k+1}(\F_q)}} |S(\overline{q}\bu,p)| |S(\overline{p}\bu,q)| \prod_{i=1}^n \left(1+\frac{|u_i|B}{P^2}\right)^{-M} .
\end{multline}
To bring the sum over $p,q$ inside the sum over $j$, we have also used the fact that the strata $V_{n-1} \subset \cdots \subset V_1 \subset V_0 = \mathbb{A}^n_{\Z}$ provided by Theorem \ref{thm_strat_details} are defined over $\Z$, independent of $p$ and $q$. 

We say that a vector $\bu\in \Z^n$ is \textit{globally $j$-bad} if $\bu\in V_j(\Z)$. For such $\bu$, it must be that $\bu \in V_j(\F_p)$ for all primes $p$. Otherwise we say $\bu\in \Z^n$ is \textit{locally $j$-bad} if $\bu \in V_j(\F_p)$ for some $p\in \calP$, but $\bu\not\in V_j(\Z)$. For a given pair of indices $j,k$,  the sum over $\bu$ splits into four parts: those $\bu$ that are globally $j$-bad and globally $k$-bad; those  $\bu$  that are globally $j$-bad and locally $k$-bad; those  $\bu$ that are locally $j$-bad and globally $k$-bad; and those  $\bu$ that are locally $j$-bad and locally $k$-bad. We first manipulate the contribution of each of these cases into a format we require in later steps.

\subsubsection{Preprocessing the global/global contribution}
Let $j$ and $k$ be fixed. The global/global contribution to (\ref{T_average_first_step}) for this $j,k$ pair takes the form
\[S^{j,k}_{\glo,\glo}:=
\frac{1}{|\calP|^2}  \sum_{\substack{\bu\in \Z^n\\ \bu \in V_{j}(\Z)\backslash V_{j+1}(\Z) \\ \bu \in V_{k}(\Z)\backslash V_{k+1}(\Z)} } \sum_{\substack{p \neq q \in \mathcal{P} \\ \bu \in V_{j}(\F_p)\backslash V_{j+1}(\F_p) \\ \bu \in V_{k}(\F_q)\backslash V_{k+1}(\F_q)}} |S(\overline{q}\bu,p)| |S(\overline{p}\bu,q)| \prod_{i=1}^n \left(1+\frac{|u_i|B}{P^2}\right)^{-M} .\]
By the nesting of the stratification (\ref{nesting_V}), 
  if $\bu \in V_{j}(\Z)\backslash V_{j+1}(\Z)$ and  $\bu \in V_{k}(\Z)\backslash V_{k+1}(\Z)$, then $j<k+1$ and $k< j+1$ (namely $j=k$). Thus we need only consider terms of the form $S^{j,k}_{\glo,\glo}$ for $j=k$, $j=0,\ldots, n-1$.
    In the internal sum over $p,q$,  it is equivalent (by homogeneity) to impose that $\overline{q}\bu \in V_{j}(\F_p)\backslash V_{j+1}(\F_p) $ as to impose $\bu \in V_{j}(\F_p)\backslash V_{j+1}(\F_p) $, and similarly we may impose $\overline{p}\bu \in V_{k}(\F_q)\backslash V_{k+1}(\F_q)$. Hence we may apply the bounds of Theorem \ref{thm_strat_details} to the relevant factors $|S(\overline{q}\bu,p)|,$ $|S(\overline{p}\bu,q)|$.
  
   It is   convenient to treat the $j=k=0$ contribution to the global/global case immediately: this is the case with square root cancellation. Since $V_0 = \A^n(\Z)$,  by an application of Lemma \ref{lemma_weight_count} (with $\kappa=n$, $Y = P^2/B \gg 1$  under the assumption of (\ref{P_size_half})), 
 \beq\label{S_global_zero}
S_{\glo,\glo}^{0,0}\ll C(n,\deg F)^2 P^{n/2}P^{n/2}\sum_{\bu \in \Z^n} \prod_{i=1}^n \left(1+\frac{|u_i|B}{P^2}\right)^{-M}\ll P^n \left(\frac{P^2}{B}\right)^{n+\ep}.
 \eeq
 (In what follows, we will allow implicit constants to depend on $n,\deg F$ and $\ep$, for an arbitrarily small $\ep>0$ that may vary from line to line, without always noting this.)
  For each $1 \leq j \leq n-1$ we  now decompose the stratum $V_j \subset \mathbb{A}^n (\Z)$ into its irreducible components, say 
\[ V_j = \bigcup_{s\in I_j} X_{j,s},\]
in which $|I_j|\ll_{n,\deg F} 1$ by Theorem \ref{thm_strat_details}(\ref{thm_strat_num_comp}).
Then the global/global contribution of the term $S^{j}_{\glo,\glo} := S^{j,j}_{\glo,\glo} $
is:
\beq\label{S_j_glo_glo_expression}
  \frac{1}{|\calP|^2} \sum_{\substack{s,t \in I_{j}\\ s',t' \in I_{j+1}}} \sum_{\substack{\bu\in \Z^n\\ \bu \in X_{j,s}(\Z)\backslash X_{j+1,s'}(\Z) \\ \bu \in X_{j,t}(\Z)\backslash X_{j+1,t'}(\Z) }} \sum_{\substack{p \neq q \in \mathcal{P} \\ \bu \in X_{j,s}(\F_p)\backslash X_{j+1,s'}(\F_p) \\ \bu \in X_{j,t}(\F_q)\backslash X_{j+1,t'}(\F_q)}} |S(\overline{q}\bu,p)| |S(\overline{p}\bu,q)| 
\prod_{i=1}^n \left(1+\frac{|u_i|B}{P^2}\right)^{-M}.
 \eeq
Here $X_{j,s}(\F_p)\setminus X_{j+1,s'}(\F_p)$ indicates $X_{j,s}(\F_p) \setminus (X_{j,s} \cap X_{j+1,s'})$, and similarly for the other conditions on   $\bu$. We will use this convention repeatedly in what follows.  For later reference, we also note that by Theorem \ref{thm_strat_details}, we can bound the above expression by
\[
S^{j}_{\glo,\glo} \ll_{n,\deg F} \max_{p \neq q \in \calP}  \sum_{\substack{s,t \in I_{j}\\ s',t' \in I_{j+1}}} \sum_{\substack{\bu\in \Z^n\\ \bu \in X_{j,s}(\Z)\backslash X_{j+1,s'}(\Z) \\ \bu \in X_{j,t}(\Z)\backslash X_{j+1,t'}(\Z) }}   p^{\frac{n+j}{2}} q^{\frac{n+j}{2}} \prod_{i=1}^n \left(1+\frac{|u_i|B}{P^2}\right)^{-M}.
 \]
 The conditions place $\bu$ in an intersection, so by positivity we only possibly enlarge the upper bound by writing this for each $1 \leq j \leq n-1$ as 
\beq\label{S_j_glo_glo_expression_alternate}
S^j_{\glo,\glo}  \ll |I_j||I_{j+1}|\max_{p \neq q \in \calP}  \sum_{s \in I_{j}} \sum_{s' \in I_{j+1}} \sum_{\substack{\bu\in \Z^n\\ \bu \in X_{j,s}(\Z)\backslash X_{j+1,s'}(\Z)  }}   p^{\frac{n+j}{2}}q^{\frac{n+j}{2}} \prod_{i=1}^n \left(1+\frac{|u_i|B}{P^2}\right)^{-M}.
 \eeq
 The contribution (\ref{S_j_glo_glo_expression}) from the global/global case can   be inserted in (\ref{T_average_first_step}), recalling the diagonal condition $j=k$, and $1 \leq j \leq n-1$. 

 \subsubsection{Preprocessing the local/local contribution}
Let $j, k$ be fixed. The local/local contribution to (\ref{T_average_first_step}) for this $j,k$ pair takes the form 
\[
S^{j,k}_{\loc,\loc} :=
\frac{1}{|\calP|^2}  \sum_{\substack{\bu\in \Z^n\\ \bu \not\in V_{j}(\Z)\backslash V_{j+1}(\Z) \\ \bu \not\in V_{k}(\Z)\backslash V_{k+1}(\Z)} } \sum_{\substack{p \neq q \in \mathcal{P} \\ \bu \in V_{j}(\F_p)\backslash V_{j+1}(\F_p) \\ \bu \in V_{k}(\F_q)\backslash V_{k+1}(\F_q)}} |S(\overline{q}\bu,p)| |S(\overline{p}\bu,q)| \prod_{i=1}^n \left(1+\frac{|u_i|B}{P^2}\right)^{-M} .\]
We can write the conditions on the sum over $\bu$ as $\bu \not\in V_{j}(\Z)$ and  $\bu \not\in V_{k}(\Z)$. (Otherwise, suppose $\bu \not\in V_{j}(\Z)\backslash V_{j+1}(\Z)$ but $\bu \in V_{j+1}(\Z)$; then it would be the case that $\bu \in V_{j+1}(\F_p)$ for all $p$, so the restricted sum over $p$ in the expression above would be empty, leading to zero contribution. An analogous argument holds for the second condition on $\bu$.) By symmetry it suffices to suppose that $j \leq k$, so that the stronger condition is $\bu \not \in V_{j}(\Z)$. 
We furthermore apply Theorem \ref{thm_strat_details} to bound the sums modulo $p$ and $q$. The local/local contribution for a given pair $j \leq k$ may thus be bounded by
\[
C(n,\deg F)^2 \frac{1}{|\calP|^2}  \sum_{\substack{\bu\in \Z^n\\ \bu \not\in V_{j}(\Z) } } \sum_{\substack{p \neq q \in \mathcal{P} \\ \bu \in V_{j}(\F_p)  \\ \bu \in V_{k}(\F_q) }} p^{\frac{n+j}{2}}q^{\frac{n+k}{2}} \prod_{i=1}^n \left(1+\frac{|u_i|B}{P^2}\right)^{-M} .\]
Furthermore, note that if   $j=0$  this contribution is zero, since the condition $\bu \not\in V_0(\Z) = \mathbb{A}^n(\Z)$ makes the sum over $\bu$ empty. 
In total,   the local/local contribution need only consider $1 \leq j \leq k \leq n-1$, and for such indices can  be represented as  
\beq\label{S_j_loc_loc_expression}
S^{j,k}_{\loc,\loc}
 \ll \frac{1}{|\calP|^2} P^{n+j/2+k/2}  \sum_{q \in \mathcal{P}} \sum_{\substack{\bu\in \Z^n\\ \bu \not\in V_{j}(\Z) \\ \bu \in V_{k}(\F_q)} }   \prod_{i=1}^n \left(1+\frac{|u_i|B}{P^2}\right)^{-M}
 \sum_{\substack{ p \in \mathcal{P} \subset [P,2P]   \\ \bu \in V_{j}(\F_p)}}1.
 \eeq
 We have chosen to count the number of $p$ as the innermost sum, and preserve the condition $\bu \in V_k(\F_q)$ in the sum over $\bu$ (rather than reversing the roles of $p$ and $q$) because $V_k$ has higher codimension ($j \leq k$).

 \subsubsection{Preprocessing the global/local contribution}
We consider one final case: the global/local contribution (the local/global contribution is equivalent, by symmetry considerations).
Let $j,k$ be fixed. The global/local contribution to (\ref{T_average_first_step}) for this $j,k$ pair takes the form 
 \[
S^{j,k}_{\glo,\loc} :=
\frac{1}{|\calP|^2}  \sum_{\substack{\bu\in \Z^n\\ \bu \in V_{j}(\Z)\backslash V_{j+1}(\Z) \\ \bu \not\in V_{k}(\Z)\backslash V_{k+1}(\Z)} } \sum_{\substack{p \neq q \in \mathcal{P} \\ \bu \in V_{j}(\F_p)\backslash V_{j+1}(\F_p) \\ \bu \in V_{k}(\F_q)\backslash V_{k+1}(\F_q)}} |S(\overline{q}\bu,p)| |S(\overline{p}\bu,q)| \prod_{i=1}^n \left(1+\frac{|u_i|B}{P^2}\right)^{-M} .\]
 We can regard the condition $ \bu \not\in V_{k}(\Z)\backslash V_{k+1}(\Z)$ as stating that $ \bu \not\in V_{k}(\Z)$, since if $\bu \in V_{k+1}(\Z)$ it would follow that $\bu \in V_{k+1}(\F_q)$ for all $q$ so that the sum over $q$ is empty, leading to zero contribution. Thus we consider all $\bu$ such that $\bu \in V_{j}(\Z) \setminus V_{j+1}(\Z)$ and $\bu \not\in V_{k}(\Z)$. By the nesting property (\ref{nesting_V}), for the set of such $\bu$ to be nonempty it must be the case that $j< k$, and then it suffices to simply write $\bu \in V_{j}(\Z) \setminus V_{j+1}(\Z)$. 
Finally, we decompose $V_j=\cup_{s \in I_j}X_{j,s}$ and $V_k=\cup_{t \in I_k}X_{k,t}$ into irreducible components, writing
\beq\label{S_j_glo_loc_expression}
S^{j,k}_{\glo,\loc} =
\frac{1}{|\calP|^2} \sum_{\substack{s \in I_j \\ t \in I_k }} \sum_{\substack{\bu\in \Z^n\\ \bu \in X_{j,s}(\Z)\backslash V_{j+1}(\Z)  } } \sum_{\substack{p \neq q \in \mathcal{P} \\ \bu \in X_{j,s}(\F_p)\backslash V_{j+1}(\F_p) \\ \bu \in X_{k,t}(\F_q)\backslash V_{k+1}(\F_q)}} |S(\overline{q}\bu,p)| |S(\overline{p}\bu,q)| \prod_{i=1}^n \left(1+\frac{|u_i|B}{P^2}\right)^{-M} .
\eeq
 We will bound this contribution for all pairs of indices $j,k$ with $0 \leq j<k \leq n-1$.\\ 

 Finally, we assemble the three types of cases from (\ref{S_j_glo_glo_expression}), (\ref{S_j_loc_loc_expression}), and (\ref{S_j_glo_loc_expression})   in (\ref{T_average_first_step}), and write the result as:
\begin{multline}\label{T_average_3_types_of_terms}
\frac{1}{|\calP|^2}\sum_{\substack{p,q\in \calP \\ p\neq q}} T(p,q;B) \ll \frac{B^n}{P^{2n}}\left( \sum_{j=0}^{n-1} S^{j}_{\glo,\glo} + \sum_{1 \leq j \leq k \leq n-1} S^{j,k}_{\loc,\loc} + \sum_{0 \leq j < k \leq n-1} S^{j,k}_{\glo,\loc} \right).
\end{multline}
Now we will bound each of these three types of contributions, starting with the easiest case, which is the local/local case.

\subsection{Bounding the local/local contribution}
Fix a pair of indices $1 \leq j \leq k \leq n-1$ and recall the expression (\ref{S_j_loc_loc_expression}) for the corresponding local/local contribution. 
Recall that $V_j = V((G_{j,s})_s)$ is defined as the common vanishing set of  homogeneous polynomials $\{G_{j,s}\}_{s \in I_j}$ as $s$ varies over a finite index set $I_j$. Then if $\bu\not\in V_j(\Z)$,  it must be that $G_{j,s}(\bu)\neq 0$ for some $s\in I_j.$ On the other hand, since $\bu\in V_j(\F_p)$, it must be that $p| G_{j,s}(\bu)$ (indeed for all $s \in I_j$). Hence,   for any $\bu\not\in V_j(\Z)$, 
\beq\label{small_number_of_p}
\sum_{\substack{p\in [P,2P]\\ \bu\in V_j(\F_p)}}1 \leq   \max_{s \in I_j, G_{j,s}(\bu) \neq 0} \nu_0(G_{j,s}(\bu)),
\eeq
in which the maximum is over a set with at least one element.
Here $\nu_0(m)$ denotes the number of distinct prime divisors of an integer $m$ in the range $[P,2P]$.  Multiplying over the distinct prime divisors of the integer $G_{j,s}(\bu)\neq 0$ that lie in the range $[P,2P]$, one sees that $(\prod p) |G_{j,s}(\bu)$ so that 
\beq\label{nu_0_bound}
P^{\nu_0(G_{j,s}(\bu))} \leq |G_{j,s}(\bu)| \ll_n \|G_{j,s}\||\bu|^{\deg_{G_{j,s}}},
\eeq
and this provides an upper bound for $\nu_0(G_{j,s}(\bu))$. 
Consequently, applying this in (\ref{S_j_loc_loc_expression}) leads to
\[
S_{\loc,\loc}^{j,k} \ll_{n,\deg F} \frac{1}{|\mathcal{P}|^2} P^{n+j/2+k/2} \deg(V_j) \sum_{q\in \calP}\sum_{\substack{\bu\in \Z^n \\ \bu\in V_k(\F_q)}} \frac{\log(1+|\bu|)+\log \|V_j\|}{\log P} \prod_{i=1}^n \left(1+\frac{|u_i|B}{P^2}\right)^{-M} ,
\]
in which we used positivity to omit the condition $\bu \not\in V_j(\Z)$. We have also adopted the convention that 
\[ \|V_j\| :=\max_{s \in I_j} \|G_{j,s}\|\]
so that by Theorem \ref{thm_strat_details}(\ref{thm_strat_coeff}), $\log \|V_j\| \ll_{n,\deg F}\log \|F\|$. Also by (\ref{thm_strat_deg}) of that theorem, $\deg V_j \ll_{n,\deg F} 1$.  
Next, we apply Lemma \ref{lemma_weight_count} to bound this smoothed sum over points of bounded height $\ll P^2/B$ whose reduction modulo $q$ lie in a variety of dimension at most $n-k$ in $\mathbb{A}^n(\F_q)$. Note that for $j \leq k$ we pre-processed the local/local sum in precisely the form (\ref{S_j_loc_loc_expression})   so that of the two options, the highest codimension $k$ now appears in this sum. Under the condition (\ref{P_size_half}) on the size of $P$, $1 \ll   P^2/B=o(P)$.
By Proposition \ref{prop_Xu_count}, the hypothesis (\ref{hypothesis_in_box_p}) of Lemma \ref{lemma_weight_count}  holds with $\kappa=n-k$ for $V_k(\F_q)$. We may apply the case $Y:=P^2/B$  in Lemma \ref{lemma_weight_count}; then $Y < q^\sig$ for some $\sig<1$ since $q \in [P,2P]$ and (\ref{P_size_half}) controls the size of $P$. This shows
\beq \label{local_local_sum_u}
\sum_{\substack{\bu\in \Z^n \\ \bu\in V_k(\F_q)}} \prod_{i=1}^n \left(1+\frac{|u_i|B}{P^2}\right)^{-M}(\log (1+|\bu|)) \ll \left( \frac{P^2}{B}\right)^{n-k+\ep}.
\eeq
Finally, we sum trivially over $q$. 
The conclusion is that for $1 \leq j \leq k \leq n-1$, for any $\ep>0$,
 \begin{align}
S_{\mathrm{loc,\loc}}^{j,k}
& \ll_{\ep,n,\deg F}(\log \|F\|) \frac{1}{|\mathcal{P}|^2} P^{n+j/2+k/2} \frac{|\calP|}{\log P} \left(\frac{P^2}{B}\right)^{n-k+\ep}\nonumber\\
& \ll_{\ep,n,\deg F}(\log \|F\|)   P^{n+k-1}   \left(\frac{P^2}{B}\right)^{n-k+\ep}. \label{S_local_bound}
\end{align}
By recording the second inequality, we may reduce the local/local terms to only considering the diagonal contribution from $k=j$, which we have shown via the first inequality are the largest terms.

\subsection{Bounding the global/global contribution, degenerate case}\label{sec_glo_glo_degenerate}
We next consider the global/global contribution.
Let $1 \leq j \leq n-1$ be fixed. Within the expression (\ref{S_j_glo_glo_expression}) for $S^{j}_{\glo,\glo}$, we now further divide the irreducible components $X_{j,s}$ of $V_j$ into two cases.  Recalling the discussion in \S \ref{sec_degenerate_primes}, if $X_{j,s}$ lies in a (proper) linear subspace of $\mathbb{A}^n(\Z)$,   we say $X_{j,s}$ is a degenerate case, and otherwise we say that $X_{j,s}$ is a nondegenerate case.
Given $j$, we decompose \[S_{\glo,\glo}^j = S_{\glo,\glo}^{j,\mathrm{deg}} +  S_{\glo,\glo}^{j,\mathrm{non}}.\]
Here, $S^{j,\mathrm{deg}}_{\glo,\glo}$ (respectively $S^{j,\mathrm{non}}_{\glo,\glo}$) is given by the terms in the sum over $s$ on the right-hand side of (\ref{S_j_glo_glo_expression}) restricted to those components such that $X_{j,s}$ is a degenerate case (respectively, nondegenerate case). (This decomposition does not distinguish those $s'$ for which $X_{j+1,s'}$ is degenerate or nondegenerate. Also, we need only consider whether $X_{j,s}$ is degenerate over $\Z$ (and consequently over $\F_p$ for all primes $p$); we do not need to handle separately any ``locally degenerate'' cases.) We will return to the non-degenerate contributions in \S \ref{sec_glo_glo_nondegenerate}; here we focus on $S_{\glo,\glo}^{j,\deg}$.
We will apply the  second moment bound from  Proposition \ref{prop_strongly_genuine_second_moment},   in the form of  (\ref{second_moment_apply}), which is adapted  to a linear transformation corresponding to the degenerate component; this specifically relies on the hypothesis of Theorem \ref{thm_main_options} that $F_L$ is strongly $n$-genuine for each $L \in \mathcal{L}(F)$.

As seen in the expression (\ref{S_j_glo_glo_expression}), the contribution of a degenerate irreducible component  $X_{j,s}$ to $S_{\mathrm{glo},\mathrm{glo}}^{j,\mathrm{deg}}$ is at most:
\begin{align}
  &  \max_{p \neq q \in \calP} |I_j| |I_{j+1}|^2 
    \sum_{\bu\in X_{j,s}(\Z)} \prod_{i=1}^n \left(1+\frac{|u_i|B}{P^2}\right)^{-M} |S(\overline{q}\bu,p)||S(\overline{p}\bu,q)| \label{glo_glo_degenerate}\\
     &\ll_{n,\deg F}  \max_{p \neq q \in \calP} \sum_{\substack{0\leq b \leq \lceil\log_2 p^n\rceil\\M_p = 2^b}} \sum_{\substack{0 \leq c \leq \lceil\log_2 q^n \rceil\\M_q = 2^c}}   \sum_{\substack{\bu\in X_{j,s}(\Z) \\ M_p < |S(\overline{q}\bu,p)| \leq 2M_p \\ M_q \leq |S(\overline{p}\bu,q)| < 2M_q}} \prod_{i=1}^n \left(1+\frac{|u_i|B}{P^2}\right)^{-M}|S(\overline{q}\bu,p)||S(\overline{p}\bu,q)| \nonumber \\
     &\ll_{n,\deg F} \max_{p \neq q \in \calP}  \sum_{M_p\nearrow p^n} \sum_{M_q\nearrow q^n} M_pM_q \sum_{\substack{\bu\in X_{j,s}(\Z) \\ |S(\overline{q}\bu,p)|>M_p  \\ |S(\overline{p}\bu,q)|>M_q}}   \prod_{i=1}^n \left(1+\frac{|u_i|B}{P^2}\right)^{-M}, \nonumber
     \end{align}
  in which the sums over $M_p, M_q$ denote dyadic sums, each with $\ll_n \log P$ summands. (We have expressed these as dyadic sums over powers of $2$, but it would be equally possible to write them as as a $p$-adic and $q$-adic sum, respectively. Since $X_{j,s} \subset V_j$, we do not need to take $M_p$ beyond $p^{(n+j)/2}$, but the additional terms here are negligible in the dyadic sum, and similarly for $M_q$.) We have also used the fact (\ref{thm_strat_num_comp}) from Theorem \ref{thm_strat_details} that the number of components $|I_j| \ll_{n,\deg F} 1$ for all $j$. 
Fix a small $\al>0$ (to be chosen later) and define the box
\[ C_\al = [-(P^2/B)^{1+\al}, (P^2/B)^{1+\al}]^n \subset \R^n.\]
 Decompose the above expression for the contribution to $S^{j,\mathrm{deg}}_{\mathrm{glo},\mathrm{glo}}$ from $X_{j,s}$ into two terms $\mathrm{I} + \mathrm{II}$,   denoting the contribution from summing over $\bu \in C_\al$ and over $\bu \not\in C_\al$, respectively.  First,
\begin{align}
   \mathrm{I} 
    & \ll_{n,\deg F}  \max_{p \neq q \in \mathcal{P}} \sum_{M_p\nearrow p^n} \sum_{M_q\nearrow q^n}  M_pM_q  \nonumber \\
     & \qquad   \cdot \#\left(\{ |\bu| \in X_{j,s}(\Z)\cap C_\al: |S(\overline{q}\bu,p)|>M_p\} \cap \{ |\bu| \in X_{j,s}(\Z)\cap C_\al: |S(\overline{p}\bu,q)|>M_q\}\right) \nonumber \\
    &\ll_{n,\deg F}    \log P \max_{p \neq q \in \mathcal{P}}\sum_{M_p\nearrow p^n} M_p^2 \cdot \#\{|\bu|\in X_{j,s}(\Z)\cap C_\al: |S(\overline{q}\bu,p)|>M_p\}\nonumber \\
      &\ll_{n,\deg F}    \log P \max_{p \neq q \in \mathcal{P}}\sum_{M_p\nearrow p^n} M_p^2 \cdot \#\{|\bu|\in X_{j,s}(\F_p): |S(\overline{q}\bu,p)|>M_p\} \nonumber \\
    &\ll_{n, \deg F}   P^{2n-1} (\log P)^2.\label{S_global_lin_I}
\end{align}
In the first inequality, we assume without loss of generality that we consider terms with $M_p \geq M_q$ (an analogous argument holds in the symmetric case). Then by positivity we can bound the size of the intersection by the size of the superlevel set defined with respect to $M_p$ (and then sum trivially over the $\ll \log P$ summands for $M_q$).  In the second inequality, we use that $P= B^\rho$ for some $\rho<1$ by assumption (\ref{P_size_half}). In particular, given $\rho$, by choosing $\al \ll_\rho 1$ sufficiently small, we can ensure that $(P^2/B)^{1+\al} \ll p$ (at least for $B \gg_\rho 1$) so that $X_{j,s}(\Z)\cap C_\al$ may be identified with a subset of $X_{j,s}(\F_p)$. To achieve  the last inequality, we plug in (\ref{second_moment_apply}). (By the choice of the sieving set in \S \ref{sec_sieving_set}, all primes in $\mathcal{P}$ have been chosen to avoid the possible exceptional primes for which (\ref{second_moment_apply}) does not apply.)

Now we turn to the term $\mathrm{II}$, which is bounded by 
\[ 
\mathrm{II} \ll_{n,\deg F}   P^{2n} (\log P)^2 \sum_{\substack{\bu\in \Z^n \\ \bu \not\in C_\al}}   \prod_{i=1}^n \left(1+\frac{|u_i|B}{P^2}\right)^{-M} \ll    P^{2n} (\log P)^2 (P^2/B)^{n - \al (M-2)}.
\]
Here we have applied that $P=B^\rho$ for $\rho > 1/2$ so that $P^2/B \gg 1$, and argued 
as in (\ref{C_ep_estimate}) of  Lemma \ref{lemma_weight_count}, (with $Y=P^2/B\gg 1$, $A=0$, and $\ep =\al$). Recall that we may take $M$ as large as we like. Our ultimate choice of $P=B^\rho$ in \S \ref{sec_combining_estimates} will choose $\rho = n/(n+1)$, so that the appropriate choice of $\al$ above depends only on $n$. (Any $\al < 1/(n-1)$ will do.) Thus we may assume from now on that $M=M(n)$ is sufficiently large that $n-\al(M-2)<0$. Then the contribution of $\mathrm{II}$ can be made arbitrarily small, and in particular is dominated by the contribution from $\mathrm{I}$, for all $B \gg_n 1$.
In total, the contribution of $X_{j,s}$ to $S_{\glo,\glo}^{j,\mathrm{deg}}$ is dominated by (\ref{S_global_lin_I}). Summing over all   $X_{j,s}$ in the degenerate case (at most $|I_j| \ll_{n,\deg F} 1$ terms) leads to 
\beq\label{S_global_degenerate_total}
S_{\glo,\glo}^{j,\mathrm{deg}}  \ll_{n, \deg F} P^{2n-1} (\log P)^2.
\eeq

\subsection{Bounding the global/global contribution, nondegenerate case}\label{sec_glo_glo_nondegenerate}
We now focus on the contribution $S_{\glo,\glo}^{j,\mathrm{non}}$ to (\ref{S_j_glo_glo_expression}) for each $1 \leq j \leq n-1$.
First, we claim that when $X_{j,s} \subset \mathbb{A}^n(\Z)$ is a nondegenerate irreducible component of $V_j$, we may assume $\dim X_{j,s}>1$ (or equivalently $\codim X_{j,s} \leq n-2$).
In other words, we claim that if $\dim(X_{j,s})= 1$, then $X_{j,s}$ must lie in a (proper) linear subspace of $\mathbb{A}^n(\Z)$, and hence fall into the degenerate case. Indeed, since $X_{j,s}$ is homogeneous, the projective dimension of such $X_{j,s}$ is zero. So, when viewed projectively, $X_{j,s}$ is a finite set of points; as such, when viewed in affine space, the irreducible component $X_{j,s}$ will be a line and thus linear (and degenerate).  
The requirement $\codim X_{j,s} \leq n-2$, combined with the fact from Theorem \ref{thm_strat_details} that $\codim X_{j,s} \geq j$, thus imposes $j \leq n-2$ from now on. Note that if $n=2$, only $j=0$ satisfies $j \leq n-2$, and the global/global contribution for $V_0$ has already been bounded in (\ref{S_global_zero}). Thus for the rest of this section, we may assume $n \geq 3$ (and if $n=2$, the sums over $j\leq n-2$ depicted below are considered to be  empty).

Second, if $X_{j,s}$ is nondegenerate, it must be that $\deg X_{j,s} \geq 2$.
Third, fix an index $j \in \{1,\ldots, n-2\}$.
Let $X_{j,s}$ be an irreducible component of the $j$-th stratum $V_j$ with codimension at least $j$ in $\mathbb{A}^n(\Z)$, and degree denoted $\deg X_{j,s}$. 
We recall the following relation between the dimension and degree of a nondegenerate variety.
\begin{prop}\label{prop_anonymous}
Let $W\subset\mathbb{P}^{N}$ be an irreducible nondegenerate variety of degree $e$ and dimension $m$. Then $e> N-m$.
\end{prop}
We quote this from \cite[Cor. 18.12]{Har92}, but we remark that it seems to have appeared originally in the note \cite{Ano57}, and soon after in Swinnerton-Dyer  \cite[Lemma 3 and Corollary]{Swi73}. 
As a result of the proposition, if  $X_{j,s}$ is nondegenerate, then $\deg X_{j,s} > \codim X_{j,s} \geq j$.   Hence  all that remains is to bound the contribution of  $X_{j,s}$ to $S_{\glo,\glo}^{j,\mathrm{non}}$ under the assumption that 
$
1 \leq j \leq \min\{\deg X_{j,s}-1,n-2\}.
$

With these restrictions, the total nondegenerate contribution to (\ref{S_j_glo_glo_expression}) for a fixed $1 \leq j \leq n-2$ is bounded by:
\beq\label{S_glo_glo_bound_before_Pila}
S_{\glo,\glo}^{j,\mathrm{non}}
\ll_{n,\deg F} |I_j||I_{j+1}|^2\sum_{\substack{s \in I_j\\j \leq \min\{\deg X_{j,s}-1,n-2\}}}P^{n+j}    \sum_{\substack{\bu \in \Z^n \\ \bu \in X_{j,s}(\Z) }} \prod_{i=1}^n \left(1+\frac{|u_i|B}{P^2}\right)^{-M}.
    \eeq
    We have obtained this expression by considering the nondegenerate contribution to the simplified form (\ref{S_j_glo_glo_expression_alternate}).
We can achieve a bound that is uniform in the coefficients of the defining polynomials of $X_{j,s}$, by applying Pila's bound (Theorem \ref{thm_Pila_multidim}). (See Remark \ref{remark_DGC_app} for why a more straightforward approach is not yet available.) Fix a component $X_{j,s}$ considered in the sum, and apply Lemma \ref{lemma_weight_count} to bound the smoothed sum over $\bu \in X_{j,s}(\Z)$, with $Y=P^2/B \gg 1$. By Pila's bound, the hypothesis (\ref{hypothesis_in_box}) holds with $\kappa = n-j-1+(\deg X_{j,s})^{-1}+\ep$ for any $\ep>0$, and an implicit constant depending only on $n,j,\deg X_{j,s},\ep$. By Theorem \ref{thm_strat_details} (ii), $\deg X_{j,s}\ll_{n,\deg F}1$.
Consequently,  
\beq\label{S_global_nondegenerate_bound}
S_{\glo,\glo}^{j,\mathrm{non}} \ll_{n,\deg F,\ep} \delta_{n\geq 3}\sum_{\substack{s \in I_j} }
   \delta_{j,\deg(X_{j,s})}  P^{n+j}  (P^2/B)^{n-j-1+(\deg X_{j,s})^{-1}+\ep} 
   \eeq
 under the definitions $\delta_{n \geq 3} =1$ if $n \geq 3$ and vanishes otherwise, and 
\beq\label{dfn_delta}
\delta_{j,e} = \begin{cases} 1 & \text{if $j \leq \min \{e-1,n-2\}$ and $e \geq 2$ } 
\\
  0 & \text{otherwise.}
\end{cases}
\eeq

    \begin{rem}\label{remark_DGC_app}
    Salberger has proved   the Dimension Growth Conjecture, and the uniform version of the conjecture for degree at least 4 \cite[Thm. 0.1, 0.3]{Sal23}.
Let $X\subset \mathbb{P}^{n-1}_\Q$ be an integral projective variety defined over $\Q$. Then for $\deg(X)\geq 2$,  for any $\varepsilon>0$, $$\#\{\bfx\in X(\Z): \|\bfx\|\leq B\} = O_{X,\varepsilon}(B^{\dim_{\mathrm{proj}}(X)+\varepsilon}).$$
 For $\deg(X) \geq 4$,  
 \beq\label{DGC_uniform}
 \#\{\bfx\in X(\Z): \|\bfx\|\leq B\} = O_{\deg(X),n,\varepsilon} (B^{\dim_{\mathrm{proj}}(X)+\varepsilon}).
 \eeq
This uniform bound was previously obtained  for hypersurfaces of degree 2 in \cite{HB02} and for geometrically integral hypersurfaces of  degree 2 or degree at least 6 in  \cite{BHBS06}. 
 
 In (\ref{S_global_nondegenerate_bound}), it would be  simpler to apply the truth of the Dimension Growth Conjecture, which would show that for any $\varepsilon>0$, $$\sum_{\substack{\bu\in \Z^n \\ \bu\in X_{j,s}(\Z)}} \prod_{i=1}^n \left(1+\frac{|u_i|B}{P^2}\right)^{-M} \ll_{X_{j,s},\varepsilon} \left(\frac{P^2}{B}\right)^{n-j-1+\varepsilon}$$
for any $\ep>0$. This is a better exponent than   in (\ref{S_global_nondegenerate_bound}), and leads to a bound $S_{\mathrm{global}}^{j,\mathrm{non}}
\ll_{V_j,\ep} P^{n+j}(P^2/B)^{n-j-1+\ep}$.  However, the implicit constant is only known to be uniform in the coefficients of the defining functions of $X_{j,s}$ if $\deg X_{j,s}\neq 3$, while we cannot rule out that $\deg X_{j,s}=3$ may occur.
  (For $\deg(X)=3$, the current record of \cite[Thm. 0.3]{Sal23} replaces the right-hand side of (\ref{DGC_uniform}) by $ O_{n,\varepsilon} (B^{\dim_{\mathrm{proj}}(X)-1+2/\sqrt{3}+\varepsilon})$,   
  with an exponent not strong enough for our application.) Thus we proceed with the bound (\ref{S_global_nondegenerate_bound}), and we will use Proposition \ref{prop_anonymous} to argue case by case that the additional $(\deg X_{j,s})^{-1}$ in the exponent is not detrimental in the proof of 
  Theorem \ref{thm_main_options}. 

\end{rem}

\subsection{Bounding the global/local contribution}\label{sec_glo_loc}
It only remains to bound the global/local contribution to (\ref{T_average_3_types_of_terms}), and this uses a combination of the above strategies. 
Recall the expression (\ref{S_j_glo_loc_expression}) for $S^{j,k}_{\glo,\loc}$ for each pair $j,k$ with $0\leq j< k \leq n-1$. We will prove two bounds for $S^{j,k}_{\glo,\loc}$, and choose the most advantageous depending on the pair $j,k$ of indices. 

For the first bound, fix indices $s \in I_j$ and $t \in I_k$. (By Theorem \ref{thm_strat_details}(\ref{thm_strat_num_comp}) there are $\ll_{n,\deg F} 1$ such tuples of indices.) The contribution of the corresponding term to $S^{j,k}_{\glo,\loc}$ is 
\beq\label{S_j_glo_loc_expression_stst}
\frac{1}{|\calP|^2}
\sum_{\substack{\bu\in \Z^n\\ \bu \in X_{j,s}(\Z)\backslash V_{j+1}(\Z)  } } \sum_{\substack{p \neq q \in \mathcal{P} \\ \bu \in X_{j,s}(\F_p)\backslash V_{j+1}(\F_p) \\ \bu \in X_{k,t}(\F_q)\backslash V_{k+1}(\F_q)}} |S(\overline{q}\bu,p)| |S(\overline{p}\bu,q)| \prod_{i=1}^n \left(1+\frac{|u_i|B}{P^2}\right)^{-M} .
\eeq
We apply the bounds $|S(\overline{q}\bu,p)|\ll P^{\frac{n+j}{2}}$ and  $|S(\overline{p}\bu,q)| \ll P^{\frac{n+k}{2}}$ from Theorem \ref{thm_strat_details}.
Since $\bu \in X_{j,s} (\Z)$, it follows that $\bu \in X_{j,s}(\F_p)$ for all $p \in \mathcal{P}$ so we sum trivially over such $p$ and gain no advantage there.  (It is tempting to fix $\bu$ and use the fact that there can be few $q$ for which $\bu$ is locally in $X_{k,t}(\F_q)$ but not in $X_{k,t}(\Z)$, but then the remaining sum over $\bu \in X_{j,s}(\Z)$ is only over a set of codimension $j$. So we currently argue a different way.)
We first sum over $\bu$ and then sum over $q$ in the remaining expression, which then is bounded above by 
\[
\ll_{n,\deg F} \max_{p \in \mathcal{P}}  \frac{1}{|\calP|} P^{\frac{n+j}{2}} P^{\frac{n+k}{2}}
\sum_{q \in \mathcal{P}} \sum_{\substack{\bu\in \Z^n\\ \bu \in X_{k,t}(\F_q)  } }   \prod_{i=1}^n \left(1+\frac{|u_i|B}{P^2}\right)^{-M} .\]
Here we have used positivity in order to omit all conditions on $\bu$ except for the condition $\bu \in X_{k,t}(\F_q)$, which we retained because $X_{k,t}$ has the higher codimension $k$. 
Now for each $q$ this sum over $\bu$ can be treated exactly as we treated the (slightly larger) sum (\ref{local_local_sum_u}) in the local/local case. Finally, we sum trivially over $q$. This leads to the first upper bound 
\beq\label{S_j_glo_loc_bound_small_j}
S^{j,k}_{\glo,\loc} \ll_{n,\deg F,\ep}  P^{n+j/2+k/2} \left( \frac{P^2}{B} \right)^{n-k+\ep}.
\eeq
This will ultimately be strong enough in two scenarios: if $j=0$ and $1 \leq k \leq n-1$, and if $1 \leq j \leq n-1$ and $k \geq j+3$.

Thus we proceed now to consider a different bound (which we will apply for all index pairs for which $1 \leq j \leq n-1$ and $k=j+\kappa$ with $\kappa=1$ or $2$). Fix $j,k$ with $j<k$.
Now we consider whether each irreducible component $X_{j,s}\subset \mathbb{A}^n(\Z)$ of $V_j$ is degenerate     or not, and we accordingly decompose 
\[ 
S^{j,k}_{\glo,\loc} = S^{j,k,\mathrm{deg}}_{\glo,\loc}+S^{j,k,\mathrm{non}}_{\glo,\loc}.
\]
Suppose first that $X_{j,s}$ is degenerate: then by summing over all $p,q$ trivially, the contribution of $X_{j,s}$ to $S^{j,k}_{\glo,\loc} $ is bounded above by:
\[
 \leq |I_j| |I_{j+1}|^2\max_{p \neq q \in \mathcal{P}}
 \sum_{\substack{\bu\in \Z^n\\ \bu \in X_{j,s}(\Z)  } }   |S(\overline{q}\bu,p)| |S(\overline{p}\bu,q)| \prod_{i=1}^n \left(1+\frac{|u_i|B}{P^2}\right)^{-M} .
\]
This is recognized as (\ref{glo_glo_degenerate}) and can be dominated by $\ll_{n,\deg F} P^{2n-1}(\log P)^2$, just as that expression was. (This also relies on the hypothesis of Theorem \ref{thm_main_options} that $F_L$ is strongly $n$-genuine for all $L \in \mathcal{L}(F)$.)  Thus we can record the contribution to $S^{j,k}_{\glo,\loc}$ from all irreducible components $X_{j,s}$ that fall in the degenerate case as 
\beq\label{S_j_glo_loc_degenerate_bound}
S^{j,k,\mathrm{deg}}_{\glo,\loc} \ll_{n,\deg F} P^{2n-1}(\log P)^2. 
\eeq
In particular, for all pairs $1 \leq j<k \leq n-1$, this contribution is dominated by the contribution (\ref{S_global_degenerate_total}) recorded for $S^{j,\mathrm{deg}}_{\glo,\glo}$.  

Now fix an $s$ for which $X_{j,s}$ is  nondegenerate. By the same arguments presented at the beginning of \S \ref{sec_glo_glo_nondegenerate}, we need only consider the case that $1 \leq j \leq \min\{ \deg(X_{j,s})-1, n-2\}$. (If $n=2$ this imposes $j=0$, which was already handled before, so from now on we may assume $n \geq 3$.) Fix any index $t \in I_k$  (of which there are $\ll_{n,\deg F} 1$ choices). The contribution we must bound is 
\[\frac{1}{|\calP|^2}
\sum_{\substack{\bu\in \Z^n\\ \bu \in X_{j,s}(\Z)\backslash V_{j+1}(\Z)  } } \sum_{\substack{p \neq q \in \mathcal{P} \\ \bu \in X_{j,s}(\F_p)\backslash V_{j+1}(\F_p) \\ \bu \in X_{k,t}(\F_q)\backslash V_{k+1}(\F_q)}} |S(\overline{q}\bu,p)| |S(\overline{p}\bu,q)| \prod_{i=1}^n \left(1+\frac{|u_i|B}{P^2}\right)^{-M} .\]
We apply the bounds $|S(\overline{q}\bu,p)|\ll P^{\frac{n+j}{2}}$ and  $|S(\overline{p}\bu,q)| \ll P^{\frac{n+k}{2}}$ from Theorem \ref{thm_strat_details}.
We again sum trivially over all $p \in \mathcal{P}$, since $\bu$ lies in $X_{j,s}(\Z)$.  On the other hand, observe that using the local property for $q$ will preserve a savings of $|\mathcal{P}|^{-1}$ for the sum over $q$. Recall that $X_{k,t}$ is a homogeneous irreducible component, so it is defined by some homogeneous polynomial $G_{k,t}$. Since $\bu \not\in X_{k,t}(\Z)$ (since $\bu \not\in V_{j+1}(\Z)$ and $k > j$)   but $\bu \in X_{k,t}(\F_q)$, we may conclude that $G_{k,t}(\bu) \neq 0$ but $q | G_{k,t}(\bu)$. Thus arguing as in (\ref{small_number_of_p}) and (\ref{nu_0_bound}) to bound the number of $q$ arising in the sum, the contribution expressed above can be bounded by
\[\ll_{n,\deg F} \deg(V_k) \frac{1}{|\calP|} P^{\frac{n+j}{2}} P^{\frac{n+k}{2}} 
\sum_{\substack{\bu\in \Z^n\\ \bu \in X_{j,s}(\Z) } } \frac{(\log(1+|\bu|) + \log \|V_k \|)}{\log P}  \prod_{i=1}^n \left(1+\frac{|u_i|B}{P^2}\right)^{-M} .\]
By Theorem \ref{thm_strat_details} (\ref{thm_strat_coeff}), $\log \|V_k\| := \max_{s \in I_k}\log \|G_{k,s}\| \ll_{n,\deg F}\log  \|F\|.$
Furthermore,  we recognize that the sum over $\bu$ in the expression above takes the same form as the sum over $\bfu$ in (\ref{S_glo_glo_bound_before_Pila}) previously encountered in the global/global nondegenerate case, aside from an extra factor of $\log (1+|\bu|)$ here. We proceed as we did to bound (\ref{S_glo_glo_bound_before_Pila}) by applying Pila's bound (Theorem \ref{thm_Pila_multidim}) as an input to the smoothed counting lemma (Lemma \ref{lemma_weight_count}) with $\kappa = n-j-1+(\deg X_{j,s})^{-1}+\ep$.
This shows that the contribution to $S_{\glo,\loc}^{j,k}$ from all nondegenerate components $X_{j,s}$ is
 \beq\label{S_glo_loc_nondegenerate_bound}
S_{\glo,\loc}^{j,k,\mathrm{non}}\ll_{n,\deg F,\ep} (\log \|F\|)   \delta_{n \geq 3} \sum_{s \in I_j} \delta_{j,\deg(X_{j,s})} P^{n+j/2+k/2-1}  (P^2/B)^{n-j-1+\deg(X_{j,s})^{-1}+\ep} .
   \eeq
 In particular, for  $k \leq j+2$, note that $P^{n+j/2+k/2-1} \leq P^{n+j}$, in which case the summand  above is dominated by the corresponding summand recorded in (\ref{S_global_nondegenerate_bound}) of $X_{j,s}$ for $S^{j,\mathrm{non}}_{\glo,\glo}$. But note   the factor of $\log \|F\|$ here.

\section{Combining the estimates}\label{sec_combining_estimates}
We now complete the proof of Theorem \ref{thm_main_options} via the polynomial sieve. We input the bounds for the global/global, local/local, and global/local cases into  (\ref{T_average_3_types_of_terms}), which we summarize in the form 
\[
\frac{1}{|\calP|^2}\sum_{\substack{p,q\in \calP \\ p\neq q}} T(p,q;B) \ll T_{\loc,\loc}+ T_{\glo,\glo} + T_{\glo,\loc}.\]
Recall that we consider $F(Y,\bX)$ of the form (\ref{F_dfn_intro}) for given degree parameters $m,d, k_1,\ldots,k_d$. 
Inputting the bounds proved in (\ref{S_global_zero}),
(\ref{S_local_bound}),(\ref{S_global_degenerate_total}), and (\ref{S_global_nondegenerate_bound}), as well as (\ref{S_j_glo_loc_bound_small_j}), (\ref{S_j_glo_loc_degenerate_bound}) and (\ref{S_glo_loc_nondegenerate_bound}), we can summarize our conclusions as follows. For the local/local terms,

\begin{align*}
T_{\loc,\loc} 
& \ll  \frac{B^n}{P^{2n}} \sum_{1 \leq j \leq k \leq n-1} S^{j,k}_{\loc,\loc}
\ll_{n,\deg F,\ep} (\log \|F\|) \frac{B^n}{P^{2n}}  \sum_{j=1}^{n-1}  P^{n+j-1}\left( \frac{P^2}{B}\right)^{n-j+\ep}\\
& \ll  (\log \|F\|)  P^{n-1}  B^\ep  \sum_{j=1}^{n-1}  (B/P)^j \\
    & \ll_{n,\deg F,\ep} (\log \|F\|) B^{n-1+\ep}.
\end{align*}
In the last inequality we have used that $B/P \gg 1$ to dominate the  sum by the contribution from $j=n-1$.

For the global/global term, recall the definition of $\delta_{j,e}$ in (\ref{dfn_delta}). 
Under the hypotheses of Theorem \ref{thm_main_options},  
\begin{align*}
T_{\glo,\glo} & \ll \frac{B^n}{P^{2n}}\left( P^n(P^2/B)^{n+\ep}+\sum_{j=1}^{n-1}S^{j,\mathrm{deg}}_{\glo,\glo} + \sum_{j=1}^{n-2}S^{j,\mathrm{non}}_{\glo,\glo} \right) \\ 
&\ll  \frac{B^n}{P^{2n}}\Bigg( P^n(P^2/B)^{n+\ep} +P^{2n-1}(\log P)^2 \\
&\qquad \qquad +\; \delta_{n\geq 3} \sum_{j=1}^{n-2}\sum_{s \in I_j } \delta_{j,\deg X_{j,s}} P^{n+j}(P^2/B)^{n-j-1+(\deg X_{j,s})^{-1}+\ep} \Bigg)\\
 &\ll    B^\ep ( P^n+ B^nP^{-1} + \delta_{n\geq 3} P^{n}\sum_{j=1}^{n-2}\sum_{\substack{s \in I_j\\(e :=\deg X_{j,s}\geq 2)}} \delta_{j,e}  \frac{B^{j+1-1/e}}{P^{j+2-2/e}} )\\
&\ll B^\ep ( P^n + B^nP^{-1} +  \delta_{n\geq 3}P^{n} \Sigma ) ,
\end{align*}
say, where $\Sig$ denotes the double sum over $j$ and $s$. The implicit constants depend on $n, \deg F,\ep$.
 We now analyze the sum $\Sig$ when $n \geq 3$ (for $n=2$ it is empty). 
\begin{lem} \label{lemma_Sigma_bound}
Suppose 
\beq\label{P_hypothesis_conclusion}
\text{$P=B^\rho$ with $n/(n+1)\leq \rho < 1$.}
\eeq
Then for $n \geq 3$,
 \beq\label{Sigma_bdd_1}
 \Sigma :=\sum_{j=1}^{n-2}\sum_{\substack{s \in I_j\\(e :=\deg X_{j,s} \geq 2)}}  \delta_{j,e}\frac{B^{j+1-1/e}}{P^{j+2-2/e}} \ll_{n,\deg F} 1.
 \eeq
 \end{lem}
 \begin{proof}
Let  $E$ denote the set of possible distinct values of  $e := \deg X_{j,s}$  that arise as $j,s$ vary over all pairs in this sum, for which  $\delta_{j,e}$  is nonzero (as defined in (\ref{dfn_delta})). Then $E$ is a finite set with  $E\subset [2,\infty)$  and $|E| \leq  \sum_j |I_j|\ll_{n,\deg F} 1$ (by Theorem \ref{thm_strat_details}(\ref{thm_strat_num_comp})). Possibly overcounting, we can write
 \[  \Sigma \leq  \sum_{e \in E} \sum_{j=1}^{\min\{e-1,n-2\}} |I_j|  \frac{B^{j+1-1/e}}{P^{j+2-2/e}}.\]
 (For clarity, we observe that  the condition $j\leq e-1$ may be imposed here because otherwise $\delta_{j,e} =0$ (and the underlying reason is due to Proposition \ref{prop_anonymous}).)
  For each $e \in E$, the inner sum over $j$ is bounded above by the contributions of its upper and lower endpoints, so we check the contribution at $j=1$, at $j=e-1$ (only if $e-1 \leq n-2$) and at $j=n-2$ (only if $e-1 \geq n-2$).  The contribution at $j=1$ is 
\[
\frac{B^{2-\frac{1}{e}}}{P^{3-\frac{2}{e}}}\ll 1 \Longleftrightarrow P\gg B^{\frac{2e-1}{3e-2}}.
\]
Note that $\frac{2e-1}{3e-2}$ is a decreasing function of $e \geq 2$ and evaluates to $3/4$ when $e=2$, so the relation above will hold as long as $\rho \geq 3/4$ (which we note holds in particular for $\rho = n/(n+1)$, for all $n \geq 3$). 
The contribution at $j=e-1$ (which we only consider if $e-1 \leq n-2$), is
    \[
\frac{B^{e-\frac{1}{e}}}{P^{e+1-\frac{2}{e}}}\ll 1 \Longleftrightarrow P\gg B^{\frac{e-\frac{1}{e}}{e+1-\frac{2}{e}}} = B^{\frac{e+1}{e+2}}.
    \]
  Since the function $\frac{e+1}{e+2}$ is increasing as a function of $e\geq 2$ and  $e-1\leq n-2$, the relation above will hold as long as $\rho \geq n/(n+1).$
    Finally, the contribution at $j=n-2$ (which we only consider if $e-1 \geq n-2$) is
    \[
\frac{B^{n-1-\frac{1}{e}}}{P^{n-\frac{2}{e}}}\ll 1 \Longleftrightarrow P\gg B^{\frac{n-1-\frac{1}{e}}{n-\frac{2}{e}}} = B^{\frac{e(n-1)-1}{en-2}}.
    \]
   Since the exponent on the right-hand side is decreasing as a function of $e\geq 2$ (for $n \geq 3$) and  since we may assume $e-1\geq n-2$ (as follows from Proposition \ref{prop_anonymous}), it follows that the relation above will hold if  $\rho \geq \frac{(n-1)^2-1}{(n-1)n-2}= \frac{n}{n+1}$.   
\end{proof}

And finally, we have proved 
\begin{align*}
T_{\glo,\loc} & \ll \frac{B^n}{P^{2n}}\sum_{0 \leq j < k \leq n-1} S^{j,k}_{\glo,\loc}  \\ 
&\ll  \frac{B^n}{P^{2n}}\left(\sum_{1 \leq k \leq n-1} S^{0,k}_{\glo,\loc}+ \sum_{\substack{1 \leq j <k\leq n-1 \\ k \geq j+3}}  S^{j,k}_{\glo,\loc} + \sum_{\substack{1 \leq j <k\leq n-1\\ k \leq j+2}}   S^{j,k}_{\glo,\loc} \right)
\end{align*}
As previously remarked below (\ref{S_j_glo_loc_degenerate_bound}) and (\ref{S_glo_loc_nondegenerate_bound}), the total contribution to $T_{\glo,\loc}$ from the last sum over $j,k$ (with $k \leq j+2$) is dominated by quantities we have already bounded when treating $T_{\glo,\glo}$, up to an additional factor $(\log \|F\|)$. Precisely, this sum contributes $\ll_{n,\deg F,\ep} (\log \|F\|)B^\ep (B^n P^{-1} + P^n)$ to $T_{\glo,\loc}$. The remaining consideration is to apply (\ref{S_j_glo_loc_bound_small_j}) to each term in the first two sums, which shows these two sums yield a total contribution to $T_{\glo,\loc}$ of at most
\begin{align*}
    &  \ll_{n,\deg F,\ep} \frac{B^n}{P^{2n}} \left( \sum_{1 \leq k \leq n-1}  P^{n+k/2} \left( \frac{P^2}{B} \right)^{n-k+\ep} + \sum_{\substack{1 \leq j <k\leq n-1 \\ k \geq j+3}} 
    P^{n+j/2 +k/2} \left( \frac{P^2}{B} \right)^{n-k+\ep} \right) 
    \\
     &  \ll   B^\ep \left( P^{n}\sum_{1 \leq k \leq n-1}  \left( \frac{B}{P^{3/2}} \right)^{k} + P^{n}\sum_{\substack{1 \leq j <k\leq n-1 \\ k \geq j+3}} 
    P^{(j-k)/2} \left( \frac{B}{P} \right)^{k} \right) .
\end{align*}
 We claim this is $\ll( P^n +B^{n-1})B^\ep$. In the first sum, it suffices to observe that $B/P^{3/2} \ll 1$, which holds  for all $n \geq 2$ under our current hypothesis (\ref{P_hypothesis_conclusion}) on the size of $P$.
 In the second sum, 
 \[ P^{n}\sum_{\substack{1 \leq j <k\leq n-1 \\ k \geq j+3}} 
    P^{(j-k)/2} \left( \frac{B}{P} \right)^{k}
    \ll_n P^{n}\sum_{2\leq k\leq n-1} 
    P^{-3/2} \left( \frac{B}{P} \right)^{k} \ll P^{n-3/2}(B/P)^{n-1} \ll B^{n-1}.
    \]
 The penultimate inequality uses $B \gg P$, and the last inequality is an over-estimate, but suffices.  In total, under the hypothesis (\ref{P_hypothesis_conclusion}), 
 \[ T_{\glo,\loc} \ll_{n,\deg F,\ep} (\log \|F\|) (B^n P^{-1}+ P^n)B^\ep.\]

\subsection{Conclusion of the proof of   Theorem \ref{thm_main_options}  }
We have shown that if $P=B^\rho$ with $n/(n+1) \leq \rho < 1$ and $n \geq 2$,  then under the hypotheses of Theorem \ref{thm_main_options},  
\[
\frac{1}{|\calP|^2}\sum_{\substack{p,q\in \calP \\ p\neq q}} T(p,q;B) \ll_{n,\deg F,\ep} ( \log \|F\|) (B^n P^{-1} + P^n)B^\ep.
\]
Inserting this in (\ref{N_bounded_by_T}) and taking the optimal choice $P = B^{\frac{n}{n+1}},$
the conclusion is that   for any $\varepsilon>0$: 
\beq\label{NFBstar}
N(F,B) \ll_{n,\deg F,\ep} ( \log \|F\|) B^{n-1+\frac{1}{n+1}+\varepsilon} .
\eeq
This consequence of the polynomial sieve has been shown under the conditions from Lemma \ref{lemma_poly_sieve} if $m \geq 2$ (or Lemma \ref{lemma_poly_sieve_cor} if $m=1$), and the lower bound (\ref{P_big_enough}) for $P$ when constructing the sieving set. These are all satisfied if
\[ P \gg_{n,\deg F} \max \{ (\log \|F\|)^3, (\log B)^3\}.\]
With $P=B^{n/(n+1)}$, the latter condition is satisfied for all $B \gg_{n,\deg F}  1$. The former condition is satisfied if $B \gg_{n,\deg F} (\log \|F\|)^6$, so (\ref{NFBstar}) holds for all such $B$. This completes the proof of Theorem \ref{thm_main_options}.

\begin{rem}[Dependence on GRH]\label{remark_GRH_number_of_apps}
The  conditional version of the polynomial sieve stated in Lemma \ref{lemma_poly_sieve_cor} for $m=1$  assumes that GRH holds for a large collection of Dedekind zeta functions (each associated to some $\bfk \in [-2B,2B]^n$, in the proof provided in \cite[\S 3.2]{BonPie24}). As $B \rightarrow \infty$, this formulation of  Lemma \ref{lemma_poly_sieve_cor} assumes GRH for the Dedekind zeta functions of potentially infinitely many fields. 

 However, an inspection of the proof in \cite[\S3.2]{BonPie24} shows that the following average statement (\ref{weaker_GRH_average}) (weaker than assuming GRH for all Dedekind zeta functions) suffices to prove  Lemma \ref{lemma_poly_sieve_cor} in the following form  (which would also imply its consequences in this paper, namely the $m=1$ case of all the main theorems in this paper). Let $G(Y,\bX)$ be monic in $Y$ and write $G(Y,\bX) = Y^D + f_1(\bX)Y^{D-1} + \cdots + f_d(\bX)$ for some $D \geq 2$. Following the notation of \cite[\S3.2]{BonPie24}, fix any $\bfk \in \Z^n$ with $G(Y,\bfk)=0$ solvable and $f_d(\bfk) \neq 0$. Let $g_{\bfk}(y)$ be an irreducible factor of the polynomial $\tilde{g}_{\bfk}$ defined by $G(Y,\bfk) = (Y-y_0)\tilde{g}_{\bfk}(Y)$. Let $F_{\bfk}$ denote the splitting field of $g_{\bfk}(Y)$ and $G_{\bfk}$   the Galois group of $F_{\bfk}$ over $\Q$, and $\pi_{\bfk}(Q)$  the number of primes in $[Q/2,Q]$ that split completely in $F_{\bfk}$. 
Suppose it is known that for a sufficiently small constant $c_D$, 
\beq\label{weaker_GRH_average}
 \sum_{\substack{\bfk\in [-2B,2B]^n\\f_{d}(\bfk)\neq 0\\ G(y,\bfk)=0 \text{ solvable in $\Z$}}}  \left|\pi_{\bfk}(Q) - \frac{1}{|G_{\bfk}|}\frac{Q}{\log Q}\right|    \\
\\ \leq c_D \frac{Q}{\log Q} \sum_{\substack{\bfk\in [-2B,2B]^n\\f_{d}(\bfk)\neq 0\\ G(y,\bfk) =0\text{ solvable in $\Z$}}} 1 ,
\eeq
for all $Q\gg B^{\kappa}$ (for a given $1/2 \leq \kappa \leq 1$) and $B \geq 1$.
Then for    $\mathcal{P}$ a set of primes $p \in [Q,2Q]$ with $|\mathcal{P}| \gg Q(\log Q)^{-1}$, a minor adaptation of the argument in \cite[\S 3.2]{BonPie24} proves that the smoothed count $\mathcal{S}(G,B)$ for $N(G,B)$ satisfies
\[ \mathcal{S}(G,B) \ll_{n,\deg G} B^nQ^{-1}\log Q +  \frac{1}{|\mathcal{P}|^2} \sum_{p \neq q \in \mathcal{P}} \left| \sum_{\bx \in \Z^n} (v_p(\bx)-1)(v_q(\bx)-1) W(\bx) \right|,\]
for all $Q \gg_{\deg G} \max\{ (\log \|G\|)^{\al_0}, B^\kappa\}$, for a choice of $\al_0>2$. Here as usual, for each $p$, $v_p(\bx) = \# \{ y \in \F_p: G(y,\bx)=0\}$ is the local count. If, as in this paper, the upper bound for the second term on the right-hand side of the sieve inequality is on the order of $\ll Q^{n+\ep}$, the relevant power is $\kappa \approx n/(n+1)$. Note that (\ref{weaker_GRH_average}) is   weaker than requiring GRH on average; it only requires that on average over $\bfk$ in the relevant set, the possible Siegel zeroes of the Dedekind zeta functions $\zeta_{F_{\bfk}}(s)$ do not contribute to the error term in the Chebotarev density theorem  in an uncontrolled way (see \cite[Eqn. (3.6)]{BonPie24}). 
\end{rem}

    \section{Construction of examples}\label{sec_examples}
    In this section, we construct examples necessary to verify genericity in \S \ref{sec_genericity}, examples of classes of polynomials to which each main theorem applies, and examples that verify   this paper recovers nearly all the previous literature on these problems.  
  
\subsection{Examples: $n$-genuine and strongly $n$-genuine polynomials}
To gain some intuition, we provide  the following examples: $\Q(X_1,X_2,\sqrt{X_1+X_2})$ is a 2-genuine extension of $\Q(X_1,X_2)$; $Y^2 - X_1-X_2$ is a $2$-genuine polynomial.
Next, $\Q(X_1,X_2,X_3, \sqrt{X_3+\sqrt{X_1+X_2}})$ is a 3-genuine extension of $\Q(X_1,X_2,X_3)$ but not a strongly 3-genuine extension; $(Y^2-X_3)^2 - X_1-X_2$ is  $3$-genuine polynomial but not a strongly $3$-genuine polynomial. 

Being $n$-genuine imposes that  $F$ has at least degree 2 in $Y$: for if $\deg_Y F(Y,X_1,\ldots,X_n) =1$ then $F(Y,X_1,\ldots,X_n)$ is not an $n$-genuine polynomial, since $\Q(X_1,\ldots,X_n)[Y]/F(Y,\bX))=\Q(X_1,\ldots,X_n)$ is not an $n$-genuine extension of $\Q(X_1,\ldots,X_n)$ (since it is trivial). Concordantly, if $\deg_Y F(Y,\bX)=1$ then trivially $B^n \ll N(F,B) \ll B^n$, since $F(Y,\bx)=0$ is solvable for all $\bx \in [-B,B]^n$. 

\begin{example}[Composita]\label{example_composita}
  Let $1\leq j <n$. Let $K$ be a $j$-genuine extension of $\Q(X_1,...,X_j)$ and $K'$ be a $(n-j)$-genuine extension of $\Q(X_{j+1},...,X_n)$. Define $L=KK'$ to be the compositum of fields. Then  $L$ is an $n$-genuine extension but not strongly $n$-genuine extension of $\Q(X_1,...,X_n)$. 
 \end{example}
    \begin{proof}
        Since $K(X_{j+1},...,X_n)\subset L$ is a subfield of $L$ generated by a polynomial in only $Y,X_1,...,X_j$,  we know that $L$ cannot be strongly $n$-genuine. It remains to show that $L$ is $n$-genuine. 

        Assume for contradiction that $L$ is not $n$-genuine, so there exists a polynomial $G(Y,\bX_I) \in \Z[Y,\bX_I]$ for some index set $I \subsetneq \{1,\ldots,n\}$ such that $$L = \Q(\bX)[Y]/G(Y,\bX_I).$$
        Assume without loss of generality that   $G$ does not depend on $X_1$ (that is, $I \subset \{2,\ldots,n\}$). (If instead, $G$ does not depend on some $X_i$ with $i \in \{j+1,\ldots,n\}$, the following argument would be applied with $K'$ in place of $K$.)  Then we can write $L = L'(X_1)$, where $$L' = \Q(X_2,...,X_n)[Y]/G(Y,X_2,...,X_n).$$
        Now, we note that $L\cap \overline{\Q(X_1,...,X_j)} = K = L'(X_1) \cap \overline{\Q(X_1,...,X_j)}.$ Thus, this implies that $K = (L'\cap \overline{\Q(X_2,...,X_j)})(X_1)$; in other words, $K$ is of the form $M(X_1)$ for $M$ an extension of $\Q(X_2,...,X_j)$. This implies that $K$ is not $j$-genuine and gives a contradiction. 
    \end{proof}

  The next example produces classes of polynomials to which Theorem \ref{thm_genuine_not_strongly} applies.
\begin{example}[$n$-genuine but not strongly $n$-genuine]\label{example_gen_not_strongly}
 Let $G(\bX_I)\in \Z[\bX_I]$ and $H(\bX_J)\in \Z[\bX_J]$, where $I$ and $J$ partition $\{1,\ldots,n\}$ nontrivially. Assume that $\deg_{X_i}G(\bX_I)\geq 1$ for all $i \in I$ and $\deg_{X_j}H(\bX_J)\geq 1$ for all $j \in J$. First, we claim that for $e \geq 2$, $Y^e - G(\bX_I)$ is an $|I|$-genuine polynomial. Indeed, $Y^e - G(\bX_I)$  generates the extension $\Q(\bX)(\sqrt[e]{G(\bX_I)})$ and there is no polynomial in fewer than $|I|$ variables for which $\sqrt[e]{G(\bX_I)}$ is a root. Similarly, we see that for $f\geq 2$, $Y^f - H(\bX_J)$ is a $|J|$-genuine polynomial. 

Assume that $e\geq 2$, $f\geq 2$, and that $\gcd(e,f)=1$. Consider the compositum field:
$$K=\Q(\bX)(\sqrt[e]{G(\bX_I)}, \sqrt[f]{H(\bX_J)}).$$
We saw in Example \ref{example_composita} that $K$ is an $n$-genuine but not strongly $n$-genuine extension. Thus, if $F(Y,\bX)$ is a polynomial satisfying that $$K = \Q(\bX)[Y]/F(Y,\bX),$$
then $F(Y,\bX)$ is an $n$-genuine but not strongly $n$-genuine polynomial.
Finally, we construct such a polynomial $F(Y,\bX)$. Consider the field
$$K' = \Q(\bX)(\alpha), \qquad \text{for $\alpha = \sqrt[e]{G(\bX_I)} (1+\sqrt[f]{H(\bX_J)}.$}$$
Since $\gcd(e,f)=1$, we can see that $\alpha^e \in K'$ implies that $\sqrt[f]{H(\bX_J)}\in K'$. So, we immediately deduce that $\sqrt[e]{G(\bX_I)}\in K'$. Thus, $K'\subset K$. However, $[K:\Q(\bX)] = ef = [K':\Q(\bX)].$ So, $K'=K$. So, we can take $F(Y,\bX)$ to be the minimal polynomial of $\alpha$. 
\end{example}

Recall from Lemma \ref{lemma_preservation_scaling} that being $n$-genuine is preserved by scaling.
But being $n$-genuine is not, in general, preserved by isomorphism of fields. For example, the field $\Q(X_1,X_2,\sqrt{X_1+X_2})$ is a 2-genuine extension, but $\Q(X_1,X_2,\sqrt{2X_1})$ is not. However the first field is  isomorphic to the second, under the map $X_1\mapsto X_1-X_2$, $X_2\mapsto X_1+X_2$.  This motivates defining the stronger notion of a $(1,n)$-allowable polynomial, which we turn to next.
 
\subsection{ Examples: $(1,n)$-allowable and strongly $(1,n)$-allowable polynomials}
First we provide some non-examples, to distinguish these two properties.

\begin{example}[Linear is not allowed]
$F(Y,\bX)=Y^D - L(X_1,\ldots,X_n)$ for $D \geq 2$ and $L$ a linear form is not $(1,n)$-allowable for $n \geq 2$. 
\begin{proof} A linear change of variables $Z_1=L(X_1,\ldots,X_n)$ shows that $F$ can be expressed as a function of only two variables $Y,Z_1$ and is hence not $n$-genuine.   Thus $F$ is not $(1,n$)-allowable.  Concordantly, as in \cite[Ch. 13]{Ser97},
$ M = \{ \bx = (x_1,\ldots,x_n) \in \Z^n : \text{$x_1$ is a square}\}$ is an affine thin set, but has $\gg B^{n-1/2}$ integral points of height at most $B$. Restricting attention to $(1,n)$-allowable polynomials rules out such cases from the settings of our main theorems. Similarly, note that the case in which $H(\bX)$ is linear is ruled out by the hypothesis of Theorem \ref{thm_cyclic_uniform}.
\end{proof}
\end{example}

  \begin{example}[$(1,n)$-allowable is stronger than having no change of variables]
  If $F(Y,\bX)$ is $(1,n)$-allowable then there does not exist a change of variables $\bX\rightarrow \bZ$ in $\mathrm{GL}_n(\Q)$ such that $F(Y,\bZ) \in \Z[Y,Z_2,...,Z_n]$. (In particular, if $F(Y,\bX)$ is weighted-homogeneous  so that the hypersurface $F(Y,X_1,\ldots,X_n)=0$ lies in weighted projective space $\mathbb{P}(e,1,\ldots,1)$, then this condition imposes that it is not contained in a hyperplane.)  Yet being $(1,n)$-allowable is stronger than having no such change of variables: for example, $F(Y,X_1,X_2,X_3) = (Y-X_1^2)^2 + X_2^2+X_3^2$ is not $(1,3)$-allowable. (To see this,  note that $F(Y,X_1,X_2,X_3)$ generates the extension $\Q(X_1,X_2,X_3,X_1^2 + \sqrt{X_2^2+X_3^2}) = \Q(X_1,X_2,X_3,\sqrt{X_2^2+X_3^2})$, which is not 3-genuine.) But it also has the property that there is no change of variables $\bX\rightarrow \bZ$ in $\mathrm{GL}_3(\Q)$ such that $F(Y,\bZ) \in \Z[Y,Z_2,Z_3]$.  
\end{example}

 We next construct a class of polynomials for which Theorem \ref{thm_allowable_not_strongly} applies.

\begin{example}[$(1,n)$-allowable but not strongly $(1,n)$-allowable]\label{ex: notstrong}
   Let $I,J$ be a nontrivial partition of $\{1,...,n\}$, and let $H(\bX_I)\in\mathbb{Z}[\bX_I]$, $G(\bX_J)\in\mathbb{Z}[\bX_J]$ be two irreducible homogeneous polynomials such that there is no $L_I\in\GL_{|I|}(\Q)$ such that $H(L_I(\bX_I))\in\mathbb{Z}[\bX_{I'}]$ with $I' \subsetneq I$, and  there is no $L_J\in\GL_{|J|}(\Q)$ such that $G(L_J(\bX_J))\in\mathbb{Z}[\bX_J']$ with $J' \subsetneq J$. Then   for each $d\geq 2$ and $e \geq 2$,
   \[
   F(Y,\bX)=(Y^{d}+H(\bX))^{e}+G(\bX)=(Y^{d}+H(\bX_I))^{e}+G(\bX_J)
   \]
   is $(1,n)$-allowable but not $(1,n)$-strongly allowable.

   \begin{proof}
    Throughout, we may write $H(\bX)=H(\bX_I)$ and $G(\bX)=G(\bX_J)$, even if each variable only depends nontrivially on certain coordinates. 
    We start by recording a simple observation: suppose a homogeneous polynomial $M(\bfX)\in\mathbb{Z}[\bfX]$ is given. Let a subset $B\subsetneq\{1,...,n\}$ be fixed. 
    Let $M^B(\bX_B)$ be the component of $M(\bX)$ containing monomials only in terms of $X_b$ for $b\in B$; that is, $M^B(\bX_B) = \left. M(\bX) \right|_{X_a=0,a \not\in B}$. 
   Suppose $\tau\in\GL_{n}(\Q)$ is a linear transformation such that its associated matrix does not have any zeros on the diagonal.
    (That is, if $\tau_i$ represents the $i$-th row of $\tau$ then $\tau_i\cdot \bX$ contains $X_i$; that is to say  $X_i$ is mapped to a linear combination that includes $X_i$.) Then   there exists 
    $\tau^{B}\in \GL_{|B|}(\Q)$ acting on $\bX_B$ such that
      \beq\label{matrix_fact}
      M_{\tau}(\bfX)=M^B_{\tau^{B}}(\bX_B) + M_{1}(\bfX),
      \eeq
      where  $\left. M_{1}(\bX) \right|_{X_a=0,a \not\in B}\equiv  0$. (That is, every monomial in $M_{1}$ contains a factor $X_{a}$ for some $a\in \{1,...,n\}\setminus B$. On the other hand, $M^B_{\tau^{B}}$ depends only on $X_{b}$ for $b\in B$.) Here we use the notation $M_\tau(\bX) = M(\tau \bX)$ and similarly for $M_{\tau^B}$.

       To verify the example, we first  show that $F(Y,\bX)$ is $(1,n)$-allowable.  By contradiction, assume we can find $L\in \GL_{n}(\mathbb{Q})$ such that $F_{L}(Y,\bX):=F(Y,L(\bX))$ is not $n$-genuine. By reordering the columns of the matrix representing $L$, we may assume that there are no zeroes on the diagonal. (That is, if $L_i$ represents the $i$-th row of $L$ then $L_i\cdot \bX$ contains $X_i$; that is to say  $X_i$ is mapped to a linear combination that includes $X_i$.) 
      By Lemma \ref{lem : spectoglobal}, there exists a nonempty subset $K\subsetneq\{1,...,n\}$ such that    
\beq\label{F_linear_factor}
\text{for every $\bx_K \in \Q^{|K|}$,  $F_L(Y,\bX_{K^c},\bx_K)$ has a linear factor in $Y$ over $\overline{\Q}$.}
\eeq
 We aim to contradict this. Let $G_L(\bX)=G(L(\bX))$. We claim that if $\ell\not\in K$, then $\deg_{X_{\ell}}G_{L}=0$. Otherwise, by the Hilbert Irreducibility Theorem, there exists some $\bx_K \in \Q^{|K|}$  such that $G_{L}(\bX_{K^c},\bx_K)$ is irreducible over $\mathbb{Q}$. Since we are assuming $\ell\not\in K$, then $\deg_{X_{\ell}}G_{L}(\bX_{K^c},\bx_K)\geq 1$, i.e. $G_{L}(\bX_{K^c},\bx_K)$ is not constant. In particular, this implies that $G_{L}(\bX_{K^c},\bx_K)$ is separable, so that it is not an $e$-power over $\overline{\Q}$. In turn, this implies that  
 \[F_{L}(Y,\bX_{K^c},\bx_K)=(Y^{d}+H_{L}(\bX_{K^c},\bx_K))^{e}+G_{L}(\bX_{K^c},\bx_K),
\]
is irreducible  over $\overline{\Q}$, hence contradicting (\ref{F_linear_factor}).  Thus $\deg_{X_{i}}G_{L}=0$ for all $i\not\in K$. Observe that this furthermore implies that $J\subset K$. On the other hand, by the fact (\ref{matrix_fact}), we can find $L^K\in\GL_{|K|}(\Q)$ such that
      \[
      G_{L}(\bfX)=G^K_{L^{K}}(\bX_K) + G_{1}(\bfX),
      \]
      where $\left. G_{1}(\bX)\right|_{X_i=0,i \not\in K}\equiv  0$ and $\deg_{X_{i}}G_{L}=0$ for every $i\not\in K$. Taking these two facts in combination, it follows that 
       \[
      G_{L}(\bfX)=G^K_{L^{K}}(\bX_K).
      \]
    Thus,
    \[
    F_{L}(Y,\bfX)=(Y^{d}+H_{L}(\bfX))^{e}+G^K_{L^{K}}(\bX_K).
\]
Next,  by Remark \ref{remark_G_sequence} we know that we can find a polynomial $P\in\Q(\bX_K)[Y]$ such that
\[
\Q(\bfX)[Y]/F_{L}=(\Q(\bX_K)[Y]/P)(\bX_{K^c}).
\]
This implies that for every $\sigma \in \GL_{|K|}(\Q)$
\[
\Q(\bfX)[Y]/F_{\sigma\circ L}=(\Q(\bX_K)[Y]/P_{\sigma})(\bX_{K^c}),
\]
so that $F_{\sigma \circ L}$ is also not $n$-genuine.
(While $\sig$ initially acts on $\bX_K$, we have extended the definition of $\sigma$ to act on $n$ variables by setting $\sigma (X_{i})=X_{i}$ for all $i\not\in K$; thus $\sigma \circ L$ is well-defined.) We claim that we can find some $\sig \in \GL_{|K|}(\Q)$ such that $G_{\sigma\circ L}=G$ and $\sigma \circ L$ has no zeroes on the diagonal.
To verify this, first recall that $G_L(\bX) = G^K_{L^K}(\bX_K)$, and by the hypothesis of the example, the polynomial depicted on the right-hand side depends nontrivially on  $X_k$ for each $k \in K$. Moreover $L^K$ is invertible when acting on $\bX_K$, so we can simply take $\sig = (L^K)^{-1}$ and then extend it to act on $n$ variables by nominally setting $\sig=\begin{pmatrix}
    (L^{K})^{-1}&0\\0&I
\end{pmatrix}$ (assuming temporarily for notational simplicity that $K=\{1,\ldots,|K|\}$); this verifies the first part of the claim. Second, we claim that $\sig \circ L$ does not have any zeroes on the diagonal. This is true since $\sig \circ L$ acts as the identity on $\bX_K$, and on the other variables, it acts by $L$, which has no zeroes on the diagonal.

Consequently, by replacing $L$ by $\sigma\circ L$, from now on we may assume that $L$ is such that
 \[
   F_{L}(Y,\bX)=(Y^{d}+H_{L}(\bX))^{e}+G(\bX),
   \]
   and $L$ has no zeroes on the diagonal. 
  Here we note that $G$ is \emph{not} modified by $L$. 
  Again by the fact (\ref{matrix_fact}) we can find $L_I \in \GL_{|I|}(\Q)$ acting on $\bX_I$ such that 
   \[
      H_{L}(\bfX)=H^I_{L^{I}}(\bX_I) + H_{1}(\bfX),
      \]
where $\left.H_{1}(\bX)\right|_{X_j=0,j \in J}\equiv  0$. Thus after restricting to $X_{j}=0$ for every $j\in J$,  the homogeneity of $G$ implies that 
\[
   \left. F_{L}(Y,\bX)\right|_{X_j=0, j \in J}=(Y^{d}+H^I_{L^{I}}(\bX_I))^{e}.
   \]
Now let $i\not\in K$; since $J \subset K$ this implies $i \in I$. By hypothesis we know that $\deg_{X_{i}}H^I_{L^{I}}\geq 1$. First suppose that $|I|=1$. Then under the hypothesis that $H(\bX_I)$ is irreducible over $\Q$, $H^I_{L^I}(\bX_I)$ is irreducible over $\Q$, hence separable, hence not a $d$-power over $\overline{\Q}$. Next, if $|I| \geq 2$, then by the Hilbert Irreducibility Theorem, we can find $\bx_{I \setminus\{i\}}=(x_{\ell})_{\ell\neq i} \in \Q^{|I|-1}$ such that $H^I_{L_{I}}(X_{i},x_{\ell})_{\ell\neq i}$ is irreducible over $\Q$, hence separable, and hence not a $d$-power over $\overline{\Q}$.  In either case, this shows the polynomial $F_{L}(Y,X_{i},X_{\ell}=x_{\ell},X_{j}=0)_{\ell\in I\setminus\{i\},j\in J}$ does not have a linear factor over $\overline{\Q}$, contradicting the conclusion of Lemma \ref{lem : spectoglobal} recorded in (\ref{F_linear_factor}). (Indeed, recalling the comment after Lemma \ref{lem : spectoglobal}, it contradicts the conclusion for specialized variables with indices in a superset of $K$.)

   Finally, we show that $F(Y,\bX)$ is  not strongly $(1,n)$-allowable.
Since $\mathbb{Q}(\bfX)[Y]/(Y^{e}+G(\bfX))$ is a subextension of $\mathbb{Q}(\bfX)[Y]/(F(Y,\bfX))$ and $G\in\ZZ[\bX_J]$ for a proper subset $J\subsetneq \{1,...,n\}$, it follows that $\mathbb{Q}(\bfX)[Y]/(F(Y,\bfX))$ is not strongly $n$-genuine, hence $F$ is not strongly $(1,n)$-allowable.
   \end{proof}
\end{example}

  We next verify the properties of six final classes of polynomials; these examples demonstrate that our general theorems recover nearly all previous literature on the topic. The first of these examples has been applied in the derivation of Corollary \ref{cor_genericity_M_Y}.
\begin{example}[Diagonal structure]\label{ex: generic diag}
    Fix integers $d\geq 2$ and   $k_j\geq 2$ for $j=1,...,n.$ Then the polynomial $$F(Y,\bX) = Y^d - \sum_{j=1}^n X_j^{k_j}$$
    is $(1,n)$-allowable.
        
    \begin{proof}    Consider the vanishing set $V(F) \subset \A_{\overline{\Q}}^{n+1}$; this variety has an isolated singularity at the vector $\mathbf{0}$ and hence $\dim_{\overline{\Q}}(\sing(V(F)))=0.$ Assume for contradiction that there exists   $L\in \GL_n(\Q)$ such that $F(Y,L(\bX))$ is not $n$-genuine, so $F(Y,L(\bX))$ generates the same extension as a polynomial $G(Y,\bX_I)$ in which $|I| < n$. Without loss of generality, we may suppose $G(Y,\bX) \in \Z[Y,X_2,\ldots,X_n]$. Then $\{y=0\}\cap\{x_i=0, i=2,...,n\}$ is contained in the singularity of $V(F(Y,L(\bX)))$, which implies that $\dim_{\overline{\Q}}(\sing(V(F))) = \dim_{\overline{\Q}}(\sing(V(F(Y,L(\bX)))))>0$, which is a contradiction.  
    \end{proof}
\end{example}

The next example confirms that  Theorem \ref{thm_explicit} immediately implies Theorem \ref{thm_cyclic_uniform}. It has also been applied (twice) in the proof of Theorem \ref{thm_genericity} to verify that the polynomials $\tilde{h}_1,\ldots,\tilde{h}_s$ are not all identically zero on the relevant moduli space.
\begin{example}[Cyclic structure]\label{ex: generic cyclic}
 Let $H(\bX) \in \Z[X_1,\ldots,X_n]$ be an irreducible polynomial of total degree at least 2 and define for a fixed $d \geq 2$,  $$F(Y,\bX) = Y^d - H(\bX).$$
Suppose there is no $L\in \GL_n(\Q)$ such that $H(L(\bX)) \in \Z[\bX_I]$ for some set $I \subsetneq \{1,\ldots,n\}$.  Then $F(Y,\bX)$ is strongly $(1,n)$-allowable.  
 \begin{proof}
 We first remark that $F$ is absolutely irreducible, since $Y^d-H(\bX)$ is reducible over $\overline{\Q}$ if and only if $H(\bX)$ is an $\ell$-power over $\overline{\Q}$ for some $\ell|d$ with $\ell>1$. Such a factorization $H(\bX)=(G(\bX))^\ell$ for some $G \in \overline{\Q}[\bX]$ cannot occur under the hypothesis, since then  for every $\bx' \in \Q^{n-1}$, $H(X_1,\bx') = (G(X_1,\bx'))^\ell$. But by the Hilbert Irreducibility Theorem, there exists some $\bx' \in \Q^{n-1}$ such that $H(X_1,\bx')$ is irreducible over $\Q$, hence separable, hence certainly not an $\ell$-power over $\overline{\Q}$. 
 
    Assume for contradiction that $F(Y,\bX)$ is not strongly $(1,n)$-allowable, so there exists   $L\in \GL_n(\Q)$ such that $F_L(Y,\bX):=F(Y,L(\bX))$ is not strongly $n$-genuine. Then there exists a polynomial $G(Y,\bX_I)$ for some set $I \subsetneq \{1,\ldots,n\}$,  such that $$K:=\Q(\bX)[Y]/(F(Y,L(\bX))) \supset \Q(\bX)[Y]/(G(Y,\bX_I)) \supsetneq \Q(\bX).$$
Without loss of generality, we may suppose $I=\{2,\ldots,n\}$.  
Let $H_L(\bX)$ denote $H(L(\bX))$. By hypothesis,  $\deg_{X_{1}}H_L(\bX)>0$, hence by the Hilbert Irreducibility Theorem we can find $\bx'=(x_{2},...,x_{n}) \in \Z^{n-1}$ such that $H_{L,\bx'}(X_{1}) := H_L(X_1,\bx')$ is irreducible over $\Q$. On the other hand, $Y^{d}-H_{L,\bx'}(X_{1})$ is reducible over $\overline{\Q}$ if and only if $H_{L,\bx'}(X_{1})$ is an $\ell$-power for some $\ell>1$ with $\ell | d$ over $\overline{\Q}$ but this is not possible since $H_{L,\bx'}(X_{1})$ is irreducible over $\Q$ (hence $H_{L,\bx'}(X_{1})$ is separable, and certainly not an $\ell$-power).  Thus under the assumption that such a change of variables $L$ exists,  we have produced an index set $I=I(L)$ (depending on $L$) with $|I|<n$,   for which there is some $\bx_I \in \Q^{|I|}$ for which $F_L(Y,\bX_{I^c},\bx_I)$ is not reducible over $\overline{\Q}$.   
By Lemma \ref{lem : spectoglobal} this yields a contradiction. For by that lemma, since $F_L(Y,\bX)$ is not strongly $n$-genuine then  there exists a nonempty subset $J=J(L)\subsetneq \{1,\ldots,n\}$  such that $F_L(Y,\bX_{J^c},\bx_J)$ is  reducible over $\overline{\Q}$ for all choices of $\bx_J \in \Q^{|J|}$. Indeed, by Remark \ref{remark_G_sequence}, the conclusion of Lemma \ref{lem : spectoglobal} holds for $J=I$, and this yields a contradiction.
 
\end{proof}
 
\end{example}

The next example has been applied in the derivation of Corollary \ref{cor_genericity_sing}.
 \begin{example}[Cyclic structure, singular locus dimension $s$]\label{ex: genericsin}
Fix $d \geq 2, m \geq 1$. Let $H(\bX) \in \Z[X_1,\ldots,X_n]$ be a homogeneous polynomial of degree $d\cdot m$ such that the projective hypersurface $V(H)\subset \mathbb{P}^{n-1}$ defined by the vanishing of $H$ has $\dim (\sing V(H))=s$, where this denotes the projective dimension in $\mathbb{P}^{n-1}$, $-1 \leq s \leq n-1$. Moreover, suppose there is no $L\in \GL_n(\Q)$ such that $H(L(\bX)) \in \Z[\bX_I]$ with $|I|<n$. 
Then $F(Y,\bX)=Y^d -H(\bX)$ is strongly $(1,n)$-allowable.

     \begin{proof}
       This follows immediately from Example \ref{ex: generic cyclic}, upon choosing $H$ with the desired dimension of the singular locus.
     \end{proof}
 \end{example}

\begin{example}[Projective smooth hypersurfaces]\label{ex: generic nonsingular}
     Let $G(Y,X_{1},...,X_{n})\in \Z[Y,X_1,\ldots,X_n]$ be a homogeneous nonsingular polynomial that is monic in $Y$ and of degree at least 2.   Then $G$ is strongly $(1,n)$-allowable. 
    
   \begin{proof} By contradiction, assume we can find $L\in \GL_{n}(\mathbb{Q})$ such that $G_{L}(Y,\bX) :=G(Y,L(\bX))$ is not strongly $n$-genuine. By Lemma $\ref{lem : spectoglobal}$ this would imply that there exists $I\subsetneq\{1,...,n\}$ such that for every $\bfx_{I}\in \Q^{|I|}$, the polynomial $G_L(Y,\bX_{I^c},\bx_I)$ is reducible over $\overline{\Q}$. Without loss of generality we may assume $I=\{2,...,n\}$. Consider the morphism
   \[
  \pi: \begin{matrix}
       V(G_{L}) &\rightarrow &\mathbb{P}^{n-2}, \\  [y:x_{1}:...:x_{n}]&\mapsto& [x_{2}:...:x_{n}].
   \end{matrix}
   \]
The linear transformation $L$ preserves the nonsingularity of the polynomial, and  the geometric fiber $\pi^{-1}(\eta)$   is smooth (here $\eta$ is the generic point in $\mathbb{P}^{n-2}$). (Indeed, smoothness is preserved by base change \cite[EGA IV, Proposition (17.3.3) (iii), p. 61]{EGAIV_part4}.) It follows that we can find $[x_{2}:...:x_{n}]$ such that $G_{L}(Y,X_{1},x_{2}T,...,x_{n}T)$ is a smooth curve. But this implies that  $G_{L}(Y,X_{1},x_{2},...,x_{n})$   is irreducible over $\overline{\Q}$ for this choice of $[x_2: \cdots :x_n]$, which yields a contradiction. Thus it must be that $G_L(Y,\bX)$ is strongly $n$-genuine for all $L$, and hence $G$ is strongly $(1,n)$-allowable.
    \end{proof}
\end{example}

The next example confirms that Theorem \ref{thm_explicit} recovers the main result of \cite{BonPie24}.
\begin{example}[Nonsingular weighted projective hypersurface]\label{ex: generic nonsingular weighted}
Fix integers $d,e,\geq 1$ and $m \geq 2$. Let $F(Y,X_{1},...,X_{n})$ be a polynomial   of the form (\ref{F_dfn_intro})   in which for each $j=1,\ldots,d$ the polynomial $f_j$ is  homogeneous of degree $m \cdot e \cdot j$. Thus  $F(Y,\bX)=0$ defines a hypersurface in weighted projective space $\mathbb{P}(e,1,\ldots,1)$, and we assume in particular that  $V(F(Y,\bX))$ is nonsingular. Then $F$ is strongly $(1,n)$-allowable. 

\begin{proof}  To verify this, suppose for contradiction that $F(Y,X_{1},...,X_{n})$ is not strongly $(1,n)$-allowable; then $F(Z^{e},X_{1},...,X_{n})$ is not strongly $(1,n)$-allowable. But by hypothesis $V(F(Y,\bX))$ is nonsingular, and as shown in \cite[\S 3.3]{BonPie24} since $m\geq 2$ it follows that  $V(F(Z^{e},X_{1},...,X_{n}))\subset \mathbb{P}^{n}$ is also nonsingular. Thus $V(F(Z^{e},X_{1},...,X_{n}))$ is a smooth variety and by Example \ref{ex: generic nonsingular}, $F(Z^{e},X_{1},...,X_{n})$ is strongly $(1,n)$-allowable, showing the desired contradiction.
\end{proof}
 \end{example}
 \begin{rem}\label{remark_correction}
  The main result of \cite{BonPie24} states that $N(F,B) \ll_{n,\deg F,\ep} B^{n-1+\frac{1}{n+1}+\ep}$ with an implicit constant independent of $F$. In fact, the argument presented in that paper only shows polylog dependence on $\|F\|$, proving $N(F,B) \ll_{n,\deg F,\ep} (\log (\|F\|+2))^e B^{n-1+\frac{1}{n+1}+\ep}$ for some $e=e(n)$;   Example \ref{ex: generic nonsingular weighted} shows this is recovered by the present work.
\end{rem}

 The next example demonstrates that our main theorems recover many cases of   \cite[Thm. 1.1]{Bon21}.
\begin{example}[Separation of variables]\label{ex: sepvars}
Fix $d \geq 2, m \geq 2$. Let $P(Y) \in \Z[Y]$ be a monic polynomial of degree $d$ and let  $G(X_1,\ldots,X_n) \in \Z[X_1,\ldots,X_n]$ be an irreducible homogeneous polynomial of degree $m$ such that there is no $L \in \GL_{n}(\Q)$ such
that $G(L(X_{1},...,X_{n}))\in\Q[\bX_I]$ with $|I|< n$. Then $F(Y,\bX)=P(Y) -G(\bX)$ is $(1,n)$-allowable. 

\begin{rem}\label{remark_sepvars}
The result \cite[Thm. 1.1]{Bon21} proves $N(F,B)\ll_{F,\ep}B^{n-1+1/(n+1)+\ep}$ for all $F(Y,\bX)=P(Y) + G(\bX)$ in which $P$ is a polynomial of degree at least 2 and $G$ is a nonsingular homogeneous polynomial.  Theorem \ref{thm_allowable_not_strongly} recovers this for all such $F$ that are $(1,n)$-allowable but not strongly $(1,n)$-allowable; Theorem \ref{thm_explicit} recovers this unconditionally for all such $F$ that are strongly $(1,n)$-allowable and $P(Y)$ is a polynomial in $Y^m$ for some $m \geq 2$; otherwise Theorem \ref{thm_explicit} only recovers this with the additional hypothesis of GRH.
\end{rem}
    \begin{proof}
        We begin with a preliminary observation. Under the hypothesis that $G$ is irreducible, for every $a\in\mathbb{C}$ the polynomial $G(\bfX)-aT^{m}$ in $n+1$ variables is absolutely irreducible since $G$ is not an $m$-power. Thus, the affine variety $\{G-a\}$, obtained by intersecting the projective variety associated to $G(\bfX)-aT^{m}$ with the open set $\{T\neq 0\}$ is absolutely irreducible. Hence, for every $a\in\mathbb{C}^{\times}$, $G(\bfX)-a$ is an absolutely irreducible polynomial. In particular, this implies that for every $y$ with $P(y)\neq 0$, $G(\bfX)-P(y)$ is absolutely irreducible, hence $F(Y,\bX)=P(Y)-G(\bfX)$ is absolutely irreducible.
        
        Now we verify the example. To seek a contradiction, suppose that we can find $L\in \GL_{n}(\mathbb{Q})$ such that $F_{L}(Y,\bX):=F(Y,L(\bX))$ is not $n$-genuine. By Lemma $\ref{lem : spectoglobal}$ we know that there exists a nonempty subset $I\subsetneq\{1,...,n\}$ such that \beq\label{sep_vars_allowable}
        \text{for every $\bx_I \in \Q^{|I|}$, the polynomial $F_L(Y,\bX_{I^c},\bx_I)$ has a linear factor in $Y$ over $\overline{\Q}$.}
        \eeq
        Without loss of generality we may assume $I=\{2,...,n\}$. For a given  root $\alpha$ of the derivative $P'(Y)$,    let $S_{\alpha}=\left\{G(\bfX)-P(\alpha)=\frac{\partial G}{\partial X_{1}}=0\right\}$.   By the preliminary observation, $G(\bX) - P(\al)$ is irreducible so in particular it shares no common factor with $\frac{\partial}{\partial X_1}(G(\bX)-P(\al)) = \frac{\partial G}{\partial X_1}$; consequently $\dim S_{\alpha}=n-2$.  In what follows, given $(x_2,\ldots,x_n) \in \Q^{n-1}$, we denote the line 
        \[H_{x_{2},...,x_{n}}:=\{X_{2}-x_{2}=...=X_{n}-x_{n}=0\}.\]
        We claim that we can find a nonempty open set $U_{\alpha}\subset\mathbb{Q}^{n-1}$ such that $S_{\alpha}\cap H_{x_{2},...,x_{n}}=\emptyset$ over $\mathbb{C}$ for every $(x_{2},...,x_{n})\in U_{\alpha}$.
        We first observe that $H_{x_2...,x_{n}}\cap S_{\alpha}\neq\emptyset$    if and only if the resultant vanishes, namely  \[
      \text{Res}\left(G(\bfX)-P(\alpha ),\frac{\partial G}{\partial X_{1}},X_{2}-x_{2},...,X_{n}-x_{n}\right)=0,
      \] 
      which gives a polynomial relation in the coefficients $x_{2},...,x_{n}$. 
     Next, we prove that we can find some choice of $(x_{2},...,x_{n}) \in \Q^{n-1}$ such that  $S_{\alpha}\cap H_{x_{2},...,x_{n}}=\emptyset$, so that this resultant polynomial is not identically zero; this suffices to show such an open set $U_\al$ exists.  Since for all $x_2 \in \Q$ the hyperplanes $\{X_{2}-x_{2}=0\}$ are all parallel, for every irreducible component of $S_{\alpha}$, say $V$, there exists at most one value of $x_{2}$ such that $V\subset\{X_{2}-x_{2}\}$. Thus, $V\cap\{X_{2}-x_{2}\}$ has dimension $n-3$ for all but one value of $x_{2}$. Since $S_{\alpha}$ has finitely many irreducible components, it follows that $S_{\alpha}\cap \{X_{2}-x_{2}\}$ has dimension $n-3$ for all but finitely many choices of $x_{2} \in \Q$.  Pick now 
 $x_{2}$ such that $S_{\alpha}\cap\{X_{2}-x_{2}\}$ has dimension $n-3$. We can iterate the same argument and find $x_{2},...,x_{n}$ such that $\dim(S_{\alpha}\cap H_{x_{2},...,x_{n}})=n-2-(n-1)=-1$, i.e. $S_{\alpha}\cap H_{x_{2},...,x_{n}}$ is empty. 
 
 Now take $U=\cap_{\alpha}U_{\alpha}$ to be the intersection over the roots $\al$ of $P'(Y)$. By the preliminary observation, $F(Y,\bX)$ is absolutely irreducible, hence $F_L(Y,\bX)$ is absolutely irreducible, so by the Hilbert Irreducibility Theorem there exists a choice of $(x_{2},...,x_{n})\in U$ such that $F_{L}(Y,X_{1},x_{2},...,x_{n})$ is irreducible over $\Q$. Observe that for such a choice of $(x_{2},...,x_{n})$ the curve defined by $F_{L}(Y,X_{1},x_{2},...,x_{n})=0$ is smooth on the affine plane. Indeed, the Jacobian for $F(Y,X_{1},x_{2},...,x_{n})=P(Y)-G(X_{1},x_{2},...,x_{n})$ is 
      \[
      (F(Y,X_{1},x_{2},...,x_{n}),P'(Y), \frac{\partial G}{\partial X_{1}}(X_{1},x_{2},...,x_{n})).
      \]
      If there is a singular point, so that the Jacobian vanishes, then there exists a root $\alpha$ of $P'$ such that the system
      \[
      G(\bfX)-P(\alpha )=\frac{\partial G}{\partial X_{1}}=X_{2}-x_{2}=...=X_{n}-x_{n}=0
      \]
          has a  solution, which is impossible, since $(x_2,\ldots,x_n)$ lies in $U_\al$ for all $\al$. 
      
      We will now use (\ref{sep_vars_allowable}) to construct a singular point of $F_L(Y,X_1,x_2,\ldots,x_n)$ in $\mathbb{A}^2$, thus obtaining a contradiction. Since we are assuming that $F_{L}(Y,X_{1},x_{2},...,x_{n})$ has a linear factor in $Y$ over $\overline{\mathbb{Q}}$ and it is irreducible over $\mathbb{Q}$, it follows from Lemma \ref{lemma_F_linear_factor_splits_completely} that 
         \beq\label{FL_factorization}
      F_{L}(Y,X_{1},x_{2},...,x_{n})=\prod_{\tau\in H}(Y-\tau(Q(X_{1}))),
      \eeq
      where $H$ is the Galois group of the extension $\mathbb{Q}(X_{1})[Y]/F_{L}(Y,X_{1},x_{2},...,x_{n})$, meaning, the group of embeddings of the extension in its Galois closure).  We claim that $Q(X_{1})$ cannot be of the form
      \[
      Q(X_{1})=X_{1}R(X_{1})+\beta,
      \]
      for some $R\in\mathbb{Q}[X_{1}]$; that is, the polynomial $Q \in \overline{\Q}[X_1]$ cannot have only its constant term lie in $\overline{\Q}\setminus \Q$.   For if $Q$ had this form, then 
  \[
      \Tr(Q(X_{1}))=\Tr(X_{1}R(X_{1}))+\Tr(\beta)=dX_{1}R(X_{1})+\Tr(\beta)\neq 0,
     \]
      which would be a contradiction. (By construction $F(Y,\bX)$ has separation of variables, so $F_L(Y,\bX)$ does as well, which would be violated via the factorization (\ref{FL_factorization}) if $\Tr(Q(X_1)) \neq 0$.) We conclude that there is some nonconstant term in $Q$ with a coefficient in $\overline{\Q}\setminus \Q$, so that we can find $\tau\neq \text{Id}$, such that $\deg_{X_{1}}(Q(X_{1})-\tau(Q(X_{1})))> 0$. This implies that $F_{L}(Y,X_{1},x_{2},...,x_{n})$ has singular points on the affine plane since
      \[
      \{Y-Q(X_{1})=Y-\tau(Q(X_{1}))=0\}
      \]
      is not empty in $\mathbb{A}^{2}$. 
      This contradicts the earlier conclusion that the curve defined by the equation $F_{L}(Y,X_{1},x_{2},...,x_{n})=0$ is smooth on the affine plane. 
    \end{proof}
\end{example}

\subsection*{Acknowledgements}

The authors thank Tim Browning for his encouragement, Julian Demeio, and Per Salberger for helpful references, and Emmanuel Kowalski and Will Sawin for discussions on uniformity in the stratification of Katz and Laumon. 
L.P.  has been partially supported during portions of this research by NSF DMS-2200470, a Joan and Joseph Birman Fellowship, the Simons Foundation (1026705), and a Guggenheim Fellowship. L. P. thanks the Hausdorff Center for Mathematics for hosting several research periods as a Bonn Research Chair and the Mittag-Leffler Institute for hosting a research period in March 2024. 
K.W. visited Duke in 2023 with funding from NSF RTG-2231514, which supports the Number Theory group at Duke University. K.W. is partially supported by NSF GRFP under Grant No.  DGE-2039656.

\bibliographystyle{alpha}

\bibliography{NoThBib}

\end{document}